 \documentclass[abstract=true, paper=a4, fontsize=11pt,DIV=12,bibliography=totoc,final]{scrartcl}

\usepackage[utf8]{inputenc}
\usepackage[T1]{fontenc}
\usepackage{lmodern}
\usepackage[english]{babel}
\usepackage{csquotes}
\usepackage{amsthm, amssymb, amsmath, amsfonts, mathrsfs, dsfont, esint,textcomp}
\usepackage[colorlinks=true, pdfstartview=FitV, linkcolor=blue, citecolor=blue, urlcolor=blue,pagebackref=false]{hyperref}
\usepackage{mathtools}
\usepackage[shortlabels]{enumitem}

\usepackage[notcite,notref,color]{showkeys}
\definecolor{labelkey}{gray}{.8}
\definecolor{refkey}{gray}{.8}

 \usepackage{authblk}

 \setkomafont{date}{\large}

 \addtokomafont{disposition}{\rmfamily}

\numberwithin{equation}{section}

\usepackage{microtype}

\newtheorem{theorem}{Theorem}[section]
\newtheorem{lemma}[theorem]{Lemma}
\newtheorem{proposition}[theorem]{Proposition}
\newtheorem{definition}[theorem]{Definition}

\newtheorem{corollary}[theorem]{Corollary}

\theoremstyle{remark}

\newtheorem{remark}[theorem]{Remark}

\def\XXint#1#2#3{{\setbox0=\hbox{$#1{#2#3}{\int}$ }
\vcenter{\hbox{$#2#3$ }}\kern-.6\wd0}}

\mathtoolsset{showonlyrefs}

\newcommand{\N}{{\mathbb N}}
\newcommand{\Z}{{\mathbb Z}}

\newcommand{\R}{{\mathbb R}}

\renewcommand{\S}{\mathbb{S}}

\newcommand{\mA}{\mathcal{A}}
\newcommand{\mB}{\mathcal{B}}
\newcommand{\B}{\mathcal{B}}
\newcommand{\mC}{\mathcal{C}}

\newcommand{\mO}{\mathcal{O}}

\newcommand{\mV}{\mathcal{V}}

\newcommand{\mP}{\mathcal{P}}

\newcommand{\mH}{\mathcal{H}}

\DeclareMathOperator{\dv}{div}
\newcommand{\pa}{{\partial}}
\newcommand{\na}{{\nabla}}

\newcommand{\eps}{{\varepsilon}}

\def\curl{\hbox{curl \!}}
\def\div{\hbox{div \!}}

\def\Id{{\mathrm{Id}}}

\newcommand{\dmin}{d_{\mathrm{min}}}
 \renewcommand{\skew}{\mathrm{skew}}
\DeclareMathOperator{\sym}{sym}
\DeclareMathOperator{\dist}{dist}
\DeclareMathOperator{\supp}{supp}
\DeclareMathOperator{\esssup}{esssup}
\newcommand{\wto}{\rightharpoonup}

\newcommand{\dd}{\, \mathrm{d}}
\renewcommand{\d}{\mathrm{d}}
\newcommand{\1}{\mathbf{1}}
\newcommand{\E}{{\mathbb E}}

\definecolor{darkblue}{rgb}{0,0,0.7} 
\definecolor{darkred}{rgb}{0.9,0.1,0.1}
\definecolor{darkgreen}{rgb}{0,0.5,0}

\newcommand{\red}[1]{{\color{darkred}{#1}}}

\DeclareUnicodeCharacter{0306}{{\u i}}

\usepackage[backend=biber,giveninits,maxnames=4,style=alphabetic,isbn=false, date=year, doi=false, eprint= true, url= false,sorting=ynt, sortcites = true]{biblatex}

\bibliography{bib.bib}

\title{Correction to  Doi Type  Models  for Suspensions} 
\author{David G\'erard-Varet\thanks{david.gerard-varet@imj-prg.fr, Institut de Mathématiques de Jussieu -- Paris Rive Gauche, Université de Paris, 8 Place Aurélie Nemours, 75013 Paris, France}, Richard M. H\"ofer\thanks{richard.hoefer@ur.de, Faculty of Mathematics, University Regensburg, Universit\"atsstraße 31, 93053 Regensburg, Germany.}}

\begin{document}
\maketitle

\begin{abstract}
    Starting from  microscopic $N$ particle systems, we study the derivation of Doi type models for suspensions of non-spherical particles in Stokes flows. While Doi models  accurately describe the effective  evolution of the  spatial particle  density to the first order in the particle volume fraction, this accuracy fails regarding the evolution of the particle orientations. We rigorously attribute this failure to the singular interaction of the particles via a $-3$-homogeneous kernel. In the situation that the particles are initially distributed according to a stationary ergodic point process, we identify the limit of this singular interaction term. It consists of two parts. The first corresponds to a classical term in Doi type models. The second new term depends on the (microscopic) $2$-point correlation of the point process. By including this term, we  provide a modification of the Doi model that is accurate to first order in the particle volume fraction.
\end{abstract}

{\small 

\smallskip

}

\tableofcontents

\section{Introduction}  
Popular multiscale models for inertialess axisymmetric particles in viscous flows are the so called Doi type models \cite[Section 8]{DoiEdwards88}. They are coupled kinetic/fluid partial differential equations, used for the description of various phenomena such as viscoelasticity, mixing in active suspensions, or phase transition in liquid crystals. We consider in the following the  simplest case of a dilute suspension of  neutrally buoyant, non-Brownian and passive particles  inside a Stokes flow, but the model can be adapted to other situations.
For a small particle volume fraction $\phi$, this model takes the form
\begin{align} \label{limit.nonlinear}
\begin{aligned}
\partial_t  \bar f_\phi +  \bar u_\phi \cdot \nabla_x \bar f_\phi+ \dv_\xi\Big(\big(M(\xi)\nabla  \bar u_\phi\big)\xi \bar f_\phi\Big) = 0\\
    \bar f_\phi(0,\cdot) = f^0, \\
    - \dv \Big(\big(2 + \phi  \mV[\bar f_\phi]\big)  D \bar u_\phi\Big) + \nabla \bar p_\phi = g, \qquad \dv  \bar u_\phi = 0.
    \end{aligned}
   \end{align}
Here, $\bar f_\phi(t,x,\xi)$ is the density of particles at time $t \geq 0$, space position $x \in \R^3$ and orientation $\xi \in \S^2$ and  $\bar u_\phi(t,x),\bar p_\phi(t,x)$ are the fluid velocity and pressure. Moreover, $g$ is a given external force field and $f^0$ a given initial datum for the particle density.
Furthermore, $D \bar u_\phi$ denotes the symmetric gradient of $\bar u_\phi$, and 
\begin{align} \label{sigma}
\mV[h] = \int_{\S^2} S(\xi)  h \dd \xi,  
\end{align}
where $ S(\xi) \in \mathcal L(\sym_0(3),\sym_0(3))$ is the linear map that determines the moment of stress 
 of  an isolated particle in a Stokes flow with prescribed strain and rotation at infinity. Finally, $M(\xi)$ is the rotational particle mobility under a given strain. More precisely, $M(\xi) \in \mathcal L(\R^{3\times3}_0,\skew(3))$ is the linear map that determines the angular velocity (identified with a skew symmetric matrix) of an isolated particle in a Stokes flow with prescribed strain and rotation at infinity. We refer to Subsection \ref{sec:single} for the precise definition of $S,M$.

    The system \eqref{limit.nonlinear} couples the convection of the particles to a Stokes equation with an effective viscosity, increased by $\phi \mV[\bar f_\phi]$ compared to the solvent viscosity. We will study its derivation -- and its failure and suitable corrections --  from a microscopic monodispersed system of  $N$ particles inside a Stokes flow.
    
\medskip  
We model the particles as follows. We introduce  $\mathcal B \subset B_1(0)$ a reference particle, which is a smooth compact and connected set with rotational symmetry, i.e.
\begin{align}\label{rod_reference}
	R \mB = \mB \quad \text{for  all } R \in SO(3) \text{ with } R e_3 =e_3, \text{where $e_3 = (0,0,1) \in \R^3$.}
\end{align}
Let $N \in \N$ and $X_i \in \R^3$, $\xi_i \in \S^2$, $1 \leq i \leq N$, and $r > 0$. The $i$-th particle is then given by  
$$\mB_i = X_i +  r R_{\xi_i} \mB$$
where the rotation matrix $R_{\xi_i} \in SO(3)$ is chosen such that $R_{\xi_i} e_3 = \xi_i$.

For a given external force field $g$, we then consider the fluid equations
\begin{equation} \label{main}
\begin{aligned}
 -\Delta u_N + \na p_N  = g, \quad \div u_N = 0, & \qquad \text{in } \R^3\setminus \cup_i \mB_i \\ 
u_N(x) \to 0  &\qquad \text{as } |x| \to \infty\\
D u_N\vert_{\mB_i}  = 0 &\qquad \forall i, \\
-\int_{\pa \mB_i} \sigma(u_N)n \dd s(x) = \int_{\mB_i} g \dd x & \qquad \forall i\\
- \int_{\pa \mB_i} \sigma(u_N)n  \times (x-X_i) \dd s(x) = \int_{\pa \mB_i} g \times (x - X_i) \dd x & \qquad \forall i.
 \end{aligned}
\end{equation}
Here $\sigma(u_N) = 2 Du_N  - p_N\Id$ denotes the fluid stress and $n$ is the outer normal of $\mB_i$, i.e. the inner normal with respect to the fluid domain. The first two equations in \eqref{main} model a Stokes fluid driven by $g$ outside the particles. The condition $D u_N = 0$ in $\mB_i$ expresses that $u_N$ is a rigid body motion inside the particles, i.e. $u_N(x) = u_N(X_i) + \frac 1 2 \curl u_N(X_i) \times (x-X_i) $ for all $x \in \mB_i$. The unknown translation and angular velocities  $u_N(X_i)$ and $\frac 1 2 \curl u_N(X_i)$ are determined implicitly  by the last two equations in \eqref{main}, corresponding respectively to the balance of forces and the balance of torques.

It is classical, that for $g \in L^{6/5}(\R^3)$ this problem admits a unique weak solution $(u_N, p_N) \in \dot H^1(\R^3) \times L^2(\R^3)/\R$, provided the particles satisfy the
non-overlapping condition
\begin{align} \label{non.overlapping}
    \mB_i \cap \mB_j \neq \emptyset
\end{align}
This is in particular satisfied for 
\begin{align} \label{non.overlapping.2}
    \dmin := \min_{i \neq j} |X_i - X_j| > 2r.
\end{align}
We now specify the particle evolution: for given initial particle positions and orientations $(X_i^0,\xi_i^0)$, it is given through
\begin{align}
    \frac {\d} {\d t} X_i(t) = u_N(t,X_i(t)), &\qquad X_i(0) = X_i^0,  \label{x.dot}\\
    \frac {\d} {\d t} \xi_i(t) = \frac 1 2  \curl u_N(t,X_i(t)) \times \xi_i, &\qquad \xi_i(0) = \xi_i^0, \label{xi.dot}
\end{align}
where $u_N(t,\cdot)$ is the solution to \eqref{main} with particle positions and orientations $X_i(t),\xi_i(t)$. We also define 
\begin{align}
    \phi := N r^3
\end{align}
which is proportional to the solid volume fraction.  We emphasize that all the quantities $X_i^0,\xi_i^0$, $\mB_i$, $r$, $\phi$ implicitly depend on $N$.
As in most previous works on close topics, see Subsection \ref{subsec_previous}, we  will assume strong separation of the particles and diluteness of the suspension:   
\begin{align}
     \exists c > 0 ~ \forall N \in \N ~ \dmin(0) \geq c N^{-1/3}, 
    \label{ass:separation} \tag{H1} \\
    \phi = o\big(1/(\log N)^2\big) 
     \label{ass:diluteness} \tag{H2}, 
\end{align}
We remark that assumptions \eqref{ass:separation}--\eqref{ass:diluteness} imply that
\begin{align}
    \frac{r}{\dmin(0)} \to 0.
\end{align} 
In particular, the non-overlapping condition \eqref{non.overlapping.2} is initially satisfied for sufficiently large $N$.
We further assume that initially  the particles uniformly lie in a bounded set $\mathcal O \subset \R^3$,  i.e.
\begin{align}
   \forall N \in \N, ~ \forall 1 \leq i \leq N, ~ X^0_i \in \mathcal O \label{ass:bounded.cloud} \tag{H3}.
\end{align}

By standard arguments, the dynamics is well-defined at least until the first collision of particles. 

\medskip
In the present work, we are interested in the behavior of the particle system as $N \to \infty$ and $r \to 0$, more precisely in the asymptotics of its empirical measure $f_N(t,\cdot) \in \mP(\R^3 \times \S^2)$ defined as
\begin{equation} \label{empirical_measure}
    f_N(t,\cdot) = \frac 1 N \sum_{i=1}^N \delta_{X_i(t),\xi_i(t)}, 
\end{equation}
with {$f_N(0, \cdot) = f_N^0 = \frac 1 N  \sum_{i=1}^N \delta_{X_i^0,\xi_i^0.}$
Our goal is to provide an approximation $f_\phi$ of $f_N$ which is accurate at first order in the particle volume fraction, that is with an error $o(\phi)$.  Surprisingly, although it contains $O(\phi)$ terms, the Doi system \eqref{limit.nonlinear} is not appropriate, and our main purpose is to provide a modified, more accurate version of the model.

\subsection{Previous results} \label{subsec_previous}

    The derivation of continuous models for suspensions, of the kind of \eqref{limit.nonlinear}, has been the matter of intense research over the last years. The derivation of a Stokes equation with effective viscosity, like  (\ref{limit.nonlinear}b),  is mathematically  well understood when the particle configuration of the microscopic suspension model is given, see \cite{NiethammerSchubert19, HillairetWu19, Gerard-VaretHoefer21}. In particular, the results in \cite{HillairetWu19} imply the formula \eqref{sigma} for the increased viscosity to leading order in $\phi$.
    Higher order expansions of the effective viscosity have been studied in \cite{Gerard-VaretHillairet19, Gerard-Varet21, Gerard-VaretMecherbet20, DuerinckxGloria20}. We point out that the study of the $O(\phi^2)$ correction to the effective viscosity depends on analyzing limits of binary interactions  with kernel $\nabla^2 G$, where $G$ is the fundamental solution of the Stokes equation (Oseen tensor). We shall face a similar difficulty here.

    %In particular, our proof in the stationary ergodic setting is inspired by the analysis in \cite{Gerard-%VaretHillairet19}. In contrast to \cite{Gerard-VaretHillairet19} the interactions in our case also depend on the %particle orientations, though, which renders the analysis more subtle. Moreover, we use quenched  stochastic two-%scale convergence, introduced in \cite{ZhiPia2006}, that allows us to shortcut some of the arguments in  %\cite{Gerard-VaretHillairet19}.

    \medskip 
    
    The rigorous derivation of the fully coupled problem  \eqref{limit.nonlinear} is much more challenging than the stationary problem. In particular, all presently known results are perturbative,  in the sense that the volume fraction $\phi$ is assumed to vanish asymptotically as the number of particles tends to infinity. Moreover, a stringent separation condition like \eqref{ass:separation} is assumed. For spherical buoyant particles, the transport-Stokes equations have been derived in \cite{Hofer18MeanField, Mecherbet19, HoferSchubert23} as a mean-field limit. Due to the buoyancy, the leading order effect of the particles on the fluid in \cite{Hofer18MeanField, Mecherbet19, HoferSchubert23} is an additional effective force in the Stokes equations proportional  to the particle density. This effect in particular dominates over the increase of viscosity, which has been captured in the same setting in \cite{Hofer&Schubert} by obtaining improved convergence rates. 

The dynamic problem for nonspherical particles has been studied in \cite{Duerinckx23}. Apart from possible additional active selfpropulsion of the particles and a possible small buoyancy, the microscopic setting in \cite{Duerinckx23} is the system \eqref{main}.
In \cite{Duerinckx23},  the author shows that under assumptions \eqref{ass:separation} and \eqref{ass:diluteness} and convergence of the initial data $f_N^0 \wto f^0$, the empirical measure $f_N$  converges to the solution of
\begin{align} 
    \partial_t f + u \cdot \nabla_x f + \dv_\xi((M(\xi)\nabla u) \xi f) = 0,  \label{eq:f}\\
    f(0,\cdot) = f^0, \label{eq:f^0}\\
    - \Delta u + \nabla p = g, \qquad \dv u = 0. \label{u}
\end{align}
More precisely, assuming the empirical measure 
satisfies 
\begin{align} 
\mathcal W_\infty(f_N^0,f^0) \to 0 
\end{align}
\cite[Proposition 5.2]{Duerinckx23} implies that for given $T>0$ we have (cf. \cite[Equation (5.22)]{Duerinckx23})
\begin{align}
\dmin(t) \geq \dmin(0) e^{-Ct}
\end{align}
and
\begin{align} \label{Mitia.0}
	\mathcal W_\infty(f_N(t),f(t)) \leq C (\mathcal W_\infty(f_N^0,f^0) + \phi \log N)
\end{align} 
where $C$ only depends on $T$, $f^0$ and the constant from \eqref{ass:separation}.

Moreover,  \cite[Theorem 1.3]{Duerinckx23} asserts that the error for the \emph{spatial} density and the fluid velocity can be improved by comparing to the system \eqref{limit.nonlinear}.
%\begin{align} \label{limit.nonlinear}
%\begin{aligned}
%\partial_t f_\phi + u_\phi \cdot \nabla_x f_\phi + \dv_\xi((M(\xi)\nabla u_\phi)\xi f_\phi) = 0\\
%    f_\phi(0,\cdot) = f^0, \\
%    - \dv((2 + \phi  \Sigma[f_\phi])  D u_\phi) + \nabla p_\phi = g, \qquad \dv u_\phi = 0.
%    \end{aligned}
%\end{align}
%Here, 
%\begin{align} \label{sigma}
%\Sigma[h] = \int_{\S^2} Z(\xi)  h \dd \xi,  
%\end{align}
%where $ Z(\xi) \in \mathcal L(\sym_0(3),\sym_0(3))$ is the linear map that determines the moment of stress 
% of  an isolated particle in a Stokes flow with prescribed strain and rotation at infinity, see Section \ref{sec:single}.
More precisely, 
\begin{align} \label{Mitia}
	\mathcal W_\infty(\rho_N, \bar \rho_\phi) \leq C (\mathcal W_\infty(f_N^0,f^0) + \phi^2 |\log \phi| \log N),
\end{align}
and a similar error for $u_N - \bar u_\phi$ in suitable norms.
The spatial densities are defined as  $\rho_N = \int_{\S^2} f_N \dd \xi, \bar \rho_\phi = \int_{\S^2} \bar f_\phi \dd \xi$. We emphasize once more that $\phi$ vanishes as $N \to \infty$ due to assumption \eqref{ass:diluteness}. 

These results, \eqref{Mitia.0}--\eqref{Mitia} do not provide any information on how particle interaction through the fluid affects the evolution of their orientations. On the one hand, \eqref{Mitia.0} only gives a comparison to the leading order effective system \eqref{eq:f}--\eqref{u}, which does not involve any interaction of the particles through the fluid as  the fluid equation \eqref{u} does not depend on $f$. On the other hand, estimate \eqref{Mitia} does not provide any improvement over \eqref{Mitia.0} regarding the evolution of the particle orientations. One could for instance obtain the same $O(\phi^2)$ approximation for the density replacing \eqref{limit.nonlinear} by the intermediate system  
\begin{align*} 
\begin{aligned}
\partial_t f^{int}_\phi + u^{int}_\phi \cdot \nabla_x  f^{int}_\phi + \dv_\xi((M(\xi)\nabla u)\xi f^{int}_\phi) = 0\\
     f^{int}_\phi(0,\cdot) = f^0, \\
    - \dv((2 + \phi  \mV[f^{int}_\phi])  D u^{int}_\phi) + \nabla p^{int}_\phi = g, \qquad \dv u^{int}_\phi = 0.
    \end{aligned}
\end{align*}
that is replacing $\bar u_\phi$ by $u$ in the expression for the angular velocity.

\medskip
This limitation of Doi type systems was emphasized in the recent paper \cite{HoferMecherbetSchubert22}, dealing with a  closely related toy model. The authors proved that no mean-field limit exists that captures the dynamics to order $O(\phi)$. More precisely, they built two initial configurations of the microscopic problem $f_N^0, \tilde f_N^0$ such that $\mathcal W_\infty(f_N^0,\tilde f_N^0) = o(\phi)$ but, for fixed $t > 0$, $\mathcal W_\infty(f_N(t),\tilde f_N(t)) \geq O(\phi)$.

\medskip
This shows that \eqref{limit.nonlinear} is not appropriate if one wants to describe phenomena where the orientation of the particles plays a role (mixing phenomena, phase transition \dots). The main input of the present paper is to derive a substitute to \eqref{limit.nonlinear}  when the initial distribution of the particles is random stationary.

%which appears more naturally as a mean field limit. 
%This latter comparison is the content of \cite[Proposition 5.1]{Duerinckx23} 
% }

\subsection{Outline of the paper} \label{subsec:outline}

The statements of our main results are contained in Section \ref{main.results}. They culminate in Theorem \ref{thm.f-f_phi}, which applies to particle configurations that are initially generated by a stationary ergodic point process. Roughly, for such configurations, we prove that there exists a deterministic matrix field $B = B(t,x,\xi)$ such that the local interaction between particles is captured at first order in $\phi$ by the model: 
\begin{align} \label{f_phi}
        \partial_t f_\phi + u_\phi \cdot \nabla f_\phi + \dv_\xi(M \nabla u_\phi \xi f_\phi  + \phi M B \xi f_\phi) &= 0  \\
        f_\phi(0) &= f^0  
    \end{align}
Here, $u_\phi$ is essentially the same as $\bar u_\phi$, see \eqref{u_phi}. We show an estimate of the form $\|f_N - f_\phi\| = o(\phi)$ in some negative Sobolev space. Note that this $o(\phi)$ approximation holds in a weaker topology than the one induced by the Wasserstein distance, so that there is no contradiction with the result in \cite{HoferMecherbetSchubert22} alluded to above. See also Remark \ref{rem_thm_2.6}. 

\medskip
The proof of Theorem \ref{thm.f-f_phi} requires several steps of different mathematical flavors. 

\smallskip
\textbf{Step 1, Subsections \ref{sec:MOR} and \ref{subsec:proof_deterministic1}.} We show that the empirical measure $f_N$ in  \eqref{empirical_measure}, associated with the microscopic system \eqref{main}-\eqref{x.dot}-\eqref{xi.dot}, is $o(\phi)$ close to the solution $f_N^{app}$ of an intermediate $N$-particle system in the infinite Wasserstein distance. In this intermediate system, the spatial particle positions are transported by the velocity $\bar u_\phi$ in (\ref{limit.nonlinear}b), while the orientations are changed according to the zero order fluid gradient $\nabla u$ plus a singular binary interaction term between the particles involving the $-3$-homogeneous function $\nabla^2 G$, where $G$ is the Oseen tensor. See \eqref{X^app}-\eqref{xi^app} and Theorem \ref{th:app} for a precise statement. The key ingredient in the proof of Theorem \ref{th:app} is the method of reflections for non-spherical particles, which  allows to derive an accurate approximation of the solution $u_N$ of \eqref{main}.  

In view of Theorem \ref{th:app}, the main issue for an accurate description of the particle orientations to first order in $\phi$ is to understand the  singular  binary interaction: since the singularity is critical, the interaction of  very close pairs of particles contributes significantly to the dynamics and therefore no mean-field limit is expected to describe well the limiting behavior of particle orientations to order $O(\phi)$ in the usual Wasserstein metric. In particular, the model \eqref{limit.nonlinear} does not give an effective description to order $O(\phi)$.

\smallskip
\textbf{Step 2, Subsection \ref{subsec:proof_deterministic2}.} To tackle the limit of the aforementioned $-3$-homogeneous binary interaction term, we re-express it in terms of the energy of some special Stokes flows, forced by smeared out dipoles, see Theorem \ref{theo2}. This approach, pioneered by S. Serfaty and her co-authors for the analysis of Coulomb gases \cite{Serfaty15}, was already used in \cite{Gerard-VaretHillairet19} for the analysis of the effective viscosity of suspensions of spheres. We adapt it here to the case of non-spherical particles.  We stress that Steps 1 and 2 are purely deterministic. They are treated in Section \ref{sec_determin}. 

\smallskip
\textbf{Step 3, Subsection \ref{subsec:proof_random1}.} To push the analysis further, we assume that, initially, the particle configuration is generated by a stationary ergodic point process. Under this assumption, using the expression of the binary interaction derived in the previous step, we show  (cf. Theorem \ref{theo3}) that it converges almost surely as $N \rightarrow +\infty$ and that the limit is deterministic. The proof of Theorem \ref{theo3} relies on homogenization theory,  already encountered in \cite{Gerard-VaretHillairet19,DuerinckxGloria21,DuerinckxGloria20}. However, we face here the extra-difficulty that  
 the stationarity of the particles configuration, satisfied at $t=0$, is not preserved at positive times. We manage to overcome this problem by adapting the theory of quenched two-scale convergence developed in \cite{ZhiPia2006}. 

\smallskip
\textbf{Step 4, Subsections \ref{sec:correlation} and \ref{subsec:proof_final}.}
Once we know the limit of the binary interaction term exists, we look for its formula. Under a mild mixing assumption on the initial point process, we show in  Theorem \ref{cor1} that this limit can be expressed with continuous quantities, namely the $2$-point correlation function of the process and the flow of $u$, \emph{cf.} \eqref{u}.  This allows to derive the appropriate correction to the Doi model, that is the appropriate field $B$ in \eqref{f_phi}. Finally, we conclude the proof of Theorem \ref{thm.f-f_phi}, by proving the $o(\phi)$ estimate on $f_N - f_\phi$.

\subsection{The tensors \texorpdfstring{$M(\xi)$}{M} and \texorpdfstring{$S(\xi)$}{S}} \label{sec:single}
We provide here the definitions of the mobility and stress tensors $M$ and $S$ involved in \eqref{limit.nonlinear} and subsequent systems.
Let $A \in \R^{3\times 3}_0$ be a tracefree matrix. Recall that for $\xi \in \S^2$, we denote by $R_{\xi_i} \in SO(3)$ a rotation matrix  i such that $R_{\xi} e_3 = \xi$. We consider the Stokes problem
\begin{equation} \label{eq:resistance.problem}
\left\{
\begin{array}{ll}	-  \Delta v_A + \nabla p_A = 0, \quad \dv v_A = 0 &\quad  \text{in } \R^3 \setminus (R_\xi\mB), \\
	D v_A = - A_{sym}  & \quad \text{in } R_\xi\mB, \\
    \int_{\partial (R_\xi \mB)} \sigma(v_A) n = 0 = \int_{\partial (R_\xi\mB)} x \times \sigma(v_A) n  , \\
    v_A(x) \to 0 & \quad \text{as } |x| \to \infty
	\end{array} \right.
\end{equation}
Note that on  $R_\xi\mB$ we have $v_A(x) =  - A_{sym} x + S_A x + V_A$ for some $S_A \in \skew(3), V_A \in \R^3$.
Then, the map 
$$M(\xi) \colon \R^{3\times 3}_0 \to \skew_0(3)$$ is defined as
\begin{align}
    M A = A +\nabla v_A(0) \label{def.M}
\end{align}
Note that  $\nabla v_A(0) = \fint_{ R_\xi\mB} \nabla v_A \dd y.$

\medskip

The map $S(\xi) \colon \sym_0(3) \to \sym_0(3)$ is defined by 
\begin{align}
    S A = \int_{\partial (R_\xi \mB)} x \otimes^0_s \sigma(v_A+ Ax,p_A) n \dd x =  \int_{\partial (R_\xi \mB)} x \otimes^0_s \sigma(v_A,p_A) n \dd x + 2 A |\B| ,
\end{align}
where $\otimes^0_s$ stands for the tracefree symmetrized tensor product, i.e.
\begin{align}
    a \otimes_s^0 b := \frac 1 2 (a \otimes b + b \otimes a) - \frac 1 3 a \cdot b \, \Id
\end{align}

\medskip

We remark that both $M$ and $S$ are smooth maps on $\S^2$.

\section{Main results} \label{main.results}

\subsection{Deterministic results}

We introduce $u_\phi \in \dot H^1(\R^3)$ as the solution to
\begin{align} \label{u_phi}
% \partial_t f_\phi + u_\phi \cdot \nabla_x f_\phi + \dv_\xi((M(\xi)\nabla u_\phi) \xi f_\phi) = 0\\
%     f_\phi(0,\cdot) = f^0, \\
    - \Delta u_\phi + \nabla p_\phi = g + \phi \dv(\mV[f]  D u) , \qquad \dv u_\phi = 0,
\end{align}
where $(f,u)$ is the solution to \eqref{eq:f}--\eqref{u}.
 Note that the difference between $u_\phi$ and $\bar u_\phi$ (from \eqref{limit.nonlinear}) is of order  $\phi^2$. Hence, replacing  $\bar u_\phi$ by $u_\phi$ does not affect our results as they all deal with effects of order $\phi$. 
 We refer to Theorem \ref{th:well-posedness} for the well-posedness of the systems \eqref{eq:f}--\eqref{u} and \eqref{eq:f}--\eqref{u_phi}, respectively.

% \rhcomment{We do not really need $f_\phi$ defined like this later on, since it is the \enquote{wrong} object anyways, right?}

%
%  This characterization involves two families of linear maps parametrized by $\xi \in \S^2$: 
%$$M(\xi) \colon \R^{3\times 3}_0 \to \skew(3) $$
%and 
%$$Z(\xi) \colon \sym_0(3) \to \sym_0(3)$$
%They are defined precisely in Section \ref{sec:single}, and are built thanks to the solution of a Stokes equation outside a single particle with orientation $\xi$, see \eqref{eq:resistance.problem}

In our first result we compare $f_N$ with $f_N^{app}$ which is defined as the empirical measure associated with the ODEs
\begin{align} \label{X^app}
& \dot X^{app}_i  = u_\phi(X_i^{app} ),   \\ \label{xi^app}
& \dot \xi^{app}_i  = M(\xi_i^{app} ) \bigg(\nabla u(X_i^{app} ) + \frac{\phi}{N}\sum_{j \neq i} \nabla^2 G(X_i^{app} - X_j^{app} ) \Big(S(\xi_j^{app} ) D u (X_j^{app} )\Big) \bigg)  \xi_i^{app}, \\
& (X_i^{app}(0),\xi_i^{app} (0)) = (X_i^0, \xi_i^0)
\end{align}
Here, $G = (G_{i,j})_{1 \le i,j \le 3}$ is the Stokes fundamental solution (or Oseen tensor), that solves 
$$ -\Delta G + \na P = \delta_0 \textrm{Id}, \quad \div G = 0 \quad \text{in } \R^3,  $$
with $\textrm{Id}$ the $3 \times 3$ identity matrix. It is explicitly given by 
$$ G(x) = \frac{1}{8\pi} \Big( \frac{\textrm{Id}}{|x|} + \frac{x \otimes x}{|x-y|^3}\Big)   $$
while 
$$ P(x) = \frac{1}{4\pi} \frac{x}{|x|^3}. $$ 
The  notation $\na^2 G  A \in \R^{3\times 3}$
for a  matrix $A \in \R^{3\times 3}$ refers to the $3\times 3$ matrix defined by 
$$ (\na^2 G  A)_{ij} = \pa_j \pa_l G_{ik} A_{kl}, \quad \text{ that is} \: \na^2 G  A =   \na \big( (A\na) \cdot G\big) $$  
\begin{theorem} \label{th:app}
Let $q>3$, $g \in W^{1,q}(\R^3) \cap L^1(\R^3)$ and  $f^0 \in \mP(\R^3 \times \S^2)$ satisfying 
\begin{itemize}
\item either $f^0 \in W^{1,q}(\R^3 \times \S^2)$
\item or $f^0(x,\xi) = h^0(x,\xi) \1_{\mO}(x)$ for some bounded  $C^{1,\alpha}$ domain $\mO$, $\alpha \in (0,1)$, and some $h^0\in C^1(\R^3 \times \S^2)$.
\end{itemize}
    Assume that \eqref{ass:separation} and \eqref{ass:diluteness} hold and that $\mathcal W_\infty(f_N^0,f^0) \to 0$. Let $T>0$. 
     Then, there exist $N_0,C> 0$, depending only on $T$, $g$, $f^0$ and $c$ from \eqref{ass:diluteness} such that for  all $t \leq T$ and $N \geq N_0$
     \begin{align} \label{W_infty.app}
         \mathcal W_\infty(f_N(t),f_N^{app}(t)) \leq C (\mathcal W_\infty(f_N^0,f^0) + \phi^{4/3} + \phi^2 |\log \phi| \log N), \\
       \frac 1 C |X_i^0 - X_j^0|  \leq |X_i - X_j| \leq  C |X_i^0 - X_j^0| \label{dmin.control.thm},
     \end{align}
     and \eqref{dmin.control.app} also holds for $X_i^{app}$ and $X_j^{app}$.
\end{theorem}
This result should be compared with the approximation result \eqref{Mitia} established in \cite{Duerinckx23}. The bound \eqref{Mitia} can be directly deduced from our proof of \eqref{W_infty.app}, more specifically from \eqref{app.positions.1}. In particular, the  error term $\phi^{4/3}$ in \eqref{W_infty.app} only arises from the approximations of the orientations (cf. \eqref{u_N-u.xi.1st.order} and \eqref{S_4.est}), and does not show up in an estimate for the spatial density. Conversely, it is  likely that the method used in \cite{Duerinckx23} would provide a bound similar to \eqref{W_infty.app}. Here, in order to derive an accurate approximation of the fluid velocity $u_N$, and rather than the cluster expansion of \cite{Duerinckx23}, we consider the method of reflections,  used for spherical particles in \cite{Hofer18MeanField} and for nonspherical particles in \cite{HillairetWu19, HoeferLeocataMecherbet22}. However, compared to \cite{HillairetWu19, HoeferLeocataMecherbet22} we need some additional pointwise estimates that we develop in Subsection \ref{sec:MOR}. From those estimates, we first obtain through a bootstrap argument an estimate for $\mathcal W_\infty(f_N,f)$ (cf. Theorem \ref{th:zero.order}), where $f$ is the solution to the effective system \eqref{eq:f}. We remark that compared to the sedimentation problems considered in \cite{Hofer18MeanField, Mecherbet19, Hofer&Schubert}, this step is easier as there is no mean-field effect involved in the zero-order effective system \eqref{eq:f}--\eqref{u}. We then conclude with the estimate \eqref{W_infty.app} in Subsection \ref{subsec:proof_deterministic1}.

\medskip 
In \eqref{xi^app}, the term 
\begin{align}
M(\xi_i^{app}) \bigg( \frac 1 N \sum_{j \neq i}   \nabla^2 G(X_i^{app} - X_j^{app}) \Big(S(\xi_j^{app}) D u (X_j^{app})\Big) \bigg)\xi_i^{app}
\end{align} is more singular than usual mean field terms, due to the critical singularity of $\nabla^2 G$. As we shall see, it depends strongly on the microstructure and does not converge to its naive mean field limit
\begin{align}
M(\xi_i^{app}) \bigg( \nabla^2 G \ast  \int S(\xi) f \dd \xi D u\bigg)(X_i^{app})  \xi_i^{app} 
\end{align}
The analysis of this term is the key feature of our paper.  Passing to the distributional formulation of the Vlasov type equation solved by $f_N^{app}$, it comes down to understanding the limiting behavior of the objects
\begin{align} \label{double.sum.original}
 \frac 1 {N^2} \sum_i \sum_{j \neq i} (\bar M(\xi_i^{app}) \na_\xi \varphi(X_i^{app},\xi_i^{app})) : \bigg( \nabla^2 G(X_i^{app} - X_j^{app})  \Big(S(\xi_j^{app}) D u (X_j^{app})\Big) \bigg).
\end{align}
where $\varphi$ is a test function and where $\bar M(\xi) \in \mathcal L(\R^3 ; \R^{3\times3}_0)$ is the linear map that satisfies
\begin{equation} \label{def_barM}
\begin{aligned}
    (\bar M(\xi) \zeta) : A  = \zeta \cdot (M(\xi) A) \xi \quad \text{for all } A \in \R^{3\times 3}_0, \zeta \in \R^3,
\end{aligned}
\end{equation}
that is $\bar M(\xi) \zeta = M(\xi) (\zeta \otimes \xi)$. 

As mentioned in Subsection \ref{subsec:outline} and worked out in Subsection \ref{subsec:proof_deterministic2}, such singular mean field interactions can be reformulated in terms of energies of solutions of some special Stokes problems, forced by a sum of (regularized)  point dipoles located at $X_i^{app}$. This small scale forcing generates oscillations in the Stokes solutions, so that computing the limit of their energies as $N \rightarrow +\infty$ draws connection to homogenization theory. This suggests to consider random stationary distributions of particles to further investigate the limit.

\subsection{Further results for stationary ergodic initial data} 
 
To push the analysis further, we assume in addition that the initial particle configuration is generated from a stationary ergodic marked point process
 $(X'_k, \xi'_k), k \in \N$ on $\R^3 \times \S^2$ with hardcore condition
 \begin{equation} \label{ass:separation'}
 |X'_k - X'_l| \ge c' > 0
 \end{equation}
 for some deterministic constant $c'$. See Section \ref{section_stationary} for details. Then, given a small parameter $\eps$, and given a  bounded $C^{1,\alpha}$ domain $\mO$, $\alpha \in (0,1)$, we set 
$$\{(X_1^0, \xi_1^0) \dots (X_N^0, \xi_N^0) \}= \{(\eps X'_k, \xi'_k), k \in \Z, \: \eps X_k' \in \mO \}.$$
Note that $N$ is random, but by the ergodic theorem it is almost surely equivalent to $\eps^{-3} |\mO| \lambda$ as $\eps \rightarrow 0$, where $\lambda$ is the intensity of the point process (see \eqref{intensity} for its definition). In particular, almost surely, $N \rightarrow + \infty$ as $\eps \rightarrow 0$. 
%Also, almost surely, the separation assumption %\dgv{$|X^0_i - X^0_j| \ge c N^{-1/3}$} is satisfied for $c := \rh{\frac {c'}{2}} |\mO|^{1/3} \lambda^{1/3}$ and small enough $\eps$. 
%\dgv{Moreover, given some arbitrary $T > 0$, it follows from the deterministic analysis leading to  Theorem \ref{th:app}  that for %some constant $C \ge 1$, for all $t \in [0,T]$,  $|X_i(t) - X_j(t)| \ge \frac{1}{C}|X^0_i - X^0_j| \ge \frac{c}{C} N^{-1/3}$. See %\eqref{dmin.control}.
%We remind that the volume fraction is defined by  
%$$ \phi = (N r^3 |\mB|)/|\mO| \sim \eps^{-3} r^3 |\mB| \lambda$$
%We then  take $\eta = \frac{C c'}{8} \eps$.}  It follows that almost surely,  for $\eps$ small enough, for all  \eqref{eta_N} is satisfied. 
Within such setting, we prove:  
\begin{theorem} \label{theo3}
Let $(X_i^0,\xi_i^0)_{1 \le i \le N}$ be generated by the stationary ergodic marked point process $(X_k',\xi_k')_{k \in \N}$ as specified above.
	Then, for any  $\Psi_1, \Psi_2 \in C(\R^3 \times \S^2; \R_0^{3\times3})$, and any time $t \geq 0$
	\begin{align*} 
	&\lim_{N \to \infty} \frac 1 {N^2} \sum_i \sum_{j \neq i} \Psi_1(X^{app}_i(t),\xi^{app}_i(t)) :  \Big(\nabla^2 G(X^{app}_i(t) - X_j^{app}(t)) \Psi_2(X_j^{app}(t),\xi_j^{app}(t))\Big)
	\end{align*}
 exists almost surely and is deterministic.
	\end{theorem} 
To prove this result, we shall rely on a reformulation of the mean-field binary interaction with Stokes energy functionals and use homogenization tools.  As emphasized in Subsection \ref{subsec:outline}, the novelty and difficulty of the analysis relies on the fact that the stationarity of the particles centers is not preserved through time. Still,  we are able to derive the limit, notably adapting the method of stochastic two-scale convergence. 

It is clear from the proof of Theorem \ref{theo3} that this limit depends on the microstructure of the suspension, contrary to what is assumed in the Doi model. The next step is to obtain an explicit formula in terms of classical quantities for the initial point process. We remind here the notions of $1$-point correlation and  2-point correlation measures
$$\mu_1 \in \mathcal{M}(\R^3 \times \S^2), \quad \mu_2 \in \mathcal{M}(\R^3 \times \S^2 \times \R^3 \times \S^2)$$ defined through
\begin{align}
    \E\bigg[\sum_i  F_1(X'_i,\xi'_i)\bigg] & =  \int F_1(y,\xi)  \mu_1(\d y , \d \xi) \quad \text{for all }   F \in C_c^\infty(\R^3 \times \S^2 )
\end{align}
and 
\begin{align}
    \E\bigg[\sum_i \sum_{j\neq i} F_2(X'_i,\xi'_i, X'_j,\xi'_j)\bigg] &= \int F_2(y_1,\xi_1,y_2,\xi_2) \mu_2(\d  y_1,\d \xi_1,\d y_2,\d \xi_2) \\
    & \hspace{2cm} \text{for all }   F_2 \in C_c^\infty(\R^3 \times \S^2 \times \R^3 \times \S^2) .
\end{align}
By stationarity and ergodicity, the 1-point correlation measure is given by 
$$\mu_1(\d y, \d \xi) = \lambda \kappa(\d\xi), \quad \kappa \in \mathcal{P}(\S^2) $$ 
In particular, the empirical measure  satisfies
\begin{align} \label{f^0}
    f_N^0 \wto f^0 = \frac {1_{\mO}} {|\mO|} \otimes \kappa
\end{align}
As regards $\mu_2$, we assume that it takes the form 
$$ \mu_2(\d y_1,\d \xi_1,\d y_2,\d\xi_2) = \nu_2(y_1-y_2,\d\xi_1,\d\xi_2) \dd y_1 \dd y_2 $$
for some $\nu_2 \in C^0(\R^3 ; \mathcal{M}(\S^2  \times \S^2))$. Dependence on $y_1 - y_2$ is a consequence of stationarity.  Moreover, we assume the mild mixing assumption
\begin{equation} \label{mixing} 
|\nu_2(y,\d\xi_1,\d\xi_2) - \lambda^2 \kappa(\d\xi_1) \kappa(\d\xi_2)|(\S^2 \times \S^2) \le h(y), \quad \int_{\R^3} \frac{h(y)}{1+ |y|^3} \dd y < +\infty 
\end{equation}
where the absolute value on the left-hand side refers to the total variation.

We can then give a formula for the limit in Theorem \ref{theo3}. More precisely, the limit can be expressed in terms of both the $1$-point and $2$-points correlation measures as well as the flow associated with \eqref{eq:f}. By flow, we mean the couple of maps  $\Phi \colon [0,\infty) \times \R^3 \to \R^3$ and $\Xi \colon [0,\infty) \times \R^3 \times \S^2 \to \S^2$ defined by
\begin{align}
    % \begin{array}{rclrcl}
    \partial_s \Phi (s,x) &= u(s,\Phi(s,x)), &&\qquad \Phi(0,x) = x, \label{Phi} \\
    \partial_s \Xi (s,x,\xi) &=M(\Xi(s,x,\xi)) \nabla u(s,\Phi(s,x)) \Xi(s,x,\xi),&&  \qquad \Xi (0,x,\xi) = \xi. \label{Xi}
    % \end{array}
\end{align}

\begin{theorem} \label{cor1} Under the assumptions of the previous theorem and  the  mild mixing assumption \eqref{mixing},
we have, for any $\Psi_1, \Psi_2 \in C(\R^3 \times \S^2; \R_0^{3\times3})$, and any time $t\geq 0$
	\begin{align} \label{limit.sum.thm}
 \begin{aligned}
	& \lim_{N \to \infty} \frac 1 {N^2} \sum_i \sum_{j \neq i} \Psi_1(X^{app}_i(t),\xi_i^{app}(t)) \Big( \nabla^2 G\big(X^{app}_i(t) - X^{app}_j(t)\big) \Psi_2\big(X^{app}_j(t),\xi^{app}_j(t)\big) \Big)   \\
 &=  \int_{(\S^2)^2 \times (\R^3)^2}  \Psi_1(x_1,\xi_1) : \Big( \nabla^2 G(x_1-x_2)   \Psi_2(x_2,\xi_2) \Big) f(t,\d x_1,\d \xi_1) f(t,\d x_2,\d\xi_2)  \\
 & +  \int_{(\S^2)^2  \times (\R^3)^2} {\Psi}_1(x,\xi_1) \Big( \na^2 G(y) \Psi_2(x,\xi_2) \Big)\nu_{2,t}(\d x, \d y,\d\xi_1,\d\xi_2)    , 
 % &=  \int_{(\S^2)^2 \times (\Phi(t,\mO))^2}  \Psi_1(x_1,\xi_1) : \Big( \nabla^2 G(x_1-x_2) :  \Psi_2(x_2,\xi_2) \lambda^2 \kappa_t(x_1,\dd \xi_1) \kappa_t(x_2,\dd\xi_2)  \dd x_1 \dd x_2 \\
 % & +  \int_{(\S^2)^2  \times \R^3 \times \Phi(t,\mO)} {\Psi}_1(x,\xi_1) \otimes \Psi_2(x,\xi_2) \na^2 G(y) \left(\nu_{2,t}(x,y,\dd\xi_1,\dd\xi_2)  -  \lambda^2 \kappa_t(x,\dd\xi_1) \kappa_t(x,\dd\xi_2) \right)  \dd y \dd x, 
 \end{aligned}
 \end{align}
    where
\begin{align} \label{nu_2,t}
    % \mu_t(\dd x, \dd \xi) &= (\Phi(t,x),\Xi(t,x,\xi))\# (\lambda \1_{\mO}(x) \dd x  \kappa(\dd \xi)), \\
      \nu_{2,t}(\d x, \d y, \d \xi_1, \d \xi_2)  
    &= \frac{1}{\lambda^2 |\mO|^2} \big(\Phi(t,x), \nabla \Phi(t,x) y , \Xi(t,x,\xi_1),\Xi(t,x,\xi_2)\big) \\
    & \hspace{1cm}\#\big( 1_{\mO}(x) \dd x  \, \nu_2(y,\d \xi_1,\d \xi_2) \d y - \lambda^2  1_{\mO}(x) \dd x \dd y \, \kappa(\d \xi_1)  \kappa(\d \xi_2)\big).
\end{align}
% \begin{align}
%     \kappa_t(x,\dd \xi) &= \Xi(t,\Phi^{-1}(t,x),\xi) \# \kappa(\dd \xi), \\
%     \nu_{2,t}(x,y,\xi_1,\xi_2) &= (\Xi(t,\Phi^{-1}(t,x),\xi_1),\Xi(t,\Phi^{-1}(t,x),\xi_2)) \# \nu_2((\nabla \Phi(t,\Phi^{-1}(t,x)))^{-1}y,\dd \xi_1,\dd \xi_2),
% \end{align}
% and where $\Phi^{-1}(t,x)$ is the inverse of the map $x \mapsto \Phi(t,x)$ for fixed $t$.
	\end{theorem}
\begin{remark} \label{rem:well.defined}
\begin{enumerate}
    \item Note that $\na^2 G$ is homogeneous of order $-3$, so that there is \emph{a priori} a problem of integrability for the first term at the right-hand side of \eqref{limit.sum.thm},  both near the diagonal $x_1 = x_2$ and at infinity. Still, this integral can be shown to be well-defined thanks to the Calderon-Zygmund theorem for singular integrals, see Subsection \ref{sec:correlation} for details. As regards the second term at the right-hand side, integrability at infinity comes from the mixing assumption \eqref{mixing}, while the singularity at $y=0$ is removed through implicit use of the principal value. More precisely, as $\nu_2$ is zero for $|y| < c'$ thanks to \eqref{ass:separation'}, $\nu_{2,t}$ is homogeneous in $y$ near the origin, so that for $\delta = \delta_t > 0$ small enough
\begin{align*}
& \int_{(\S^2)^2  \times (\R^3)^2} {\Psi}_1(x,\xi_1) \Big( \na^2 G(y) \Psi_2(x,\xi_2) \Big)\nu_{2,t}(\d x, \d y,\d\xi_1,\d\xi_2) \\
= &      \int_{(\S^2)^2  \times (\R^3)^2} {\Psi}_1(x,\xi_1) \Big( \na^2 G(y) 1_{|y| \ge \delta}\Psi_2(x,\xi_2) \Big)\nu_{2,t}(\d x, \d y,\d\xi_1,\d\xi_2)
\end{align*}

    \item The first term on the right-hand side of \eqref{limit.sum.thm} corresponds to the naive mean field limit: If $G$ was less singular at the origin, we would only get this term. 
    In particular, inserting $\Psi_1(x,\xi) = \bar M(\xi) \nabla_\xi \varphi(x,\xi)$ and $\Psi_2(x,\xi) = S(\xi) Du(t,x)$ (see \eqref{double.sum.original}), we obtain that this term satisfies
    \begin{align}
        &\int_{(\S^2)^2 \times (\R^3)^2}  (\bar M(\xi_1) \nabla_\xi \varphi(x_1,\xi_1)) : (\nabla^2 G(x_1-x_2)   (S(\xi_2) Du(t,x_2)))  \\
        & \hspace{8cm} f(t,\d x_1,\d \xi_1) f(t,\d x_2,\d\xi_2) \\
        &\hspace{5cm} = \frac 1 \phi \int_{\R^3 \times \S^2} \nabla_\xi \varphi(x,\xi) M D(u_\phi - u)(t,x) \xi f(t,\d x, \d \xi)
    \end{align}
    Indeed, by definition of $u_\phi$ in \eqref{u_phi}, we have
    \begin{align}
       D(u_\phi - u)(t,z) &= \phi (\nabla^2 G \ast \mathcal V[f] Du)(t,z) \\ &= \phi \int_{\R^3 \times \S^2} \nabla^2 G(z - x) (S(\xi) Du(t,x))  f(t,\d x,\d \xi).
    \end{align}
% This follows from the fact that $x \mapsto \Phi(t,x)$ is a volume preserving diffeomorphism.
% Observe that $\lambda \1_{\mO}(x) \dd x  \kappa(\dd \xi) = f^0(\dd x,\dd \xi)$ and hence $\mu_t = f(t,\cdot)$.
% \rhcomment{I prefer this second formulation but both are a bit heavy to write. Maybe write $t$ as an index rather than a variable everywhere to make it slightly more compact. \dgv{I agree this second formulation is better, notably with regards to the proof of Theorem 1.6}}
\end{enumerate}
\end{remark}

% \begin{theorem}
%     Under the same assumptions as in the previous theorem, let $f_\phi$ be the solution to the problem
%     \begin{align}
%         \partial_t f_\phi + u_\phi \cdot \nabla f_\phi + \dv_\xi(M \nabla u \xi f_\phi + M \nabla (u_\phi - u) \xi f + \phi MA \xi) &= 0 ,\\
%         f_\phi(0) &= f^0,
%     \end{align}
%     where
%     \begin{align}
%         A(t,x,\dd \xi) = \int_{\S^2 \times \R^3} S(\zeta) Du(x) \na^2 G(y) \left(\nu_{2,t}(x,y,\dd\zeta,\dd\xi)  -  \lambda^2 \kappa_t(x,\dd\zeta) \kappa_t(x,\dd\xi) \right)  \dd y \1_{\Phi(t,\mO)}(x).
%     \end{align}
%     Assume that 
%     \begin{align}
%         \mathcal W_\infty(f_N^0,f^0) = o(\phi).
%     \end{align}
%     Then, for all $t>0$
%     \begin{align}
%         \mathcal W_\infty(f_N(t),f_\phi(t)) = o(\phi).
%     \end{align}
% \end{theorem}

    Our final and main result is that $f_N$ is well approximated at order $\phi$ through the solution $f_\phi$ to  \eqref{f_phi}
 % \begin{align}
 %        \partial_t f_\phi + u_\phi \cdot \nabla f_\phi + \dv_\xi(M \nabla u_\phi \xi f_\phi  + \phi M B \xi f_\phi) &= 0 \label{f_phi} \\
 %        f_\phi(0) &= f^0 = \lambda \1_{\mO} \otimes \kappa
 %    \end{align}
    for a suitable function $B$. Since without $B$, $f_\phi$ would solve the \enquote{naive} mean-field limit,  $B$ must (to leading order) stem from the second right-hand side term in \eqref{limit.sum.thm} with the choice $\Psi_1(x,\xi) = \bar M(\xi) \nabla_\xi \varphi(x,\xi)$ and $\Psi_2(x,\xi) = S(\xi) Du(x)$, see \eqref{double.sum.original}. 
    By the Radon-Nikodym theorem and the disintegration theorem, we get the following representation, which we prove in Appendix \ref{sec:App.Disintegration}.
    \begin{lemma} \label{lem:Disintegration}
        There exists a measurable matrix field matrix field $B = B(t,x,\xi)$ such that for all $t \geq 0$ and all $\varphi \in C_c^\infty(\R^3 \times \S^2)$
        \begin{align}
              & \int_{(\S^2)^2  \times (\R^3)^2} \Big(\bar M(\xi_1) \nabla_\xi \varphi(x,\xi_1) \Big) : \Big( \na^2 G(y) S(\xi_2) Du(x) \Big) \nu_{2,t}(\d x, \d y,\d\xi_1,\d\xi_2)\\
   = &  \int_{\R^3 \times \S^2}   \nabla_\xi \varphi(x,\xi_1) : \big(M(\xi_1) B(t,x,\xi_1) \xi_1 \big) f(t,\d x, \d\xi_1) .
        \end{align}
    \end{lemma}

%By the disintegration theorem, we have
%\begin{align*}
%    \mu_2(\d y_1,\d \xi_1, \d y_2, \d \xi_2) =  \tilde  \sigma(y_1,\xi_1, \d y_2, \d \xi_2) \mu_1(\d y_1,%\d \xi_1)
%\end{align*}
%where $(y_1,\xi_1) \mapsto \tilde  \sigma(y_1,\xi_1,\cdot)$
%is a  measurable map from $\R^3 \times \S^2$ to $\mP(\R^3 \times \S^2)$. 
%Hence, 
%\begin{align}
%    \nu_2(y,\d \xi_1, \d \xi_2) =  \sigma(y, \xi_1, \d \xi_2)\lambda \kappa(\d \xi_1)
%\end{align}
%where $(y_1,\xi_1) \mapsto  \sigma(y_1,\xi_1,\cdot)$
%is a measurable map from $\R^3 \times \S^2$ to $\mP( \S^2)$. 
%We then denote
%\begin{align}
%    \sigma_{t,x}(y, \xi, \d \zeta) = 
%    \Xi(t,x,\zeta) \# \sigma((\nabla \Phi(t,x) )^{-1} y,  \Xi^{-1}(t,x,\xi), \d \zeta)
%\end{align}
%and
%     \begin{align} \label{B}
%        B(t,x,\xi) := \int_{\S^2 \times \R^3} S(\zeta) Du(x) \na^2 G(y) \left(\sigma_{t,x}(y, \xi_2, \d \zeta)  - \lambda \Xi(t,x,\zeta) \# \kappa(\d \zeta) \right)  \dd y.
%    \end{align} 
%such that, the second term on the right-hand side of \eqref{limit.sum.thm} for $\Psi_2(x,\xi) = S(\xi) %Du(x)$ is given by

\begin{theorem} \label{thm.f-f_phi} Assume that $B \in L^\infty(0,T;L^\infty_x W^{1,\infty}_\xi)$. 
Then, for all $T>0$, there exists a unique weak solution $f \in L^{\infty}((0,T) \times \R^3 \times \S^2)$ to \eqref{f_phi} with $f^0$ given through \eqref{f^0}. Assume further that there exists  $q > 5$, such that $N^{-1/3 + 1/q}= o(\phi)$, and $\|f_N^0 - f^0\|_{W^{-1,q'}_x W^{-2,q'}_\xi} = o(\phi)$. Under these assumptions and those of the previous theorem, 
     $$ \lim_{N \to \infty} \sup_{t \in (0,T)} \frac 1 \phi \|f_N(t)-f_\phi(t)\|_{W^{-1,q'}_x W^{-2,q'}_\xi} = 0.$$ 
\end{theorem}
Here, we denote $W^{-1,q'}_x W^{-2,q'}_\xi$ as the dual of
\begin{align}
    W^{1,q}_x W^{2,q}_\xi := \{ \varphi \in W^{1,q}(\R^3 \times \S^2) : \nabla_\xi \varphi \in  W^{1,q}(\R^3 \times \S^2)\}.
\end{align}
\begin{remark} \label{rem_thm_2.6}
\begin{enumerate}
\item The condition $N^{-1/3 + 1/q}= o(\phi)$ is an additional assumption made on $\phi$ in order to simplify our approximation statement. It ensures that spatial discretization errors are $o(\phi)$: consider the empirical measure $g_N = N^{-1} \sum_{i =1}^N \delta_{X_i} \in \mathcal P(\R^3)$ for a hardcore constant $|X_i - X_j | \geq c N^{-1/3}$, and its smeared out  continuous version $\frac 1 N c^{-3} | B_1|^{-1}\sum_{i =1}^N \1_{B(X_i,c N^{-1/3})}$. Then $\langle g_N - g, \varphi \rangle \leq N^{-1/3 + 1/q} [\varphi]_{C^{0,1-3/q}}$ for any $\varphi \in W^{1,q}(\R^3)
\subset C^{0,1-3/q}(\R^3)$. Moreover, the bound is sharp in the following sense: For any sequence of empirical measures $g_N \in \mathcal P(\R^3)$ and any continuous density $g \in \mathcal P$, one has $\|g_N - g\|_{(C^{0,\alpha}(\R^3))'} \geq O( N^{-\alpha/3}) $. Indeed, since $g$ is continuous, there exists $\eps > 0$  such that $|\{g \geq \eps\}| > 0$ and we set $d > 0$ such that $|\{g \geq \eps\}| = 2 |B_1| d^3$. Consider 
\begin{align}
    \varphi (x) := \min\bigg\{1, (d N)^{\alpha/3} \min_i |x - X_i|^{\alpha}\bigg\}.
\end{align}
Then $\|\varphi\|_{C^{0,1-3/q}} \lesssim (d N)^{\alpha/3} $, $\langle \varphi, g_N\rangle = 0$ and 
\begin{align}
    \langle \varphi, g\rangle \geq \int_{\R^3 \setminus \cup_{i=1}^N B(X_i,N^{-1/3} d) } \dd g \geq  \eps \bigg| \{g \geq \eps\} \setminus \bigcup_{i=1}^N B(X_i,N^{-1/3} d) \bigg| \geq \eps d^3 |B_1|.
\end{align}

    \item 
 The assumption $B \in L^\infty(0,T;L^\infty_x W^{1,\infty}_\xi)$ is essentially a regularity assumption on $\nu_2$. It is in particular satisfied if $\nu_2$ is of the form  $\tilde \nu_2(y) \kappa(d\xi_1) \kappa(d\xi_2)$, i.e. the initial particle orientations are uncorrelated.
 
 \item We do \emph{not}  expect $\mathcal W_p(f_N(t),f_\phi(t)) = o(\phi)$ (in accordance with the result in \cite{HoferMecherbetSchubert22}). Indeed, the PDE \eqref{f_phi} only captures the average angular velocity of particles with orientation $\xi$ at a macroscopic position $x$, because it arises through the limit of the double sum \eqref{double.sum.original}. However, we expect the actual angular velocities to be sensitive to the microscopic particle positions relative to the other particles, in the sense that two particles with the same orientation $\xi_i = \xi_j$ and very close positions $|X_i - X_j| = O(N^{-1/3})$ can have very different angular velocities $|\dot \xi_i - \dot \xi_j| = O(\phi)$. Such a behavior cannot be captured by the PDE \eqref{f_phi}. The above theorem measures the distance between $f_\phi$ and $f_N$ in a  Sobolev norm of order $-2$ with respect to the orientation variable $\xi$. This space is sufficiently weak that it only sees the the (macroscopic) local  averaged angular particle  velocities. In typical applications, (experimentally) observable quantities are such averaged quantities.

 We leave  the study of more accurate limit models, which need to retain information on the microstructure, to future research.
 \end{enumerate}
\end{remark}

\section{Proof of the  deterministic results} \label{sec_determin}

We first state the following (standard) well-posedness results for the effective systems \eqref{eq:f}--\eqref{u} and \eqref{u_phi}. 
% These can be proved as \cite[Theorems 2.7--2.9]{Hofer&Schubert} and \cite[Proposition 2.1]{Duerinckx23} and we therefore omit the proof here.

\begin{theorem} \label{th:well-posedness}
Let $q>3$, $g \in W^{1,q}(\R^3) \cap L^1(\R^3)$ and either 
\begin{enumerate}[(1)]
    \item $f^0 \in \mP(\R^3 \times \S^2) \cap W^{1,q}(\R^3\times \S^2)$ or
    \item   $ f^0(x,\xi) = h^0(x,\xi) \1_{\mO}(x)$ for some bounded  $C^{1,\alpha}$ domain $\mO$, $\alpha \in (0,1)$, and some $h^0\in C^1(\R^3 \times \S^2)$.
\end{enumerate}
Then,
    \begin{enumerate}[(i)]
        \item For all $T > 0$ there is a unique solution 
        $$(u,f) \in W^{3,q}(\R^3) \times L^\infty(0,T;L^\infty(\R^3 \times \S^2))$$
        to  \eqref{eq:f}--\eqref{u}. Moreover, if (1) holds, then $f\in L^\infty(0,T;W^{1,q}(\R^3 \times \S^2))$, and, if (2) holds, then 
         $ f(t,\cdot) = h(t,\cdot) \1_{\Phi_t(\mO)}$ for some $C^2$ diffeomorphism $\Phi_t$, and some $h\in L^\infty(0,T;C^1(\R^3 \times \S^2))$.
        \item For all $T > 0$ there exists a unique solution $u_\phi \in L^\infty(0,T;W^{1,\infty}(\R^3))$ to \eqref{u_phi}. Moreover, 
            \begin{align} \label{regularity.u}
        \|u\|_{L^\infty(0,T;W^{2,\infty}(\R^3))} + \|u_\phi\|_{L^\infty(0,T;W^{1,\infty}(\R^3))} &\leq C, \\
                \|u - u_\phi\|_{W^{1,\infty}} &\lesssim \phi. \label{u-u_phi}
        \end{align}
    where $C$ depends only on $T$,  $f^0$ and $g$.
    \end{enumerate}

\end{theorem}
\begin{proof}
Existence and regularity of $u$ follows from standard regularity theory of the Stokes equation. Existence and regularity of $f$ then follows from the fact that $u$ is divergence free and the regularity of $u$. Indeed, the flow $(\Phi,\Xi)$ defined in \eqref{Phi}--\eqref{Xi} is then well defined and
\begin{align}
    f(t,\cdot,\cdot) = (\Phi(t,\cdot), \Xi(t,\cdot,\cdot))\# f^0.
\end{align}

In case of (1), part (ii) follows from standard $L^p$ regularity theory of the Stokes equation and Sobolev embedding.

In case of (2), we can argue as in the proof of \cite[Proposition 3.3]{Gerard-VaretHillairet19} to obtain the desired Lipschitz estimate on $u_\phi$. 

Alternatively, we split $u_\phi = u_\phi^1 + u_\phi^2$, where (supressing the time dependence)
\begin{align} 
    u_\phi^1 =  G \ast ( g + \phi  \dv(\mV[f]  D u)|_{\Phi(\mO)}), \\
       u_\phi^2 = \phi G \ast ((\mV[f]  D u)|_{\partial \Phi(\mO)^-} \nu  \mathcal H^2|_{\partial \Phi(\mO)}),
\end{align}
where $|_{\partial \Phi(\mO)^-}$ denotes the trace taken from the interior. 
Then, $u_\phi^1 \in W^{1,\infty}$ since $g + \phi  \dv(\mV[f]  D u)|_{\Phi(\mO)} \in L^\infty(\R^3) \cap L^1(\R^3)$. 
Moreover, $u_{\phi}^2$ is a single layer potential.
By the H\"older regularity of $((\mV[f]  D u)|_{\partial \Phi(\mO)^-} \nu$ on the $C^{1,\alpha}$ surface ${\partial \Phi(\mO)^-}$,  this single layer potential  satisfies $u_\phi^2 \in W^{1,\infty}(\R^3)$ by the Lyapunov Theorem (see e.g. \cite[Chapter II.7]{Guenter67} for the Poisson equation and \cite[Chapter 3.2]{Ladyzhenskaya69} for the adaptations to the Stokes equation).

By the same argument, \eqref{u-u_phi} follows since $u_{diff} = u - u_\phi$ solves
    \begin{align}
        -\Delta u_{diff} + \nabla p_{diff} = \phi\dv ( \mV[f] D u), \qquad \dv u_{diff} = 0. \label{u_diff}
    \end{align}    
% For estimate \eqref{u-u_phi},  we observe that 
%     and hence, the definition of $\mV[h]$ in \eqref{sigma}, standard $L^p$ regularity of the Stokes equations, Sobolev embedding and the regularity of $f$ and $u$ provided by Theorem \ref{th:well-posedness} yield
%     \begin{align}
%         \|u_{diff}\|_{W^{1,\infty}(\R^3)} \lesssim \phi.
%     \end{align}
\end{proof}

\subsection{Approximation of the fluid velocity through the method of reflections} \label{sec:MOR}

The key ingredient towards the proof of Theorem \ref{th:app} is an approximation  of the fluid velocity $u_N$ through the method of reflections.
Throughout this section, we consider a fixed particle configuration $(X_i,\xi_i)_{i=1}^N$ satisfying 
\begin{align}
    \dmin \geq 4 r.
\end{align}
We denote for $1 \leq i \leq N$ and $k \in \N$
\begin{align}
    d_{ij} &:= |X_i - X_j|, \\
    S_k &:= \frac 1 N \sup_i \sum_{j \neq i} \frac 1 {d_{ij}^k}.
\end{align}
Moreover, we denote by $u,u_N \in \dot H^1(\R^3)$ the unique weak solutions to \eqref{u} and \eqref{main}, respectively.

We will use the following estimates (see for instance \cite[Lemma 4.8]{NiethammerSchubert19})
\begin{align} \label{est:S_k}
    N S_{k} \lesssim \begin{cases} \displaystyle{\frac{N^{\frac{3-k}{3}}}{\dmin^k}} &\quad \text{for } k=1,2\\[12pt]
   \displaystyle{ \frac{\log N}{\dmin^3}} &\quad \text{for } k=3\\[12pt]
    \displaystyle{\frac{1}{\dmin^k}}  & \quad \text{for } k \geq 4.
    \end{cases}
\end{align}

We will frequently use the short notation for averages
\begin{align}
    (h)_i := \fint_{B_i} h \dd x
\end{align}

\begin{proposition}\label{pro:MOR}
    There exists $\eps >0$ with the following property. Let $X_i \in \R^3$, $1\leq i \leq N$, and $r>0$ be given such that $\dmin \geq 4r$ and $\phi S_3 < \eps$. Let $u_N$ and $u$ be the solution to \eqref{main} and \eqref{u}, respectively. Then for all $1 \leq i \leq N$
    \begin{align}
         |(u_N- u)(X_i)| &\lesssim r+ \phi S_2, \label{u_N-u}\\
         |u_N(X_i) - u_N(X_j)| &\lesssim d_{ij} \label{u_N.Lipschitz} , \\
         |(u_N - u)(X_i) - (u_N-u)(X_j)| &\lesssim  \Big(\frac r {d_{ij}} + \phi S_3 \Big) d_{ij} \label{nabla.u-u_N} , \\
          |\nabla u_N(X_i) - M(\xi_i) \nabla u(X_i)| &\lesssim r +  \phi S_3.\label{u_N-u.xi}
    \end{align}
    Moreover,
    \begin{align} \label{u_N-u.1st.order}
        \bigg| (u_N- u)(X_i) - \frac \phi N \sum_{j \neq i} \nabla G(X_j - X_i) : S(\xi_j ) D u (X_j)\bigg|  &\lesssim r + \phi^2 S_2 S_3  , 
    \end{align}
    \begin{align}
        \bigg|\nabla u_N(X_i)  - M(\xi_i ) \Big(\nabla u(X_i ) + \frac{\phi}{N}\sum_{j \neq i} \nabla^2 G(X_i - X_j ) S(\xi_j ) D u (X_j)  \Big)\bigg| 
     \lesssim  (\phi S_3)^2 + r + r^4 N S_4 \qquad \label{u_N-u.xi.1st.order}
    \end{align}
\end{proposition}

The proof is based on the method of reflections that has been extensively used in related problems, see e.g. \cite{Hofer18MeanField, Hoefer19} for spherical particles and \cite{HillairetWu19, HoeferLeocataMecherbet22} for non-spherical particles. 
The desired pointwise estimates stated in Proposition \ref{pro:MOR} for non-spherical particles seemingly have not appeared in the literature.

To prove them, we closely follow the setup in \cite{HoeferLeocataMecherbet22}. Let $\dot H^1_{\mathfrak s}(\R^3)$ the space of solenoidal vector fields in $\dot H^1(\R^3)$.  Let $Q_i \colon \dot H^1_{\mathfrak s}(\R^3) \to \dot H^1_{\mathfrak s}(\R^3)$ be the orthogonal projection operator  $v \mapsto Q_i v$,
\begin{equation} \label{eq:Q}
\left\{
\begin{array}{ll}
	- \Delta (Q_i v) + \nabla p = 0, \quad \dv v =  0& \quad \text{ in } \R^3 \setminus \mB_i, \\
	D (Q_i v) = D v &\quad \text{ in } \mB_i, \\
0 = \int_{\partial \mB_i} \sigma(v,p) n = \int_{\partial \mB_i} \sigma(v,p) n \times (x- X_i).
	\end{array}\right.
\end{equation} 

The main point is then to show that under the assumptions of Proposition \ref{pro:MOR}
\begin{align}
    u_N = \lim_{k \to \infty} \bigg(\Id - \sum_{i} Q_i\bigg)^k u,
\end{align}
which allows to obtain the estimates in Proposition \ref{pro:MOR} through  decay estimates on $Q_i$.

For the proof, we rely on the following lemmas. 
\begin{lemma}[{\cite[Lemma 4.8]{HoeferLeocataMecherbet22}}] \label{lem:decay.Q}
Let $w \in W^{1,\infty}(\mathcal{B}_i)$ with $\dv w = 0$. There exists a universal constant $C>0$ such that for $l=0,1,2$ and all $x \in \R^3 \setminus B(X_i,2r)$
\begin{align} \label{est:Q_i.pointwise}
	|\nabla^l (Q_i w)(x)| \leq \frac {C r^3}{|x - X_i|^{l+2}} \|D w\|_{L^\infty(\B_i)} .
\end{align}
Moreover, 
	\begin{align} \label{est:Q_i.H^1}
		\|Q_i w\|_{\dot H^1(\R^3)} \leq C r^{3/2} \|D w\|_{L^\infty(\B_i)}.
	\end{align}
\end{lemma}

\begin{lemma}
Let $w \in W^{2,\infty}(\mathcal{B}_i)$ with $\dv w = 0$.   Then,  exists  $C>0$ such that 
    \begin{align} \label{Q_i.M}
        |(\nabla w - \nabla Q_i w)_i - M(\xi_i) (\nabla w)_i| \leq C r \|\nabla D w\|_{L^\infty(B_i)} . 
    \end{align}
    Moreover, for $l=0,1$ and all $x \in \R^3 \setminus B(X_i,4r)$
    \begin{align} \label{Q_i.Z}
       & \left|\nabla^l\left((Q_i w)(x) - r^3 \nabla G(x-X_i) : S(\xi_i) (Dw)(X_i)\right)\right| \\
        &\qquad \leq C  \frac{r^4}{|x-X_i|^{3+l}}\left(\|D w\|_{L^\infty(\B_i)} + |x-X_i|\|\nabla D w\|_{L^\infty(B_i)}\right).
    \end{align}
\end{lemma}
\begin{proof}
    Let $\tilde w(x) = (\nabla w)_i x$. By definition of $M(\xi_i)$ (cf. \eqref{def.M}), we have with 
    \begin{align}
      (\nabla w)_i -  M(\xi_i) (\nabla w)_i =  (\nabla (Q_i\tilde w))_i.
    \end{align}
    Hence, by \eqref{est:Q_i.H^1}
    \begin{align}
        |(\nabla w - \nabla Q_i w)_i - M(\xi_i) (\nabla w)_i| &= |(\nabla Q_i (w-\tilde w))_i|\\ &\lesssim r^{-3/2} \|\nabla  Q_i (w-\tilde w)\|_{L^2(\R^3)}  \\
        &\lesssim \|D (w - \tilde w)\|_{L^\infty(B_i)} \lesssim r \|\nabla D w\|_{L^\infty(B_i)}.
    \end{align}

    To prove \eqref{Q_i.Z}, we may assume $X_i=0$. Then, we observe that  \cite[Proposition 2.2]{HillairetWu19} implies that for all $|x| \geq 4r$\footnote{ \cite[Proposition 2.2]{HillairetWu19} gives the estimate for  $r=1$, for general $r$ it  follows from a scaling argument. In the notation of \cite{HillairetWu19}, $\mathcal U$ corresponds to $G$ and $\mathbb M[A,\B]$ corresponds to $S(\xi)  (D w)_i$. Also note the different sign convention for the normal $n$ on $\partial \B$. }
    \begin{align}
        |\nabla^l (Q_i \tilde w -  r^3 \nabla G:S(\xi_i)(D w)_i)(x)| \lesssim \frac{r^{4}}{|x|^{3+l}} |(D w)_i|.
    \end{align}
    Moreover, by \eqref{est:Q_i.pointwise}
    \begin{align}
        |\nabla^l (Q_i (w -\tilde w))(x)| \lesssim \frac{r^3}{|x|^{l+2}} \|D (w-\tilde w)\|_{L^\infty(\B_i)} \lesssim \frac{r^4}{|x|^{l+2}} \|\nabla D w\|_{L^\infty(\B_i)}.
    \end{align}
    Combining these estimates with $|S(\xi)((D w)_i - (Dw)(X_i))| \lesssim r \|\nabla D w\|_{L^\infty(\B_i)}$ and the decay $|\nabla^k G|(x) \lesssim |x|^{-k-1}$ finishes the proof.
\end{proof}

\begin{proof}[Proof of Proposition \ref{pro:MOR}]
    We define iteratively
    \begin{align}
        v_0 := u, \qquad v_{k+1} = v_k - \sum_{i} Q_i v_k. 
    \end{align}

    Using the fact that $D(Q_i v_k) = D v_k$ in $\B_i$ we get 
    $$Dv_{k+1}= D v_k - \underset{j}{\sum} D Q_j v_k=- \underset{j \neq i}{\sum} D Q_j v_k, \text{ in } \B_i.  $$ 
    Thus, using again Lemma \ref{lem:decay.Q}, we get
	\begin{align*} 
		\sup_i \|D v_{k+1}\|_{L^\infty(\B_i)} &\leq \sup_i \sum_{j \neq i}  \|D(Q_j v_k)\|_{L^\infty(\B_i)} \\
		& \leq \sup_i \sum_{j \neq i} \frac{C r^3}{d_{ij}^3}  \|D v_{k}\|_{L^\infty(\B_j)}\\
		& \leq  C \phi S_3 \sup_j \|D v_{k}\|_{L^\infty(\B_j)}.
	\end{align*}
	Hence, by iteration
	\begin{align} \label{est:Dv_k}
		\sup_i \|D v_{k}\|_{L^\infty(\B_i)} \leq (C \phi S_3)^k \sup_i\|D u\|_{L^\infty( \B_i)} \lesssim  (C \phi S_3)^{k}.
	\end{align}
    where in the last estimate we used $v_0 = u \in W^{1,\infty}$ (cf. Theorem \ref{th:well-posedness}).   Choosing $\eps$ from the statement of the proposition such that $C \eps \leq \frac 1 2 $ with $C$ from above, we find that
    $\sup_i \|D v_{k}\|_{L^\infty(\B_i)} \to 0$ as $k \to 0$. We show that this implies that $v_k \to u_N$ in $\dot H^1(\R^3)$. Indeed, the difference $w_k := v_k - u_N$ solves 
\begin{equation} 
\begin{aligned}
 -\Delta w_k + \na p_k  = g, \quad \div w_k = 0, & \qquad \text{in } \R^3 \setminus \cup_i \mB_i \\ 
D(w_k)\vert_{\mB_i}  = D v_k &\qquad \forall i, \\
0 = \int_{\pa \mB_i} \sigma(w_k)n \dd s(x) = \int_{\pa \mB_i} \sigma(w_k)n  \times (x-X_i) \dd s(x) & \qquad \forall i,
 \end{aligned}
\end{equation}
and thus $w_k$ minimizes the Dirichlet energy among all functions $\varphi $ with $D \varphi = D v_k$ in $\cup_i \B_i$. By a classical extension operator (cf. \cite[Lemma 4.6]{HoeferLeocataMecherbet22}) we thus find
\begin{align}
    \|w_k\|_{\dot H^1(\R^3)} \lesssim  \|D v_k\|_{L^2(\cup_i \B_i)} \lesssim  C \phi^{\frac 1 2} (C \phi S_3)^{k} \to 0 \quad \text{as } k \to \infty.
\end{align}

To show \eqref{u_N-u}, we observe that \eqref{est:Q_i.H^1} and Sobolev embedding implies for any $w \in \dot H^1_\sigma(\R^3)$
\begin{align} \label{(Q_iw)_i}
    |(Q_i w)_i| \lesssim r^{-3+5/2} \|Q_i w\|_{L^6(B_i)} \lesssim r \|Dw\|_{L^\infty(B_i)}.
\end{align}
Hence, using also \eqref{est:Q_i.pointwise} and \eqref{est:Dv_k},
\begin{align}
    |(u_N - u)_i| &\leq  \sum_{k=0}^\infty |(v_{k+1} - v_k)_i| \leq \sum_{k=0}^\infty \sum_{j} |(Q_j v_k)_i| \\
    &\lesssim \sum_{k=0}^\infty \bigg( r \|D v_k\|_{L^\infty(B_i)} + \sum_{j\neq i} \frac{r^3}{d_{ij}^2} \|D v_k\|_{L^\infty(B_j)} \bigg) \\
    &\lesssim  (r + \phi S_2)  \sum_{k=0}^\infty(C \phi_n S_3)^{k} \lesssim  (r + \phi S_2).
\end{align}
Using that $(u_N)_i = u_N(X_i)$, since $D u_N = 0$ in $B_i$, and $\|\nabla u\|_{L^\infty} \lesssim 1$ yields \eqref{u_N-u}.

\medskip

For \eqref{nabla.u-u_N}, we compute in a similar way
\begin{align}
    |(u_N - u)_i - (u_N-u)_j| &\leq  \sum_{k=0}^\infty |(v_{k+1} - v_k)_i - (v_{k+1} - v_k)_j| \\
    &\leq \sum_{k=0}^\infty \sum_{l} |(Q_l v_k)_i - (Q_l v_k)_j| \\
    &\lesssim \sum_{k=0}^\infty \bigg(r + \frac{r^3}{d_{ij}^2}\bigg) \left(\|D v_k\|_{L^\infty(B_i)} + \|D v_k\|_{L^\infty(B_j)}\right) \\
    &+ \sum_{k=0}^\infty \sum_{l\not \in \{i,j\}} r^3 d_{ij}
 \bigg(\frac 1 {d_{il}^3}  + \frac 1 {d_{jl}^3} \bigg) \|D v_k\|_{L^\infty(B_l)} \\
        &\lesssim \sum_{k=0}^\infty d_{ij} \bigg(\frac r {d_{ij}} +  \sum_{l\neq i}\frac{r^3}{d_{il}^3} + \sum_{l\neq j}\frac{r^3}{d_{jl}^3}\bigg )\sup_l \|D v_k\|_{L^\infty(B_l)} \\
    &\lesssim  \bigg(\frac r {d_{ij}} + \phi S_3 \bigg) d_{ij}.
\end{align}
Using again $(u_N)_i = u_N(X_i)$ for all $i$ and $\|\nabla u\|_{\infty} \lesssim 1$, this establishes \eqref{nabla.u-u_N}. Moreover, \eqref{u_N.Lipschitz} follows from \eqref{nabla.u-u_N}, $\|\nabla u\|_\infty \lesssim 1$, $r \leq d_{ij}$ and that by assumption $S_3 \phi \lesssim 1$.

\medskip

For the proof of \eqref{u_N-u.xi}, we estimate using \eqref{Q_i.M} as well as $Q_i u_N = 0$ and $Q_i Q_i = Q_i$
\begin{align} \label{u_N-u.xi.0}
    |(\nabla( u_N - u +  Q_i u))_i| &\leq \sum_{k=0}^\infty |(\nabla (\Id - Q_i)(v_{k+1} - v_k))_i| \\
    &\leq \sum_{k=0}^\infty \sum_{j \neq i} |(\nabla (\Id - Q_i)Q_j v_k)_i| \\
    &\lesssim  \sum_{k=0}^\infty \sum_{j \neq i} |(\nabla (\Id - Q_i)Q_j v_k)_i - M(\xi_i) (\nabla Q_j v_k)_i| + |M(\xi_i) (\nabla Q_j v_k)_i| \\
    & \lesssim  \sum_{k=0}^\infty \sum_{j \neq i} r \| \nabla D Q_j v_k\|_{L^\infty(B_i)} + \| \nabla Q_j v_k\|_{L^\infty(B_i)} \\
    & \lesssim \sum_{k=0}^\infty \sum_{j \neq i} \bigg(\frac{r^4}{d_{ij}^4} + \frac{r^3}{d_{ij}^3}\bigg) \|D v_k\|_{L^\infty(B_j)} \lesssim \phi S_3 .
\end{align}
On the other hand, by \eqref{Q_i.M} and $\|\nabla^2 u\|_{\infty} \lesssim 1$
\begin{align}
     |(\nabla u - \nabla Q_i u)_i - M(\xi_i) \nabla u(X_i)| \leq r.
\end{align}
Combining those estimates and using that $\nabla u_N = \mathrm{const}$ in $\B_i$ yields \eqref{u_N-u.xi}.

\medskip 

To prove \eqref{u_N-u.1st.order}, we estimate
\begin{align}
|(u_N - v_1)_i| &\leq  \sum_{k=1}^\infty |(v_{k+1} - v_k)_i| \leq \sum_{k=1}^\infty \sum_{j} |(Q_j v_k)_i| \\
    &\lesssim \sum_{k=1}^\infty \bigg( r \|D v_k\|_{L^\infty(B_i)} + \sum_{j\neq i} \frac{r^3}{d_{ij}^2} \|D v_k\|_{L^\infty(B_j)} \bigg) \\
    &\lesssim  (r + \phi S_2)  \sum_{k=1}^\infty(C \phi_n S_3)^{k} \lesssim  (r + \phi S_2) \phi S_3.
\end{align}
Moreover, we have
\begin{align}
    (v_1)_i = (u)_i - (Q_i u)_i - \sum_{j \neq i} (Q_j u)_i.
\end{align}
By \eqref{Q_i.Z} in combination with $\|\nabla u\|_{W^{1,\infty}}  \lesssim 1$ we have
\begin{align}
     \bigg|\sum_{j \neq i} (Q_j u)_i - \frac \phi N \sum_{j \neq i} \nabla G(X_j - X_i) : S(\xi_j ) D u (X_j)\bigg| 
    &\lesssim r^4 \sum_{j \neq i} \bigg( \frac{1}{d_{ij}^3} + \frac{1}{d_{ij}^2} \bigg) \\
    &\lesssim r \phi S_3  \lesssim r
\end{align}
 Combining these estimates with \eqref{(Q_iw)_i} yields \eqref{u_N-u.1st.order}.

 \medskip 

Finally, for \eqref{u_N-u.xi.1st.order} we estimate analogously as in \eqref{u_N-u.xi.0}, with the sum starting at $k=1$ instead of $k=0$: 
\begin{align}
    |(\nabla( u_N - v_1 +  Q_i v_1))_i|  
    \lesssim  (\phi S_3)^2.
\end{align}
Moreover, we have  by  \eqref{Q_i.M} and \eqref{est:Q_i.pointwise} in combination with $\|\nabla u\|_{L^\infty} \lesssim 1$
\begin{align}
     |(\nabla v_1 - \nabla Q_i v_1)_i - M(\xi_i) (\nabla v_1)_i| &\leq r \|\nabla D v_1\|_{L^\infty(\B_i)} \\
     &\leq r \sum_{j \neq i}  \|\nabla D Q_j u\|_{L^\infty(\B_i)} 
     \lesssim r^4 \sum_{j \neq i} \frac{1}{d_{ij}^4} = N r^4 S_4.
\end{align}
Applying \eqref{Q_i.M} with $w = Q_i u$ and using $Q_i Q_i = Q_i$   in combination with $\|\nabla u\|_{W^{1,\infty}}  \lesssim 1$, also yields
\begin{align}
    |M(\xi_i) (\nabla Q_i u)_i|\lesssim r
\end{align}
Using that $(\nabla v_1)_i = (\nabla u)_i - (Q_i u)_i - \sum_{j \neq i} (Q_j u)_i$
we combine these two estimates to 
\begin{align}
\bigg|(\nabla v_1 - \nabla Q_i v_1)_i - M(\xi_i) \bigg(\nabla u - \sum_{j \neq i} \nabla  Q_j u\bigg)_i\bigg| \lesssim r + r^4 N S_4.
\end{align}
Moreover, by \eqref{Q_i.Z}  in combination with $\|\nabla u\|_{W^{1,\infty}}  \lesssim 1$, we have
\begin{align}
\bigg|\sum_{j \neq i} \nabla (Q_j u)_i - \frac \phi N \sum_{j \neq i} \nabla^2 G(X_j - X_i) : S(\xi_j ) D u (X_j)\bigg| 
    &\lesssim r^4\sum_{j \neq i} \bigg( \frac{1}{d_{ij}^4} + \frac{1}{d_{ij}^3} \bigg)  \lesssim r^4 N S_4.
\end{align}
Collecting the above estimates and using that $\|M(\xi_i)\| \lesssim 1$ yields \eqref{u_N-u.xi.1st.order}.
\end{proof}

\subsection{Proof of Theorem \ref{th:app}} \label{subsec:proof_deterministic1}

In this section, we prove Theorem \ref{th:app}.
We first show that $f_N$ converges to the solution of \eqref{eq:f}--\eqref{u} and that the interparticle distances remain well controlled. \noeqref{eq:f^0} 
As discussed in the introduction, this has been already shown in \cite{Duerinckx23} but  for self-containedness we give the short proof here that relies on the estimates in Proposition \ref{pro:MOR}.

\begin{theorem} \label{th:zero.order}
     For all $T>0$ there exists $N_0 \in \N$ such that for all $N \geq N_0$ and all $t \leq T$
    \begin{align}
        \mathcal W_\infty(f_N(t),f(t)) \leq C (\mathcal W_\infty(f_N^0,f^0) +  \phi\log N) \label{W_infty.1}
    \end{align}
    \begin{align}
        \frac 1 C |X_i^0 - X_j^0|  \leq |X_i - X_j| \leq  C |X_i^0 - X_j^0|.  \label{dmin.control}\\
    \end{align}
    where $C$ depends only on $T$, $f^0$ and $g$.
\end{theorem}
\begin{remark}
    In particular, combining  \eqref{est:S_k}, \eqref{dmin.control} and \eqref{ass:separation} we have uniformly on $[0,T]$, for all $N$ sufficiently large
    \begin{align}
        S_2 &\lesssim 1 \label{S_2.est},\\
        S_3 &\lesssim \log N \label{S_3.est}, \\
        S_4 &\lesssim \dmin^{-1}(0) \lesssim N^{1/3}.\label{S_4.est}
    \end{align}
\end{remark}
\begin{proof}
We first show \eqref{dmin.control}. Indeed, this follows immediately from \eqref{u_N.Lipschitz} and a Gronwall argument as long as the assumptions of Proposition \ref{pro:MOR} are fulfilled. By \eqref{est:S_k} and assumptions \eqref{ass:separation}--\eqref{ass:diluteness}, we have
\begin{align}
    \phi S_3(0) \lesssim \phi \log N \to 0.
\end{align}
Hence, the assumptions of Proposition \ref{pro:MOR} are fulfilled at time $t=0$ for $N$ large enough. Moreover, at later time they are still fulfilled for sufficiently large $N$  once we have shown \eqref{dmin.control}. A standard continuity argument then yields that both \eqref{dmin.control} 
and the assumptions of Proposition \ref{pro:MOR} are fulfilled on $[0,T]$ for $N$ large enough.

    Let $T_0$ be an optimal transport plan for $W_\infty(f_N^0,f^0)$, i.e. $f_N^0 =(T_0)_\# f^0 $ and 
$|T_0(x,\xi) - (x,\xi)| \leq \mathcal W_\infty(f_N^0,f^0) $ for almost all $(x,\xi)$ in $\supp f^0$. Let $Y = (Y_{1}, Y_{2})$ and $Y^N = (Y^N_1,Y^N_2)$ be the flow maps associated to the limit system \eqref{eq:f}--\eqref{u} and the microscopic system \eqref{main}, \eqref{x.dot}--\eqref{xi.dot} respectively, i.e.
\begin{align}
    \partial_t Y(t;s,x,\xi) &= \Big(u(t,Y_1(t;s,x,\xi)),(M(Y_2(t;s,x,\xi))\nabla u(t,Y_1(t;s,x,\xi)))Y_2(t;s,x,\xi)\Big), \\  Y(s;s,x,\xi) &= (x,\xi),
\end{align}
and 
\begin{align}
    \partial_t Y^N(t;s,x,\xi) &= \Big(u_N(t,Y^N_1(t;s,x,\xi)),  \nabla u_N(t,Y_1(t;s,x,\xi))  Y^N_2(t;s,x,\xi)\Big), \\
    Y^N(s;s,x,\xi) &= (x,\xi)
\end{align}
We will only evaluate $Y^N(t;s,x,\xi)$ for $(x,\xi) = (X_i(s),\xi_i(s))$ for some $1 \leq i \leq N$ such that there is no issue with the pointwise evaluation of $u_N$ and $\nabla u_N$ for the latter system. Moreover, we remind that $\nabla u_N = \frac{1}{2} (\curl u_N \, \times)$ at each particle location, which explains why we replaced the skew-symmetric part by the full gradient. 

We define $T_t = Y^N(t,0,\cdot) \circ T_0 \circ Y(0,t,\cdot)$.
We then have that $T_t$ is a transport plan for $(f_N(t),f(t))$ and define
\begin{align}
  \eta(t) :=   f(t,\cdot) - \esssup |T_t(x,\xi) - (x,\xi)|.
\end{align}
Let $(x,\xi) \in \supp f(t)$, let $1 \leq i \leq N$ be such that $(X_i,\xi_i) = T_0(Y_1(0;t,x,\xi),Y_2(0;t,x,\xi))$. Then, abbreviating $(x_s,\xi_s) = Y_1(s;t,x,\xi),Y_2(s;t,x,\xi)$ and using that $D u_N(s,X_i(s)) = 0$, we find as long as the assumptions of Proposition \ref{pro:MOR} are fulfilled
\begin{align}
    |T_t(x,\xi) - (x,\xi)| &=|(X_i(t),\xi_i(t)) - (x,\xi)| \\
    &\leq  |(X_i(0),\xi_i(0)) - (x_0,\xi_0)| \\
    &+      \int_0^t |u_N(s,X_i(s)) - u(s,x_s)| + | \nabla u_N(s,X_i(s)) \xi_i(s)) - M(\xi_s)\nabla u(s,x_s) \xi_s| \dd s \\
    & \leq \eta(0) + \int_0^t \eta(s) + r + \phi (S_2(s) + S_3(s)) \dd s, 
\end{align}
where we used \eqref{u_N-u}, \eqref{u_N-u.xi} as well as $u \in W^{2,\infty}(\R^3), M \in W^{1,\infty}(\S^2)$ in the last estimate. Taking the essential supremum in this inequality and inserting the estimate \eqref{S_2.est}--\eqref{S_3.est} yields
\begin{align}
    \eta(t) \lesssim \eta(0) +  r + \phi \log N.
\end{align}
By construction $\eta(0) = W_\infty(f_N^0,f^0)$ and 
    $W_\infty(f_N(t),f(t)) \leq \eta(t)$. Hence
    \begin{align}
        W_\infty(f_N(t),f(t)) \lesssim W_\infty(f_N^0,f^0) +  r + \phi \log N.
    \end{align}

   To eliminate $r$, we use that the Wasserstein distance between an empirical measure of $N$ particles and a continuous (and bounded) density is bounded below by $N^{-1/d}$ where $d$ is the space dimension (cf. e.g. \cite[Eq. (1.11)]{HoferSchubert23}). Since the dimension of $\R^3 \times \S^2$ is $5$,
    \begin{align} \label{discretisation.error}
        W_\infty(f_N^0,f^0)) \gtrsim N^{-\frac 1 5} \gtrsim \dmin^{3/5} \gtrsim r.
    \end{align} 
This finishes the proof.
% \begin{align}
%     \eta(t) := \sup_{0 \leq s \leq t} f(s,\cdot) - \esssup |T_s(x,\xi) - (x,\xi)|
% \end{align}
% Then,  (cf. \cite[Lemma 3.2]{Hofer&Schubert})
% \begin{align}
%     \eta(t) \leq W_\infty(f_N^0,f^0) 
% \end{align}
\end{proof}

% \begin{proposition}
% For all $1 \leq i,j \leq N$ and all $t \leq T$
%     \begin{align}
%     |X_i^0 - X_j^0|  &\lesssim |X_i^{app} - X_j^{app}| \lesssim |X_i^0 - X_j^0| \label{dmin.control.app}, \\
%     \|u_\phi\|_{W^{2,\infty}} \lesssim 1, \qquad \|u - u_\phi\|_{W^{1,\infty}} &\lesssim \phi, \label{u-u_phi}, \\
%      |X_i(t) - X_i^{app}(t)| &\lesssim  r+\phi,  \label{app.positions.0} \\
%     |\xi_i(t) - \xi_i^{app}(t)| &\lesssim r + \phi \log N, \label{app.orientations.0} \\
%         |X_i(t) - X_i^{app}(t)| &\lesssim  r + \phi (\phi \log \phi + \mathcal W_\infty(f_N^0,f^0)|\log \mathcal W_\infty(f_N^0,f^0)|) , \label{app.positions.1} \\   
%          |X_i(t) - X_i^{app}(t) - (X_j(t) - X_j^{app}(t))| &\lesssim \phi \log N |X_i^0 - X_j^0| \label{relative.distance}
%     \end{align}
% \end{proposition}
\begin{proof}[Proof of Theorem \ref{th:app}]
We already proved \eqref{dmin.control.thm}, see \eqref{dmin.control}. It therefore suffices to prove that, for $t \leq T$,  we have for all $N$ sufficiently large and all $1 \leq i \leq N$ 
\begin{align}
  |X_i^0 - X_j^0|  &\lesssim |X_i^{app} - X_j^{app}| \lesssim |X_i^0 - X_j^0| \label{dmin.control.app},
\end{align}
and
        \begin{align} \label{app.positions.orientations}
        |(X_i,\xi_i)(t) - (X_i^{app}, \xi_i^{app})(t)| \lesssim \phi^{4/3} +  \phi \log N (\phi (|\log \phi| + \log N)  + \mathcal W_\infty(f_N^0,f^0) ). \qquad 
    \end{align}
The statement then follows by considering the transport map $T_t$ that maps $(X_i,\xi_i)(t)$ to $(X_i^{app}, \xi_i^{app})(t)$. 
In the following, let $T >0$ be fixed  and let $N$ be sufficiently large such that the conclusions of Proposition \ref{pro:MOR} and Theorem \ref{th:zero.order} hold on $[0,T]$.
Moreover, let $t \in [0,T]$ and $1\leq i,j \leq N$\\[3mm]
\emph{Step 1: Proof of \eqref{dmin.control.app}.}

Through a Gronwall argument, the estimate \eqref{dmin.control.app}  follows from the Lipschitz continuity of $u_\phi$ provided  by \eqref{regularity.u}.\\[3mm]
\emph{Step 2: Proof that
\begin{align}
     |X_i(t) - X_i^{app}(t)| &\lesssim  r + \phi (\phi \log \phi + \mathcal W_\infty(f_N^0,f^0)|\log \mathcal W_\infty(f_N^0,f^0)|) , \label{app.positions.1}
\end{align}
    }

We first show the weaker estimate
\begin{align}
         |X_i(t) - X_i^{app}(t)| &\lesssim  r+\phi,  \label{app.positions.0}
\end{align}
We estimate
\begin{align*}
    |X_i(t) - X_i^{app}(t)| &\leq \int_0^t |u_N(s,X_i(s)) - u_\phi(s,X_i^{app}(s))| \dd s \\
    &\leq \int_0^t |X_i(s) - X_i^{app}(s)|\|\nabla u\|_\infty +  \|u - u_\phi\|_\infty +  |u_N(s,X_i(s)) - u(s,X_i(s))| \dd s \\
    &\lesssim  \phi + r + \int_0^t |X_i(s) - X_i^{app}(s)| \dd s.
\end{align*}
where we used \eqref{u_N-u}, \eqref{regularity.u}, \eqref{u-u_phi}, and \eqref{S_2.est}. Estimate \eqref{app.positions.0} follows by Gronwall's estimate.

% \medskip 

% For \eqref{app.orientations.0}, we first observe that \eqref{xi.dot} and $D u_N(X_i)  = 0$ implies
% \begin{align} \label{xi.dot.full.gradient}
%    \dot \xi_i(t) = \nabla u(X_i) \xi_i.
% \end{align}
% Moreover, by \eqref{xi^app}, $\|\nabla u\|_{\infty} \lesssim 1$, $\|Z\|\lesssim 1$, $\|M\|\lesssim 1$ and the $-1$-homogeneity of $G$
% \begin{align}
%     |\dot \xi_i^{app} -(M(\xi_i^{app}) \nabla u(X_i^{app})) \xi_i^{app}| \lesssim \frac \phi N \sum_{j\neq i} \frac 1 {d_{ij}^3} \lesssim \phi S_3 \lesssim \phi \log N
% \end{align}
% where we used \eqref{S_3.est} in the last estimate.
% Combining this with \eqref{u_N-u.xi}, using that $\|\nabla^2 u\|_{L^\infty} \lesssim 1$ as well as \eqref{app.positions.0} yields \eqref{app.orientations.0} by a Gronwall argument  analogously to the previous one.

\medskip 

For \eqref{app.positions.1}, we compute with the help of \eqref{u_N-u.1st.order}
\begin{align}
    |u_N(X_i) - u_\phi(X_i)| \leq r + \phi^2 S_2 S_3 + \bigg|u_{diff}(X_i) + \frac \phi N \sum_{j \neq i} \nabla G(X_j - X_i) : S(\xi_j ) D u (X_j)\bigg|, \qquad  \label{app.positions.1.0}
\end{align}
where $u_{diff} = u - u_\phi$ solves \eqref{u_diff}. Hence, recalling the notation $\mV[h]$ from \eqref{sigma} and using the convention $\nabla G(0) = 0$
\begin{align}
    u_{diff}(X_i) + &\frac \phi N \sum_{j \neq i} \nabla G(X_j - X_i) : S(\xi_j ) D u (X_j) = \phi (\nabla G \ast (\mV[f] D u - \mV[f_N] Du))(X_i) \\
    &= \phi (\nabla G \ast (\mV[f - f_N]   Du))(X_i).\label{app.positions.1.1}
\end{align}
% Note that $\|\mV[h]\|_1 \lesssim \|h\|_1$. Hence, by \eqref{u-u_phi}, we have 
% \begin{align} \label{app.positions.1.2}
%     \|\nabla G \ast (\mV[f] D u_{diff})\|_\infty \lesssim \|\mV[f] D u_{diff}\|_{L^{1} \cap L^\infty} \lesssim \phi.
% \end{align}
Let $ T \colon \R^3 \times \S^2 \to \R^3 \times \S^2$ be an optimal transport plan for $\mathcal W_\infty(f_N,f)$, i.e. $f_N =T_\# f $ and 
$|T(x,\xi) - (x,\xi)| \leq \mathcal W_\infty(f_N,f) $ for almost all $(x,\xi)$ in $\supp f$. We denote $T = (T_1,T_2)$, where $T_1$ and $T_2$ take values in $\R^3$ and $\S^2$, respectively. Then
\begin{align}
    &(\nabla G \ast (\mV[f - f_N]   Du))(X_i) \\
    &= \int_{\R^n \times \S^2}\Big(\nabla G(X_i - x) S(\xi) D u (x) -  \nabla G(X_i - T_1(x,\xi)) S(T_2(x,\xi)) Du(T_1(x,\xi))\Big) \dd f(x,\xi) 
\end{align}
Using that the map $S \colon \S^2 \to \mathcal L(\sym_0(3);\sym_0(3))$ is smooth, \eqref{regularity.u} and the decay of $G$, we find
\begin{align}
    &|(\nabla G \ast (\mV[f - f_N]   Du))(X_i)| \\
    &\lesssim \mathcal W_\infty(f_N,f) + \mathcal W_\infty(f_N,f) \int_{(\R^3 \setminus B_{2 \mathcal W_\infty(f_N,f)}(X_i)) \times \S^2} \left( \frac{1}{|X_i - T_1(x,\xi)|^3} +  \frac{1}{|X_i - x|^3} \dd f(x,\xi) \right) \\
    &\hspace{5cm}+  \int_{B_{2 \mathcal W_\infty(f_N,f)}(X_i) \times \S^2} \left( \frac{1}{|X_i - T_1(x,\xi)|^2} +  \frac{1}{|X_i - x|^2} \right) \dd f(x,\xi) .
\end{align}
We note that for $(x,\xi) \in (\R^3 \setminus B_{2 \mathcal W_\infty(f_N,f)}(X_i)) \times \S^2$  we have $\frac{1}{|X_i - T_1(x,\xi)|^3} \lesssim \frac{1}{|X_i - x|^3}$ to estimate
\begin{align}
     \int_{(\R^3 \setminus B_{2 \mathcal W_\infty(f_N,f)}(X_i)) \times \S^2} \left( \frac{1}{|X_i - T_1(x,\xi)|^3} +  \frac{1}{|X_i - x|^3} \right) \dd f(x,\xi) \lesssim |\log{\mathcal W_\infty(f_N,f)}|
\end{align}
On the other hand, we observe that for  $(x,\xi) \in B_{2 \mathcal W_\infty(f_N,f)}(X_i) \times \S^2$ we have $|T_1(x,\xi) - X_i| \leq 3 \mathcal W_\infty(f_N,f)$. Appealing to \cite[Lemma 3.1]{Hofer&Schubert} this implies\footnote{Note that we can apply this result for the induced spatial measures $\rho_N = \frac 1 N \sum_i \delta_{X_i}$,  $\rho = \int_{\S^2} f \dd \xi$, which satisfy $\mathcal W_\infty(\rho_N,\rho) \leq \mathcal W_\infty(f_N,f)$.}
\begin{align}
    \int_{B_{2 \mathcal W_\infty(f_N,f)}(X_i) \times \S^2} \frac{1}{|X_i - T_1(x,\xi)|^2} \dd f(x,\xi) &\leq \sum_{\substack{j \neq i\\ |X_i -X_j| \leq 3 \mathcal W_\infty(f_N,f)}} \frac{1}{|X_i - X_j|^2} \\
    &\lesssim \frac{W_\infty(f_N,f)}{N^{2/3}\dmin^2} \lesssim  W_\infty(f_N,f).
\end{align}
Thus, we arrive at
\begin{align} \label{app.positions.1.3}
    |(\nabla G \ast (\mV[f - f_N]   Du))(X_i)| \lesssim  \mathcal W_\infty(f_N,f) (1 + |\log{\mathcal W_\infty(f_N,f)}|).
\end{align}
% where we used \eqref{dmin.control} and \eqref{ass:separation} in the last estimate.
Combining  \eqref{app.positions.1.0}, \eqref{app.positions.1.1}, and \eqref{app.positions.1.3} yields
\begin{align}
    |u_N(X_i) - u_\phi(X_i)| &\lesssim r + \phi^2 (1 + S_2 S_3) + \phi W_\infty(f_N,f) (1 + |\log{\mathcal W_\infty(f_N,f)}|) \\
    &\lesssim  r + \phi (\phi \log N + \mathcal W_\infty(f_N^0,f^0)|\log \mathcal W_\infty(f_N^0,f^0)|)
\end{align}
where we used the estimates \eqref{W_infty.1} \eqref{S_2.est} and \eqref{S_3.est} in the last estimate as well as monotonicity of $z \mapsto z (1 + |\log z|)$ and $\phi + W_\infty(f_N^0,f^0)) \ll 1$.

Appealing to \eqref{X^app}, \eqref{app.positions.1} follows from the above estimate, the bound $\|\nabla u_\phi \|_\infty \lesssim 1$ from \eqref{regularity.u} and a Gronwall argument.\\[3mm]
\emph{Step 3: Proof that
\begin{align}
             |X_i(t) - X_i^{app}(t) - (X_j(t) - X_j^{app}(t))| &\lesssim r + \phi \log N |X_i^0 - X_j^0|. \label{relative.distance}
\end{align}
}
    We use that for any $h\in W^{2,\infty}(\R^d)$ and $a_i,b_i \in \R^d$, $i=1,2$, 
    \begin{align}
        & |h(a_1) - h(a_2) - (h(b_1) - h(b_2))| \\
      &  = \bigg| (a_1 - a_2 - (b_1-b_2)) \cdot \int_0^1 \na h(ta_1 + (1-t)a_2)\dd t \\
       &   + (b_1 - b_2) \cdot \int_0^1  \Big( \na h(ta_1 + (1-t)a_2)  - \na h(tb_1 + (1-t)b_2) \Big) \dd t  \bigg|
       \\
        % & \qquad \leq \|\nabla h\|_{W^{1,\infty}(\R^d)} \Big(|a_1 - a_2 - (b_1-b_2)|(1+|b_1-b_2|) + |a_2-b_2||b_1-b_2|\Big) \\
        &  \leq \|\nabla h\|_{W^{1,\infty}(\R^d)} \Big(|a_1 - a_2 - (b_1-b_2)| + (|a_1 -b_1| + |a_2-b_2|)|b_1-b_2|\Big).
    \end{align}
    Moreover, for $\bar h \in W^{1,\infty}(\R^d)$, we have
    \begin{align}
        |h(a_1) - h(a_2) - (\bar h(a_1) - \bar h(a_2))| \leq \| \nabla(h - \bar h)\|_{L^{\infty}(\R^d)}|a_1-a_2|
    \end{align}
    so that combination of these estimates yields
    \begin{align*}
        |\bar h(a_1) - \bar h(a_2) - (\bar h(b_1) -\bar h(b_2))| &\leq  
         \| \nabla(h - \bar h)\|_{L^{\infty}(\R^d)}(|a_1-a_2| + |b_1-b_2|) \\
         &\hspace{-2cm} +\|\nabla h\|_{W^{1,\infty}(\R^d)}\Big(|a_1 - a_2 - (b_1-b_2)| + (|a_1 -b_1| + |a_2-b_2|)|b_1-b_2|\Big) .
    \end{align*}
    Hence, applying this with $h = u$ and $\bar h = u_\phi$ and using \eqref{regularity.u}--\eqref{u-u_phi}, we find
    \begin{align} \label{relative.distance.0}
    \begin{aligned}
        &|X_i(t) - X_i^{app}(t) - (X_j(t) - X_j^{app}(t))| \\
        &\leq \int_{0}^t \left|u_N(s,X_i(s)) - u_N(s,X_j(s)) - \left(u_\phi(s,X_i^{app}(s)) - u_\phi(s,X_j^{app}(s))\right) \right| \dd s \\
        &\leq \int_{0}^t\Big|u_N(s,X_i(s)) - u_N(s,X_j(s)) - \left(u_\phi(s,X_i(s)) - u_\phi(s,X_j(s))\right) \Big| \dd s \\
        & \quad + C \int_{0}^t \Big|X_i(s) - X_j(s) - (X_i^{app}(s)-X_j^{app}(s))%\Big|(1+|X_i^{app}(s)-X_j^{app}(s)| 
        \\
        & \qquad \qquad + \Big(|X_i(s)-X_i^{app}(s)| + |X_j(s)-X_j^{app}(s)|\Big)|X_i^{app}(s)-X_j^{app}(s)| \\
        & \hspace{5cm} + \phi \Big(|X_i(s)-X_j(s)| + |X_i^{app}(s)-X_j^{app}(s)|\Big)\dd s .
    \end{aligned}
    \end{align}
    Combining \eqref{u-u_phi},  \eqref{nabla.u-u_N} and \eqref{S_3.est}, we have
    \begin{align}
        \Big|u_N(s,X_i(s)) - u_N(s,X_j(s)) - \Big(u_\phi(s,X_i(s)) - u_\phi(s,X_j(s))\Big) \Big| \lesssim r + \phi \log N |X_i - X_j|.
    \end{align}
    Inserting this in \eqref{relative.distance.0} and using \eqref{app.positions.0},  \eqref{dmin.control} and \eqref{dmin.control.app} yields
    \begin{align}
        &\Big|X_i(t) - X_i^{app}(t) - (X_j^{app}(t) - X_j(t))\Big| \\
        &\lesssim \int_0^t \Big|X_i(s) - X_i^{app}(s) - (X_j(s) - X_j^{app}(s))\Big| \dd s + r + \phi \log N |X_i^0 - X_j^0|.
    \end{align}
    Gronwall's inequality implies \eqref{relative.distance}.
    \\[3mm]
\emph{Step 4: Proof that
\begin{align} \label{app.orientations.1}
    |\xi_i(t) - \xi_i^{app}(t)| \lesssim \phi (\phi \log \phi + \mathcal W_\infty(f_N^0,f^0)|\log \mathcal W_\infty(f_N^0,f^0)|) +(\phi \log N)^2 + r 
    \qquad 
\end{align}
}

We first observe that \eqref{xi.dot} and $D u_N(X_i)  = 0$ imply
\begin{align} \label{xi.dot.full.gradient}
   \dot \xi_i(t) = \nabla u_N(X_i) \xi_i.
   \end{align}
    Using \eqref{xi.dot.full.gradient}, \eqref{xi^app} and \eqref{u_N-u.xi.1st.order}
    yields
    \begin{align}
        &|\xi_i(t) - \xi_i^{app}(t)|\\ &\leq \int_0^t  \bigg| M(\xi_i ) \bigg(\nabla u(X_i ) + \frac{\phi}{N}\sum_{j \neq i} \nabla^2 G(X_i - X_j )  S(\xi_j ) D u (X_j) ) \bigg)  \xi_i  \\
        & - M(\xi_i^{app} ) \bigg(\nabla u(X_i^{app} ) + \frac{\phi}{N}\sum_{j \neq i} \nabla^2 G(X_i^{app} - X_j^{app} )  S(\xi_j^{app} ) D u (X_j^{app} ) \bigg)  \xi_i^{app} \bigg| \dd s \\
        &+ (\phi S_3)^2 + r + r^4 N S_4.
    \end{align}
    Using that $M$ and $S$ are smooth functions on $\S^2$, the decay of $G$, \eqref{regularity.u}, \eqref{S_3.est} and \eqref{dmin.control.app} yields
    \begin{align}
        &\bigg| M(\xi_i ) \bigg(\nabla u(X_i ) + \frac{\phi}{N}\sum_{j \neq i} \nabla^2 G(X_i - X_j )  S(\xi_j ) D u (X_j) ) \bigg)  \xi_i  \\
        & - M(\xi_i^{app} ) \bigg(\nabla u(X_i^{app} ) + \frac{\phi}{N}\sum_{j \neq i} \nabla^2 G(X_i^{app} - X_j^{app} )  S(\xi_j^{app} ) D u (X_j^{app} ) \bigg)  \xi_i^{app} \bigg| \\
        & \lesssim (1 + \phi \log N) (\sup_j |X_j - X_j^{app}| + |\xi_j - \xi_j^{app}|)
        + \frac{\phi}{N}\sum_{j \neq i} \left|\nabla^2 G(X_i - X_j ) - \nabla^2 G(X_i^{app} - X_j^{app} ) \right|.
    \end{align}
    Moreover, by \eqref{relative.distance}, the decay of $G$, \eqref{dmin.control} and \eqref{dmin.control.app}, we have
    \begin{align}
        &\frac{\phi}{N}\sum_{j \neq i} \left|\nabla^2 G(X_i - X_j ) - \nabla^2 G(X_i^{app} - X_j^{app} ) \right| \\
        &\lesssim \frac{\phi}{N}\sum_{j \neq i} |X_i(t) - X_i^{app}(t) - (X_j(t) - X_j^{app}(t))| \bigg( \frac { 1 }{ |X_i - X_j|^4 } + \frac { 1 }{ |X_i^{app} - X_j^{app}|^4} \bigg) \\
        & \lesssim S_3 \phi^2 \log N + r \phi S_4 \lesssim (\phi \log N)^2 + r^4 N S^4,
    \end{align}
    where we used \eqref{S_3.est} in the last estimate.
    We also note that \eqref{S_4.est} and \eqref{non.overlapping.2}  imply
    \begin{align}
        r^4 N S_4 \lesssim r^4 \dmin^{-4}  \lesssim \phi^{4/3}.
    \end{align}
    Collecting the previous inequalities and combining with \eqref{app.positions.1} as well as \eqref{ass:diluteness} yields
    \begin{align}
        \sup_i |\xi_i(t) - \xi_i^{app}(t)| &\lesssim \int_0^t \sup_i |\xi_i(s) - \xi_i^{app}(s)|\dd s \\
        &+  \phi (\phi \log \phi + \mathcal W_\infty(f_N^0,f^0)|\log \mathcal W_\infty(f_N^0,f^0)|) +(\phi \log N)^2 + r + \phi^{4/3}.
    \end{align}
    Gronwall's inequality implies
    \begin{align}
        |\xi_i(t) - \xi_i^{app}(t)| \lesssim \phi (\phi \log \phi + \mathcal W_\infty(f_N^0,f^0)|\log \mathcal W_\infty(f_N^0,f^0)|) +(\phi \log N)^2 + r + \phi^{4/3}
    \end{align}\\[3mm]
   \emph{Step 5: Conclusion.}
By \eqref{discretisation.error} we have $ r \lesssim \mathcal W_\infty(f_N^0,f^0)$ and $|\log \mathcal W_\infty(f_N^0,f^0)| \lesssim \log N$.
    Combining this \eqref{app.orientations.1} with \eqref{app.positions.1} yields \eqref{app.positions.orientations}.
\end{proof}

\subsection{Regularization procedure} \label{subsec:proof_deterministic2}

In the light of Theorem \ref{th:app} and \eqref{xi^app}, to determine a $O(\phi)$ approximation of $f_N$ requires the analysis of the singular mean-field interaction \eqref{double.sum.original}.  More generally, for any family $(X_i, \xi_i)_{1 \le i \le  N}$ satisfying \eqref{ass:separation} and any continuous $\Psi_1, \Psi_2 \in C(\R^3 \times \S^2 ; \R_0^{3\times 3})$, we want to study the asymptotics of 
\begin{equation} \label{def:IN}
I^N := I^N[\Psi_1,\Psi_2] :=  \frac 1 {N^2} \sum_i \sum_{j \neq i} \Psi_1(X_i,\xi_i) :  \big( \nabla^2 G(X_i - X_j)  \Psi_2(X_j,\xi_j)\big)
\end{equation}
We will use an idea introduced in \cite{Serfaty15} in the context of Coulomb gases: we will reformulate this quantity in terms of appropriate continuous quadratic energy functionals. In our context, this idea 
%Our second result compares \dgv{such mean field terms} to integral expressions involving solutions of some Stokes problems: namely, for %$\Theta^\eta = \Theta^\eta[A]$  defined in \eqref{def_Psi_eta}, and for any given function $\Psi \in C(\R^3 \times \S^2; %\R_0^{3\times3})$, \dgv{we consider $h^{N,\eta}_i[\Psi, \mathcal{X}]$ the solution of 
%\begin{equation} \label{h_i}
%-\Delta h^{N,\eta}_i[\Psi, \mathcal{X}] + \na p^{N,\eta}_i[\Psi, \mathcal{X}] = \div\Big( \frac{1}{N} \Theta^\eta[\Psi(X_i,\xi_i)]%(\cdot - X_i) \Big), \quad \div h^{N,\eta}_i[\Psi, \mathcal{X}] = 0 \quad \text{in } \R^3
%\end{equation}
%We further define 
%$$h^{N,\eta}[\Psi,\mathcal{X}] = \sum_i h^{N,\eta}_i[\Psi,\mathcal{X}]$$}
comes with the remark that for any matrix $A \in \R_0^{3 \times 3}$, the vector-valued function 
$$ h[A] = A\na \cdot G = (A_{lm} \pa_m G_{kl} )_{1 \le k \le 3} $$
satisfies 
\begin{equation} \label{eq_hS}
-\Delta h[A] + \na p[A] = (A \na) \delta_0, \quad \div h[A] = 0 \quad \text{in } \: \R^3
\end{equation}
Introducing, for any $\Psi \in C(\R^3 \times \S^2 ; \R^{3\times 3})$, the function 
$$h_i^N[\Psi] = \frac{1}{N}h[\Psi(X_i,\xi_i)](x-X_i), \quad h^N[\Psi] = \frac 1 N \sum_i h_i^N[\Psi]   $$
we have formally 
\begin{equation} \label{formal_comp}
\begin{aligned}
I^N  & \approx \frac{1}{N^2}
\sum_{i \neq j}  \int_{\R^3}  \Psi_1(X_i,\xi_i) :  \big( \na^2 G(x-X_j)  \Psi_2(X_j,\xi_j) \big) \delta_{X_i}(\d x) \\
&  \approx  - \frac{1}{N}
\sum_{i \neq j}  \int_{\R^3}  h^N_j[\Psi_2] \cdot  (\Psi_1(X_i,\xi_i) \na)  \delta_{X_i}(\d x)   \\
& \approx  - \sum_{i \neq j}  \int_{\R^3} 
 h^N_j[\Psi_2] \cdot  
(-\Delta h^N_i[\Psi_1] + \na p_i^N[\Psi_1]) \dd x  \approx - \int_{\R^3} \sum_{i \neq j} \na h_i^N[\Psi_1] : \na h_j^N[\Psi_2] \\ 
& \approx - \int_{\R^3} \na h^N[\Psi_1] \cdot \na h^N[\Psi_2] + \sum_i \int_{\R^3} \na h_i^N[\Psi_1] \cdot \na h_i^N[\Psi_2] 
\end{aligned}
\end{equation}
 However, this series of formal calculations is not rigorous, as functions $h^N_i[\Psi]$ do not belong to $\dot{H}^1(\R^3)$. We must therefore go through a regularization process, replacing $h[A]$, (resp. $h^N_i[\Psi]$, $h^N[\Psi]$) with some $h^\eta[A]$ (resp. $h^{N,\eta}_i[\Psi]$, $h^{N,\eta}[\Psi]$) where $\eta$ is the regularization parameter. Typically, we wish to replace the Dirac mass at the r.h.s. of \eqref{eq_hS} by some function with typical amplitude $O(\eta^{-3})$ and typical lengthscale $\eta$. Many different choices  ensure convergence of the regularized quadratic functional $I^{N,\eta}$  to $I^N$ as $\eta \rightarrow 0$ for fixed $N$. But we will show more: for some well-chosen regularization, in the regime where $N$ goes to infinity, $\eta$ goes to zero and $\eta \lesssim N^{-\frac13}$, we have $I^{N,\eta} = I^N$.  This kind of result already appeared  in \cite{Gerard-VaretHillairet19} when looking for $O(\phi^2)$ correction to the effective viscosity, and we follow the choice of regularization made there. Namely, we introduce $(h^\eta[A], p^\eta[A])$ of $(h[A],p[A])$, defined by: for all $A \in \R_0^{3 \times 3}$, 
 \begin{align*}
    h^\eta[A] = h[A], \quad p^\eta[A] = p[A] &\qquad \text{in } \R^3 \setminus B(0,\eta), \\
     -\Delta h^\eta[A] + \na p^\eta[A] = 0, \quad \div h^\eta[A] = 0   &\qquad\text{in }  B(0,\eta) \\ \quad h^\eta[A]\vert_{\pa B(0,\eta)^+} = h[A]\vert_{\pa B(0,\eta)^-} &
 \end{align*}
 Note that, by homogeneity, 
\begin{align}
    h^\eta[A](x) = \frac{1}{\eta^2} h^1[A](x/\eta) \label{homogeneity.h}
\end{align}
 Moreover, 
 $$ -\Delta  h^\eta[A] + \na p^\eta[A] = S^\eta[A]  $$
where $S^\eta[A]$ is the measure supported on the sphere $\pa B(0,\eta)$ defined by 
$$S^\eta[A] = - \big[2D(h^\eta[A])n - p^\eta[A] n\big]\vert_{\pa B(0,\eta)} s^\eta $$
with $n$ the unit normal vector pointing outside $B(0,\eta)$ and $s^\eta$  the surface measure on the sphere and where $[F]\vert_{\pa B(0,\eta)} := F\vert_{\pa B(0,\eta)^+} - F\vert_{\pa B(0,\eta)^-}$ denotes the jump over the surface $\pa B(0,\eta)$. We now state  properties of this regularization, extending those given in  \cite[Lemma 3.5]{Gerard-VaretHillairet19} to the case where the matrix $A$ is not necessarily  symmetric. 
\begin{lemma} \label{lemma_regul}
For all $A \in \R_0^{3 \times 3}$, for all $\eta > 0$, $S^\eta[A] = \div \Theta^\eta[A]$ in $\R^3$, where 
\begin{equation} \label{def_Psi_eta}
\begin{aligned}
    \Theta^\eta[A](x) &= \frac{3}{\pi \eta^5} \left(Ax \otimes x + x \otimes Ax - \frac{5}{2}|x|^2 A + \frac{5}{4} \eta^2 A \right) \\ & - 2 D(h^\eta[A])(x) + p^\eta[A](x) \textrm{Id}, \quad x \in B(0,\eta) \\
   \Theta^\eta[A](x) &= 0, \quad x \in B(0,\eta)^c 
\end{aligned}
\end{equation}
Moreover, $\Theta^\eta[A] \rightarrow A\delta_0$ in the sense of distributions as $\eta \rightarrow 0$, so that $S^\eta[A] \rightarrow (A\na)\delta_0$
\end{lemma}
\begin{proof}
From the explicit expression of the Stokes fundamental solution $(G,Q)$, we deduce that $h[A] = (A\na) \cdot G$ and $p[A] = (A\na) \cdot Q$ are given by 
\begin{align*}
h[A] & = -\frac{3}{8\pi} (Ax \cdot x) \frac{x}{|x|^5} - \frac{1}{8\pi} \frac{(A - A^T)x}{|x|^3} \\
p[A]& =-\frac{3}{4\pi} \frac{Ax \cdot x}{|x|^5}
\end{align*}
From this expression, we deduce that 
$$ 2 D(h[A]) = - \frac{3}{4\pi}\frac{Ax \otimes x + x \otimes Ax}{|x|^5} + \frac{15}{4\pi}\frac{(Ax \cdot x)x \otimes x}{|x|^7} - \frac{3}{4\pi} \frac{Ax \cdot x}{|x|^5} $$
so that 
\begin{align*}
\big( 2 D(h[A])n - p[A]n\big) \vert_{\pa B(0,\eta)^+} = \frac{3}{4\pi \eta^3} \big(4(An \cdot n)n - An \big) 
\end{align*}
and 
$$S^\eta[A] = \left(- \frac{3}{4\pi \eta^3} \big(4(An \cdot n)n - An \big) + \big(2D(h^\eta[A)n - p^\eta[A]n\big)\vert_{\pa B(0,\eta)^-} \right) s^\eta $$
On the other hand, a direct calculation using \eqref{def_Psi_eta}
 shows that $\div \Theta^\eta = 0$ outside $\pa B(0,\eta)$, and that 
 \begin{align*}
 [\Theta^\eta[A]n]\vert_{\pa B(0,\eta)} s^\eta = - \Theta^\eta[A]n\vert_{\pa B(0,\eta)^-} s^\eta = S^\eta[A]   
 \end{align*}
 This proves the first part of the lemma. The second part can be taken mutatis mutandis from the proof of \cite[Lemma 3.5]{Gerard-VaretHillairet19} . 
\end{proof}
We then introduce approximations of $h_i^N[\Psi]$ and $h^N[\Psi]$, namely
\begin{equation} \label{def:hNeta}
h_i^{N,\eta}[\Psi] = \frac{1}{N}h^\eta[\Psi(X_i,\xi_i)](x-X_i), \quad h^{N,\eta}[\Psi] = \frac 1 N \sum_i h_i^{N,\eta}[\Psi]  
\end{equation}
We show in the following proposition that these approximations allow to make rigorous the calculations in \eqref{formal_comp}.
\begin{proposition} \label{lim_eta_Nfixed}
$$ I^N = \lim_{\eta \rightarrow 0} - \int_{\R^3} \na h^{N,\eta}[\Psi_1] \cdot \na h^{N,\eta}[\Psi_2] + \sum_i \int_{\R^3} \na h_i^{N,\eta}[\Psi_1] \cdot \na h_i^{N,\eta}[\Psi_2] $$    
\end{proposition}
\begin{proof}
    We start from the right-hand side, that we express as 
\begin{align*}
&    - \int_{\R^3} \na h^{N,\eta}[\Psi_1] \cdot \na h^{N,\eta}[\Psi_2] + \sum_i \int_{\R^3} \na h_i^{N,\eta}[\Psi_1] \cdot \na h_i^{N,\eta}[\Psi_2]  \\
 &= -  \sum_{i \neq j} \int_{\R^3}\na h_i^{N,\eta}[\Psi_1] :  \na h_j^{N,\eta}[\Psi_2] = - \sum_{i \neq j}  \int_{\R^3}\big(-\Delta  h_i^{N,\eta}[\Psi_1] + \na p_i^{N,\eta}[\Psi_1] \big) :  h_j^{N,\eta}[\Psi_2]  \\
 &= - \frac{1}{N} \sum_{i \neq j}  \int_{\R^3}S^\eta[\Psi_1(X_i,\xi_i)](\cdot - X_i) :   h_j^{N,\eta}[\Psi_2] \\
&  \xrightarrow[\eta \rightarrow 0]{}  - \frac{1}{N} \sum_{i \neq j}
\int_{\R^3}\Big( (\Psi_1(X_i,\xi_i)\na) \delta_{X_i} \Big) :  h_j^{N}[\Psi_2] 
\end{align*}
where the last line comes from Lemma \ref{lemma_regul} and from the fact that for $\eta$ obeying \eqref{eta_N}, $h_j^{N,\eta}[\Psi_2] = h^{N}_j[\Psi_2] = \frac{1}{N}(\Psi_2(X_j,\xi_j)\na) \cdot G(\cdot - X_j)$ is smooth in the vicinity of $B(X_i,\eta)$ for all $i \neq j$. We end up with 
    \begin{align*}
&    - \int_{\R^3} \na h^{N,\eta}[\Psi_1] \cdot \na h^{N,\eta}[\Psi_2] + \sum_i \int_{\R^3} \na h_i^{N,\eta}[\Psi_1] \cdot \na h_i^{N,\eta}[\Psi_2] \\
&  \xrightarrow[\eta \rightarrow 0]{}  \frac{1}{N^2}  \sum_{i \neq j}
 (\Psi_1(X_i,\xi_i)\na) \cdot (\Psi_2(X_j,\xi_j)\na) \cdot G(X_i-X_j) = I^N.
\end{align*}
\end{proof}

We are now ready to state the main theorem of this subsection: 
\begin{theorem} \label{theo2}
	For a given sequence of particle configurations $(X_i, \xi_i)_{1 \le i \le N}$ satisfying \eqref{ass:separation}, let $\eta$ satisfy
	\begin{align} \label{eta_N}
	\eta \le \frac{c}{4} N^{-\frac 1 3}. 
	\end{align}
	Then, for any $\Psi_1, \Psi_2 \in C(\R^3 \times \S^2; \R_0^{3\times3})$, we have 
	\begin{align}
	& \frac 1 {N^2} \sum_i \sum_{j \neq i} \Psi_1(X_i,\xi_i) :  \big( \nabla^2 G(X_i - X_j)  \Psi_2(X_j,\xi_j)\big)  \\
	 & \hspace{2cm} = -  \int \nabla h^{N,\eta}[\Psi_1] :  \nabla h^{N,\eta}[\Psi_2] \dd x +   \sum_i \int \nabla h^{N,\eta}_i[\Psi_1] :  \nabla h^{N,\eta}_i[\Psi_2] \dd x .
	\end{align}
\end{theorem} 
Theorem \ref{theo2} is a direct consequence of the previous proposition and the following one.
\begin{proposition}
    The quantity 
\begin{equation} \label{def:INeta}
I^{N,\eta} =-  \int_{\R^3} \na h^{N,\eta}[\Psi_1] : \na h^{N,\eta}[\Psi_2] + \sum_i \int_{\R^3} \na h_i^{N,\eta}[\Psi_1] : \na h_i^{N,\eta}[\Psi_2] 
\end{equation}
is constant for $\eta \le \frac{c}{4}N^{-1/3}$, where $c$ is the constant in \eqref{ass:separation}. 
\end{proposition}
\begin{proof}
   Let $\alpha < \eta \le \frac{c}{4}N^{-1/3}$. We wish to prove that $I^{N,\eta} = I^{N,\alpha}$, which is equivalent to
   \begin{align*}
       \sum_{i \neq j} \int_{\R^3} \na h^{N,\eta}_i[\Psi_1] : \na h^{N,\eta}_j[\Psi_2] =  \sum_{i \neq j} \int_{\R^3} \na h^{N,\alpha}_i[\Psi_1] \cdot \na h^{N,\alpha}_j[\Psi_2]
   \end{align*}
   This identity can be reformulated as 
\begin{align*} & \sum_{i \neq j}    \int_{\R^3} \big(\na h^{N,\eta}_i[\Psi_1] - \na h^{N,\alpha}_i[\Psi_1] \big) : \na h^{N,\eta}_j[\Psi_2] \\
+ & \sum_{i \neq j}   \int_{\R^3}  \na h^{N,\alpha}_i[\Psi_1]  : \big( \na h^{N,\eta}_j[\Psi_2] - \na h^{N,\alpha}_j[\Psi_2] \big)  = 0 .
   \end{align*}
Now, by integration by parts 
 \begin{align*}
& \int_{\R^3} \big(\na h^{N,\eta}_i[\Psi_1] - \na h^{N,\alpha}_i[\Psi_1] \big) : \na h^{N,\eta}_j[\Psi_2] \\
= &  - \int_{\R^3} 
\big(h^{N,\eta}_i[\Psi_1] - h^{N,\alpha}_i[\Psi_1] \big) : 
dS^\eta[\Psi_2(X_j,\xi_j)](\cdot - X_j) = 0
\end{align*}
as $h^{N,\eta}_i[\Psi_1] = h^{N,\alpha}_i[\Psi_1] = h^N_i[\Psi_1]$ on the support $\pa B(X_j, \eta)$ of $S^\eta[\Psi_2(X_j,\xi_j)](\cdot - X_j)$. Similarly, 
\begin{align*}
    \int_{\R^3}  \na h^{N,\alpha}_i[\Psi_1] ) : \big( \na h^{N,\eta}_j[\Psi_2] - \na h^{N,\alpha}_j[\Psi_2] \big) = 0. 
\end{align*}
This concludes the proof of the proposition. 
\end{proof}

\begin{corollary} \label{coro:bound_IN}
$I^N = I^N[\Psi_1,\Psi_2]$ is bounded uniformly in $N$. More precisely, it is a bilinear form obeying the estimate 
\begin{align*} 
|I^N[\Psi_1,\Psi_2]| & \le C \left( \frac{\sharp \{i, \Psi_1(X_i,\xi_i) \neq 0\}}{N} \frac{\sharp \{i, \Psi_2(X_i,\xi_i) \neq 0\}}{N}\right)^{1/2} 
\|\Psi_1\|_{\infty} \|\Psi_2\|_{\infty} 
\end{align*}
where $C$ depends only on the constant $c$ in \eqref{ass:separation} 
\end{corollary}
\begin{proof}
By Theorem \ref{theo2}, it is enough to show that for $\eta_N = \frac{c}{4}N^{-1/3}$, both quantities 
$$ \int_{\R^3} \na h^{N,\eta_N}[\Psi_1] : \na h^{N,\eta_N}[\Psi_2]   \ \:\text{and} \: \sum_{i=1}^N \int \na h_i^{N,\eta_N}[\Psi_1] : \na h_i^{N,\eta_N}[\Psi_2] $$
 obey the estimate of the corollary
This follows from Cauchy-Schwarz inequality and from the simple energy estimates (based on the homogeneity \eqref{homogeneity.h}  and  the observation that the functions $\Theta^{\eta_N}[\Psi(X_i,\xi_i)](\cdot - X_i)$ have disjoint support due to \eqref{ass:separation}): 
\begin{align*}
    \|\na h^{N,\eta_N}[\Psi]\|_{L^2} & \le \frac{1}{N}\|\sum_i  \Theta^{\eta_N}[\Psi(X_i,\xi_i)](\cdot - X_i)\|_{L^2} \\
    & \le \frac{\big(\sharp \{i, \Psi(X_i,\xi_i) \neq 0\}\big)^{1/2}}{N \eta_N^{3/2}} \sup_i \|\Theta^1[\Psi((X_i,\xi_i))]\|_{L^2} \\
    & \le C \left( \frac{\sharp \{i, \Psi(X_i,\xi_i) \neq 0\}}{N} \right)^{1/2}\|\Psi\|_\infty 
\end{align*}
as well as 
\begin{align*}
\|\na h^{N,\eta_N}_i[\Psi(X_i,\xi_i)]\|_{L^2}  & = 1_{\Psi(X_i,\xi_i) \neq 0}  \frac{1}{N \eta_N^{3/2}}  \|\na h^1[\Psi(X_i,\xi_i)]\|_{L^2} \\
& \le 1_{\Psi(X_i,\xi_i) \neq 0} \frac{C'}{N^{1/2}} \|\Psi\|_\infty 
\end{align*}
\end{proof}
It remains to understand the possible accumulation points of the quantity $I^N$  in \eqref{def:IN}, in the special case where $(X_i,\xi_i) = (X_i^{app}(t),\xi_i^{app}(t))$ for some arbitrary time $t$. We wish to take advantage of Theorem \ref{theo2}, and consider the convergence of $I^{N,\eta}$ in the regime $\eta \approx N^{-1/3}$, $N \rightarrow +\infty$. A keypoint is to understand the limit of $\int_{\R^3} |\na h^{N,\eta}[\Psi]|^2$ for any continuous $\Psi = \Psi(x,\xi)$. Note that  $h^{N,\eta}[\Psi]$ solves the Stokes equation
$$-\Delta  h^{N,\eta}[\Psi] + \na p^{N,\eta}[\Psi] \big) =  \frac{1}{N} \sum_{i}  S^\eta[\Psi_1(X^{app}_i,\xi^{app}_i)](\cdot - X^{app}_i), $$
with a source term that varies at typical lenghtscale $N^{-1/3} \ll 1$. This oscillation will be transferred to the solution $h^{N,\eta}$, so that understanding the limit of the energy, quadratic in $h^{N,\eta}$, is a kind of  homogenization problem. To go further in the analysis of this problem, we will therefore consider a typical framework in homogenization theory, namely random stationarity (here of the particles distribution). One difficulty of our setting, to be handled in the next section, is that this random stationarity, even if assumed initially, will not propagate at positive times.

\section{Proof of the results for stationary ergodic initial data} \label{section_stationary}

\subsection{Random stationary model for the particles distribution}

We shall consider in this section the case of initial particle positions distributed according to a stationary ergodic point process, as briefly described before the statement of Theorem~\ref{theo3}. We first recall a few general notions around point processes. Given a complete separable metric space $X$,  we denote $\text{Point}_X$ the set  of all boundedly finite subsets $\Lambda$ of $X$, which means that $\sharp \Lambda \cap B < \infty$ for all Borel bounded subset $B$ of $X$.  Identifying each $\Lambda \in \text{Point}_X$ with the measure $\mu_{\Lambda} = \sum_{x \in \Lambda} \delta_x$, we can identify $\text{Point}_X$ with a subset of boundedly finite Borel measures on $X$, which can be equipped with the metrizable topology of weak-* convergence, cf. \cite[Corollary A2]{Daley.Jones.book}. Hence, $\text{Point}_X$ inherits a structure of  metric space, and of  measurable space with the corresponding Borel $\sigma$-algebra. A (simple) point process on $X$ is then a random variable from a probability space $(\Omega, \mA,P)$ to  $\text{Point}_X$. Given two complete separable metric spaces $X$ and $M$,  a marked point process on $X$ with marks in $M$ is a point process on $X \times M$ such that its projection on $X$ is a point process.  See \cite{Daley.Jones.book2} for more. 

\medskip
To model our initial collection of particles with various orientations, we start from a marked  point process $\Lambda = \Lambda(\omega) = \{(X'_k(\omega),\xi'_k(\omega)), k \in \N\}$ on $\R^3 \times \S^2$, where the $X'_k$'s are the centers of our particles and the $\xi'_k$'s their orientations. We further impose a separation condition between the centers: for all $k \neq l$, $|X'_k - X'_l| \ge c'$ for some fixed deterministic $c' > 0$.  Classically, through push-forward of $P$ by $\Lambda$,  we can always assume that the  probability space is 
$$\Omega = \Big\{ \omega = \{(X'_k, \xi'_k), k\in \N\}, |X'_k - X'_l| \ge c' \: \forall k \neq l \Big\} $$
(and $\Lambda(\omega) = \omega$). From there, we can define the family of shifts 
$(\tau_y)_{y \in \R^3}$ by 
$$ \tau_y : \Omega \rightarrow \Omega, \quad  \tau_y \Big(\{(X'_k, \xi'_k), k \in \N\}\Big) = \{(X'_k-y, \xi'_k), k \in \N\}  $$
Finally, we assume  
\begin{itemize}
\item Stationarity: the law $P$ of the marked point process satisfies $ P \circ \tau_y = P$ for all $y \in \R^3$. 
\item Ergodicity: for all $A \in \mA$, if $\tau_y(A) = A$ for all $y \in \R^3$, then $P(A) = 0$ or $P(A) =1$
\end{itemize}
Under these assumptions, one has for some $\lambda > 0$ and $\kappa$ a  measure on the sphere: 
\begin{equation}\label{intensity}\E \sum_{k} \delta_{(X'_k,\xi'_k)} = \lambda \dd \kappa(\xi) 
\end{equation}
The  ergodic theorem (see \cite[Theorem 7.2]{JKO2012}) yields that for all $g \in L^1(\Omega)$, almost surely 
\begin{equation} \label{convergence_ergodic}
\frac{1}{n^3|A|}\int_{n A} g(\tau_y \omega) \dd y \xrightarrow[n \rightarrow +\infty]{} \E g  
\end{equation}
for all bounded measurable subset $A \subset \R^3$ with $|A| > 0$.
Actually, this ergodic theorem can be generalized, replacing $\dd y$ at the left-hand side by a stationary random measure $\mu_\omega$, and $\E g $ by $\int_\Omega g d\mu$, where $\mu$ is the Campbell measure associated to the random measure. See \cite[Chapter 13]{Daley.Jones.book2} for details. In the setting of our marked point process, where $\mu_\omega = \sum_{k} \delta_{(X'_k,\xi'_k)}$ for $\omega = \{(X'_k,\xi'_k)\}$,  we get that for any $F \in L^1(d\kappa)$, almost surely, 
\begin{equation}\label{convergence_ergodic2}
\frac{1}{n^3 |A|} \sum_{k, X'_k \in n A} F(\xi'_k) \xrightarrow[n \rightarrow +\infty]{}  \lambda \int_{\S^2} F(\xi) \dd \kappa(\xi). 
\end{equation}
for all bounded measurable subset $A \subset \R^3$ with $|A| > 0$.
For later use of stochastic two scale convergence, we record  
\begin{lemma} \label{lemma_Omega_compact}
$\Omega$ is a compact metric space.
\end{lemma}
\begin{proof} As explained above, the metrizable topology of $\Omega$ is inherited from the metrizable topology of weak-$\ast$ convergence on the space of boundedly finite Borel measures on $\R^3 \times \S^2$, by identifying any $\omega = \{ (X'_k, \xi'_k) \}$ with $\mu_\omega = \sum_{k} \delta_{(X'_k, \xi'_k)}$. To show compactness, it is then enough to prove that for any sequence  $(\omega_n)_{n \in \N}$ in $\Omega$, $\mu_{\omega_n}$ has a subsequence that weakly-* converges to an element of the form $\mu_\omega$, $\omega \in \Omega$. This is an easy consequence of the separation condition. Indeed, given any measurable bounded set $B \subset \R^3$, denoting $N_B(n)$ the number of points $(x, \xi)$ in $\supp(\mu_{\omega_n}) \cap (B \times \S^2)$, we have $N_B(n) \lesssim \frac{|B|}{(c')^3}$. We deduce that there exists a subsequence, still denoted $N_B(n)$ to lighten notations,  such that $N_B(n) = N_B$ is constant for large $n$.  From there, for large $n$, the restriction $\mu^n_B$ of $\mu_{\omega_n}$ to measurable subsets of $B \times \S^2$ reads  $\mu^n_B = \sum_{i=1}^{N_B} \delta_{Y_i^n, \eta_i^N}$ with $|Y_i^n - Y_j^n| \ge c'$. By compactness of $B \times \S^2$, one can further extract a subsequence of $\mu^n_B$ weakly converging to $\mu_B = \sum_{i=1}^{N_B} \delta_{Y_i, \eta_i}$,  with $|Y_i - Y_j| \ge c'$. By a diagonal process, the weak-* convergence of $\mu_{\omega_n}$ to some $\mu_\omega$ follows. 
\end{proof}
Finally, given a small parameter $\eps$, and given a smooth bounded domain $\mO$, we set 
$$\{(X_1^0, \xi_1^0) \dots (X_N^0, \xi_N^0) \}= \{(\eps X'_k, \xi'_k), k \in \N, B_\eps (X_k') \in \mO \}$$
Taking $n = \eps^{-1}$ and $A = \mO$ in \eqref{convergence_ergodic2} we find that almost surely 
\begin{equation} \label{asymptN}
    N \sim \frac{\lambda|\mO|}{\eps^3}, \quad \eps \rightarrow 0.
\end{equation}
This implies that almost surely, for $\eps$ small enough,
$$ |X_i^0 - X_j^0| \ge c' \eps  \ge \frac {c'}{2}|\mO|^{1/3} \lambda^{1/3} N^{-1/3}  $$
so that the separation assumption $|X_i^0 - X_j^0| \ge c N^{-1/3}$ is satisfied taking $c := \frac {c'}{2} |\mO|^{1/3} \lambda^{1/3}$. 

\subsection{Proof of Theorem \ref{theo3}} \label{subsec:proof_random1}
We first show that  we can replace the particle positions and orientations $X_i^{app}, \xi_i^{app}$, $1 \le i \le N$ by their  $0$-th order approximation:
\begin{align} 
\frac{\d}{\d t} \bar X^{app}_i & = u(\bar X_i^{app} ), \quad \bar X_i(0) = X_i^0, \\ 
\frac{\d}{\d t} \bar \xi^{app}_i & = M(\bar \xi_i^{app} ) \nabla u(\bar X_i^{app} )   \bar \xi_i^{app},  \quad \bar \xi_i^{app} (0) = \xi_i^0.
\end{align}
We recall  these approximate positions and orientations are generated through the flow $\Phi (s,x)$ and $\Xi (s,x,\xi)$ defined in \eqref{Phi}--\eqref{Xi}.
% \begin{align}
%     \partial_s \Phi (s,x) = u(s,\Phi(s,x)), \qquad \Phi(0,x) = x, \\
%     \partial_s \Xi (s,x,\xi) = M(\Xi(s,x,\xi)) \nabla u(s,\Phi(s,x)) \Xi(s,x,\xi),  \qquad \Xi(s,x,\xi) (0,x,\xi) = \xi.
% \end{align}
In the following, we will typically fix $t>0$ and often omit the time variable.
By Theorem \ref{th:well-posedness}, $\Phi \in W^{2,\infty}(\R^3;\R^3)$, $\Xi \in W^{1,\infty}(\R^3 \times \S^2;\S^2)$ and, since $\dv u = 0$,  $\Phi$ is a volume preserving bi-Lipschitz map.

\begin{lemma} \label{lem.sum.app}
For any $\Psi_1, \Psi_2 \in C(\R^3 \times \S^2; \R_0^{3\times3})$ we have
\begin{align}
    \lim_{N\to \infty}\frac 1 {N^2} \sum_i \sum_{j \neq i} &\Psi_1 (X_i^{app},\xi_i^{app}) :  \big( \nabla^2 G(X_i^{app} - X_j^{app})  \Psi_2(X_j^{app},\xi_j^{app})\big) \\
    &- \Psi_1(\bar X_i^{app},\bar \xi_i^{app}) :  \big( \nabla^2 G(\bar X_i^{app} - \bar X_j^{app})  \Psi_2(\bar X_j^{app},\bar \xi_j^{app})\big) = 0
\end{align}
\end{lemma}
\begin{proof}
Mimicking the proof of \eqref{relative.distance}, we get for all $1 \leq i,j \leq N$
\begin{align}
        |X_i - \bar X_i^{app} - (X_j - \bar X_j^{app})| &\lesssim r + \phi \log N |X_i - X_j|.
\end{align}
Combining with \eqref{relative.distance} itself, we get 
$$|X_i^{app} - \bar X_i^{app} - (X_j^{app} - \bar X_j^{app})| \lesssim r + \phi \log N |X_i - X_j|.
$$
In particular, 
\begin{align}
   &| \nabla^2 G(X_i^{app} - X_j^{app})  - \nabla^2 G(\bar X_i^{app} - \bar X_j^{app})| \\
   &\hspace{2cm}\lesssim (r+ \phi \log N |X_i - X_j|)\bigg(\frac{1}{|X_i^{app} - X_j^{app}|^4} + \frac{1}{|\bar X_i^{app} - \bar X_j^{app}|^4} \bigg) \\
   &\hspace{2cm} \lesssim  \frac{r}{|X_i^0 - X_j^0|^4} + \phi \log N \frac 1 {|X_i^0 - X_j^0|^3}
\end{align}
where we used that $\Phi$ is bi-Lipschitz and \eqref{dmin.control}.
Arguing as in the proof of Theorem \ref{th:app}, Steps 2 and 3, we deduce that for all $1 \leq i \leq N$
\begin{align}
        |X_i^{app} - \bar X_i^{app}| + |\xi_i^{app} - \bar \xi_i^{app}| &\lesssim \phi \log N  + r.
\end{align}
From this we deduce the claim since $\phi (\log N)^2 \to 0$ by \eqref{ass:diluteness}.
\end{proof}
We now turn to the proof of Theorem \ref{theo3}. Let $T > 0$ be arbitrarily large. From the bi-Lipschitz character of the flow $\Phi$ of $u$, there exists $C > 0$ such that for all $t \in [0,T]$ 
\begin{equation} \label{dmin.control2}
\frac{1}{C} |X^0_i - X^0_j| \le  |\bar X_i^{app}(t) - \bar X_j^{app}(t)| \le C |X^0_i - X^0_j|  .
\end{equation} 
 We shall rely on Theorem \ref{theo2} and  prove that both 
$$ \int \nabla \bar h^{N,\eta}[\Psi_1] :  \nabla \bar h^{N,\eta}[\Psi_2] \dd x  \quad \text{ and } \quad    \sum_i \int \nabla \bar h^{N,\eta}_i[\Psi_1] :  \nabla \bar h^{N,\eta}_i[\Psi_2] \dd x$$
have individually a limit, almost surely, with the choice $\eta = \frac{c'}{ 8 C}\eps  = c'' \eps$, where $C$ is the constant from \eqref{dmin.control2}.  Here $\bar h^{N,\eta},\bar h^{N,\eta}_i$ are defined as $ h^{N,\eta}, h^{N,\eta}_i$ corresponding to the particle configuration $(\bar X_i^{app},\bar \xi_i^{app})_i$.
Note that for $\eps$ small enough, still by \eqref{asymptN}, $\eta$ obeys the bound \eqref{eta_N}. Moreover, by polarization, we can restrict to the case $\Psi_1 = \Psi_2 = \Psi$. 
\begin{proposition} \label{pro:det.limit}
Almost surely, $\int |\nabla \bar h^{N,\eta}[\Psi]|^2 \dd x$ converges to a deterministic limit.
\end{proposition}
\noindent
{\em Proof.} We remind the definition of $\Theta^\eta[A]$ in \eqref{def_Psi_eta}, and denote 
\begin{align}
    F_\eps(x,\omega) &:= g\Big(\Phi^{-1}(x),\tau_{\frac{\Phi^{-1}(x)}\eps} \omega\Big), \\
    g(x,\omega) &:= \1_\mO(x)\sum_{k \in \N} \Theta^{c''}[\Psi(\Phi(x),\Xi(x,\xi_k')](-\nabla\Phi(x) X_k').
\end{align}
Let $h_\eps = h_\eps(x,\omega) \in \dot H^1(\R^3)$ be the solution to
\begin{align} \label{h^eps}
\begin{aligned}
     - \Delta h_\eps + \na  p_\eps & = \div F_\eps \\
     \div  h_\eps & = 0.
\end{aligned}
\end{align}
Moreover let $r_\eps$ be defined through
$$ \bar h^{N,\eta}[\Psi] = \frac{1}{N\eps^3} h_\eps + r^\eps. $$ \\
\emph{Step 1: Proof that
\begin{align} 
    \lim_{\eps \to 0} \|\na r^\eps\|_{L^2(\R^3)} = 0.
\end{align}
}
We define 
\begin{align} \label{tilde.F_eps}
    \tilde F_\eps(x,\omega) =  \eps^3 \sum_{k \in \N} \1_{\mO}(\eps X'_k) \Theta^{c'' \eps}[\Psi(\bar X^{app}_k,\bar \xi^{app}_k)](x -  \bar X^{app}_k).
\end{align}
Then,
\begin{align}
     - \Delta \bar h^{N,\eta}[\Psi] + \na \bar p^{N,\eta}[\Psi] & = \frac{1}{N\eps^3}  \div \tilde F_\eps \\
     \div \bar h^{N,\eta}[\Psi] & = 0.
\end{align}
Hence, it suffices to show that 
\begin{align}  \label{F.tilde.F}
    \| \tilde F_\eps(\cdot,\omega) - F_\eps(\cdot,\omega)\|_{L^2(\R^3)} = o(1).
\end{align}

% so that we are left with the asymptotic behaviour of $\int_{\R^3} |\na h_\eps|^2$. 

% \begin{align}
%     \tilde F_\eps(x,\omega) = \1_{\mO}(x) \eps^3 \sum_{k \in \N} \Theta^{c''\eps\eps}[\Psi(\bar X^{app}_k,\bar \xi^{app}_k)](x -  \bar X^{app}_k).
% \end{align}
% Let $\tilde h_\eps = \tilde h_\eps(x,\omega) \in \dot H^1(\R^3)$ be the solution to
% \begin{align}
%      -\tilde \Delta h_\eps + \na \tilde p_\eps & = \div \tilde F_\eps \\
%      \div \tilde h_\eps & = 0.
% \end{align}

By the choice of $c''$ (before the statement of Proposition \ref{pro:det.limit}), the sum in \eqref{tilde.F_eps} contains at most one nonzero term.
Since $\Theta^\eta[A](x) = \frac{1}{\eta^3}\Theta^{1}[A](x/\eta)$,
we have 
\begin{align}
    \tilde F_\eps(x,\omega) =  \sum_{k \in \N} \1_{\mO}(\eps X'_k) \Theta^{c''}[\Psi(\bar X^{app}_k,\bar \xi^{app}_k)]\bigg(\frac{x -  \bar X^{app}_k}{\eps}\bigg).
\end{align}
Using  $\supp  \Theta^{c''} = \overline{B(0,c'')}$, that $A \mapsto \Theta^{c''}[A]$ is linear and satisfies $\|\Theta^{c''}[A]\|_{L^\infty} \leq C$ 
 and recalling $(\bar X^{app}_k,  \bar \xi^{app}_k) = (\Phi(\eps X_k'),\Xi(\eps X_k',  \xi_k'))$, we find
\begin{align}
    &\bigg|\1_{\mO}(\eps X'_k) \Theta^{c''}[\Psi(\bar X^{app}_k,\bar \xi^{app}_k)]\bigg(\frac{x -  \bar X^{app}_k}{\eps}\bigg) -  \1_{\mO}(\eps X'_k)\Theta^{c''}[\Psi(x,\Xi(\Phi^{-1}(x),  \xi_k'))]\left(\frac{x -  \Phi(\eps X_k')}{\eps}\right) \bigg| \\
    &= \bigg|\1_{\mO}(\eps X'_k) \Theta^{c''}[\Psi(\bar X^{app}_k,\bar \xi^{app}_k) - \Psi(x,\Xi(\Phi^{-1}(x),  \xi_k'))]\left(\frac{x -  \Phi(\eps X_k')}{\eps}\right) \bigg| \\
    & \leq C \delta_\eps,
\end{align}
where 
\begin{align}
    \delta_\eps := \sup_{(x,\xi) \in \mO \times \S^2} \sup_{y \in B(x,c'' \eps)} |\Psi(x,\Xi(\Phi^{-1}(x),  \xi)) - \Psi(y,\Xi(\Phi^{-1}(y),  \xi))| \to 0 \quad \text{as } \eps \to 0
\end{align}
since $\Psi$ is uniformly continuous on compact sets and $\Phi$ and $\Xi$ are bi-Lipschitz.
Therefore,
\begin{align} \label{manip.F.tilde.1}
\begin{aligned}
    \tilde F_\eps(x,\omega) 
    &  =  \sum_k \1_{\mO}(\eps X'_k)\Theta^{c''}[\Psi(x,\Xi(\Phi^{-1}(x),  \xi_k'))]\left(\frac{x -  \Phi(\eps X_k')}{\eps}\right) + \1_{\Phi(\mO)} o(1) \\
    & = \1_{\Phi(\mO)}(x)   \sum_k \Theta^{c''}[\Psi(x,\Xi(\Phi^{-1}(x),  \xi_k'))]\left(\frac{x -  \Phi(\eps X_k')}{\eps}\right) \\
    & + \1_{B(\partial \mO,c''\eps)}(x) O(1) + \1_{\Phi(\mO)} o(1) ,
     % \\
     % &= \Theta^{c''}[\Psi(\bar X^{app}_k,\bar \xi^{app}_k)]\left(\nabla\Phi(\Phi^{-1}(x))\left(\frac{\Phi^{-1}(x)}{\eps} -  X_k'\right)\right) + O(\eps)
\end{aligned}
\end{align}
 where $B(\partial \mO,c''\eps) = \{x \in \R^3 : \dist(x,\partial \mO) < c'' \eps\}$.
By Taylor expansion, we write
\begin{align}
    x - \Phi(\eps X_k') = x - \Phi(\Phi^{-1}(x)) + \nabla \Phi(\Phi^{-1}(x))( \Phi^{-1}(x) - \eps X_k') + O(\eps^2) 
\end{align}
for all $x$ and $k$ with  $\frac{x -  \Phi(\eps X_k')}{\eps} \in \supp \Theta^{c''}$.
Since $\Theta^{c''}$ is smooth in $B(0,c'')$ and $\Psi$ is uniformly bounded in $\mO \times \S^2$, we deduce
\begin{align} \label{manip.F.tilde.2}
\begin{aligned}
    &\1_{\Phi(\mO)}(x)   \sum_k \Theta^{c''}[\Psi(x,\Xi(\Phi^{-1}(x),  \xi_k'))]\left(\frac{x -  \Phi(\eps X_k')}{\eps}\right) \\
    & \qquad = \1_{\Phi(\mO)}(x)   \sum_k \Theta^{c''}[\Psi(x,\Xi(\Phi^{-1}(x),  \xi_k'))]\bigg(\frac{\nabla \Phi(\Phi^{-1}(x))( \Phi^{-1}(x) - \eps X_k')}{\eps}\bigg) + O(\eps) \1_{\Phi(\mO)}(x)  \\
    & \qquad \qquad + O(1) \1_{\Phi(\mO)}(x)   \sum_k \1_{B(\partial B(c'' \eps,\Phi(\eps X_k')),L \eps^2)},
\end{aligned}
\end{align}
where $L := c'' \|\nabla^2 \Phi\|_\infty \|\nabla \Phi^{-1}\|_\infty $.
Combining \eqref{manip.F.tilde.1} and \eqref{manip.F.tilde.2} yields \eqref{F.tilde.F}.\\[3mm]
\emph{Step 2: Two-scale convergence of $(h_\eps \circ \Phi,p_\eps \circ \Phi)$.}

As $\Omega$ is compact, \emph{cf.} Lemma \ref{lemma_Omega_compact}, we can exploit the  notion of quenched 2-scale convergence introduced by Piatnitski and Zhikov in \cite{ZhiPia2006}, see also \cite{HeNeVa2022} for recent developments. The properties we need are listed in Appendix \ref{appA}.

First, from the continuity of Riesz transforms over $L^p(\R^3)$ for all $1 < p < +\infty$, we find that almost surely
$$ \|\na h_\eps \circ \Phi\|_{L^p} +  \|p_\eps \circ \Phi \|_{L^p} \le C \|F_\eps\|_{L^p(\R^3)} \leq C, $$
where we used in the first estimate that $\Phi$ is a volume preserving diffeomorphism and in the second estimate that $F$ is bounded and compactly supported in $\mathcal O$ with respect to the first variable.
In particular, $\|h_\eps \circ \Phi\|_{H^1(K)} \le C$ for all compact subset $K$. From the results of Appendix \ref{appA}, it follows that almost surely,  there exists $h_0 =h_0(x) \in L^2_{loc}(\R^3)$,  $D_H = D_H(x,\omega) \in L^2_{loc}(\R^3, L^2_{pot}(\Omega))$, $P_0 = P_0(x,\omega) \in L^2_{loc}(\R^3, L^2(\Omega))$  and a subsequence in $\eps$, along which 
\begin{itemize}
    \item $h_\eps \circ \Phi$, $\na (h_\eps \circ \Phi) $ and $p_\eps \circ \Phi$   two-scale converge respectively  to $h_0$, $\na h_0 + D_H$ and $P_0$.
    \item $h_\eps \circ \Phi$ and $p_\eps \circ 
    \Phi$ converge weakly to $h_0$ and to $p_0(x) = \E [P_0(x, \cdot)]$, in $\dot{H}^1(\R^3)$ and $L^2(\R^3)$ respectively. 
\end{itemize}
We stress that the fields  $D_H$, $P_0$ and the subsequence in $\eps$ \emph{a priori} depend on the realization.

To characterize the two-scale limits, we need to pass to the limit in the weak formulation of the Stokes equation \eqref{h^eps}.\\[3mm]
% \begin{align}
% & \int_{\R^3} \na h_\eps \cdot \na u  - \int_{\R^3} (\div u) p_\eps = - \int_{\R^3} F(x,x/\eps) : \na u(x) \dd x, \quad \forall u \in C^1_c(\R^3, \R^3)  \\
% & \int_{\R^3} (\div h_\eps) \psi= 0, \quad \forall \psi \in C^0_c(\R^3)  
% \end{align}
% The right-hand side of the Stokes equation converges thanks to the following lemma:
\emph{Step 3: Proof that under a typical realization $\omega$, we have $F_\eps(\cdot,\omega) \circ \Phi  \xrightarrow[]{2} g$.
%  all $u \in C^1_c(\R^3, \R^3)$
% $$\int_{\R^3} F_\eps(x,x/\eps) : \na u(x) \dd x \xrightarrow[\eps \rightarrow 0]{} \int_{\R^3} \E F(x,0) : \na u(x) dx =  \lambda  \int_{\mO} \Big(\int_{\S^2} \Psi(\cdot,\xi) d\kappa(\xi) \Big) : \na u. $$
}

% We introduce 
% $$g(x,\omega) = \sum_{k \in \N} \Theta^{c'/8}[\Psi(x,\xi'_k)](-X'_k) : \na u(x) $$ 
% (the dependance of the right-hand side on $\omega$ is hidden inside $X'_k$ and $\xi'_k$). Note that
This follows from the ergodic theorem \eqref{convergence_ergodic} upon freezing the slow variable. More, precisely, we partition a neighborhood of $\mO$ by a collection $\mC$ of cubes of diameter $\alpha$, and choose a point $x_C$ in each cube $C \in \mC$. Let $\varphi \in C_c^\infty(\R^3)$.
We then introduce (supressing the dependence on $\omega$)
\begin{align}
    f_{\eps,\alpha}(x) := \sum_{C \in \mC, C \subset \mO} \varphi(x_C) g(x_C,\tau_{x/\eps} \omega) \1_{C}(x)
\end{align}
We observe that for all $k\in \N$, $C \in \mC$ and  $x \in C$ 
\begin{align}
    \Big|\nabla \Phi(x_C)\Big(\frac x \eps - X_k'\Big)\Big| \leq c'' (1 - \alpha L) \quad \Longrightarrow \quad \Big|\nabla \Phi(x)\Big(\frac x \eps - X_k'\Big)\Big| \leq c'', \\
    \Big|\nabla \Phi(x_C)\Big(\frac x \eps - X_k'\Big)\Big| \geq c'' (1 + \alpha L) \quad \Longrightarrow \quad \Big|\nabla \Phi(x)\Big(\frac x \eps - X_k'\Big)\Big| \geq c'',
\end{align}
where $L =  \|(\nabla \Phi)^{-1}\|_\infty \|\nabla^2 \Phi\|_\infty$. Indeed, for $x \in C$, 
\begin{align}
    \Big|(\nabla \Phi(x) - \nabla \Phi(x_C))\Big(\frac x \eps - X_k'\Big)\Big| &\leq \alpha \|\nabla^2 \Phi\|_\infty \Big|\frac x \eps - X_k'\Big| \\
    &\leq \alpha \|\nabla^2 \Phi\|_\infty |(\nabla \Phi(x_C))^{-1} | \Big|\nabla \Phi(x_C)\Big(\frac x \eps - X_k'\Big)\Big|
\end{align}
which yields both implications (the second one by contraposition).
Hence, using once more that $\Theta^{c''}$ is smooth in $\supp \Theta^{c''} = B(0,c'')$
and that $\Psi$ is locally uniformly continuous, we deduce
\begin{align} \label{eps.alpha.uniform}
     \limsup_{\alpha \to 0} \sup_{\eps}\|f_{\eps,\alpha} - \varphi  F_\eps \circ \Phi\|_{L^2(\R^3)} = 0.
\end{align}

Let  $b \in L^2(\Omega)$. Then, by the ergodic theorem \eqref{convergence_ergodic}, we have
\begin{align}
    \lim_{\eps \to 0} \int_{\R^3} f_{\eps,\alpha}  b(\tau_{x/\eps} \omega) &= \lim_{\eps \to 0} \sum_{C \in \mC, C \subset \mO} \int_{C} g(x_C,\tau_{x/\eps}\omega) \varphi(x_C) b(\tau_{x/\eps} \omega) \dd x \\
    &=  \lim_{\eps \to 0} \sum_{C \in \mC, C \subset \mO} \eps^3 \varphi(x_C) \int_{\eps^{-1} C}  g(x_C,\tau_y \omega) b(\tau_{y} \omega)  \\
     & = \sum_{C \in \mC, C \subset \mO} |C| \varphi(x_C) \E[g(x_C,\cdot) b].
\end{align}
The right-hand side converges to $\int \E[g(x,\cdot))  b] \varphi(x) \dd x $ as $\alpha \to 0$. Since the uniform bound \eqref{eps.alpha.uniform} allows us to commute the limits in $\eps$ and $\alpha$ this proves the two-scale convergence of $F_\eps$. \\[3mm]
\emph{Step 4: Characterization of the limit PDEs.}
% We observe that 
% \begin{align*}
%      \E [g(x,\cdot)] & = 1_\mO(x) \E \int_{\R^3 \times \S^2} \Theta^{c''}[\Psi(\Phi(x),\Xi(x,\xi)](-\nabla\Phi(x) y) \sum_k \delta_{(X'_k,\xi'_k)}(dy,d\xi) \\
%      & = \lambda 1_\mO(x) \int_{\R^3} \int_{\S^2} \Theta^{c''}[\Psi(\Phi(x),\Xi(x,\xi))](-\nabla\Phi(x) y) d\kappa(\xi) dy \\ 
%      & = \lambda 1_\mO(x) \int_{\R^3} \int_{\S^2} \Theta^{c''\eta}[\Psi(\Phi(x),\Xi(x,\xi))]( z)d\kappa(\xi) dz 
% \end{align*}
% where the last equality results from the change of variable $z = -\eta \nabla \Phi(x) y$. As, by Lemma \ref{lemma_regul}, $\Theta^\eta[A] \rightarrow A\delta_0$ in the sense of distributions, we deduce 
% $$ \E [g(x,\cdot)]  = \lambda 1_\mO(x) \int_{\S^2} \Psi(\Phi(x),\Xi(x,\xi)) d\kappa(\xi) $$
We  test the weak formulation of \eqref{h^eps} with $\varphi \circ \Phi^{-1}$ where $\varphi  \in C^1_c(\R^3, \R^3)$ and apply a change of variables yielding
\begin{align}
    \int_{\R^3} (\nabla h_\eps \circ \Phi)  : (\nabla \varphi(x) (\nabla \Phi)^{-1}) - (p_\eps \circ \Phi) \na \varphi : (\nabla \Phi)^{-1} = - \int_{\R^3}  (F_\eps \circ \Phi) : (\nabla \varphi(x) (\nabla \Phi)^{-1}).
\end{align}
From the weak convergence of $\na h_\eps \circ \Phi$, $p_\eps \circ \Phi$, and $F_\eps \circ \Phi$, we find 
\begin{align} \label{eqh0}
\begin{aligned}
& \int_{\R^3} \na h_0 (\na \Phi)^{-1} : (\na \varphi (\na \Phi)^{-1}) - \int_{\R^3} p_0 \na \varphi : (\nabla \Phi)^{-1} \\
&\hspace{5cm} = - \int_{\mO}  \E [g] : (\nabla \varphi (\nabla \Phi)^{-1})  \quad \forall \varphi \in C^1_c(\R^3, \R^3),  \\
& \int_{\R^3} \na h_0 : (\nabla \Phi)^{-1} \psi= 0 \quad \forall \psi \in C^0_c(\R^3).     
\end{aligned}
\end{align}
% Hence, in strong form, we have 
% \begin{equation} \label{eqh0}
% \begin{aligned}
%  -\Delta h_0 + \na P_0 & = \div \Big( \lambda 1_{\mO} \E [g(x,\cdot)] \Big)  \\
%      \div h_0 & = 0
% \end{aligned}
% \end{equation}
Next, we  test the weak formulation of \eqref{h^eps} with
$$ \tilde \varphi(x) = \eps  \tilde \psi(\Phi^{-1}(x)) b(\tau_{\Phi^{-1}(x)/\eps}\omega), \quad \psi(x) = \tilde \psi(\Phi^{-1}(x)) c(\tau_{\Phi^{-1}(x)/\eps}\omega)  $$
where $\tilde \psi \in C^\infty_c(\R^3)$, while $b=b(\omega)$  belongs to $C^1(\Omega)$ (with values in $\R^3$, see Appendix \ref{appA} for the definition of $C^1(\Omega)$), and $c=c(\omega)$ belongs to $L^2(\Omega)$  (with values in $\R$). 
% From the same arguments as in the proof of Lemma \ref{lemma_ergodic}, we deduce 
% $$\lim_{\eps \rightarrow 0} \int_{\R^3} F(x,x/\eps) : \na_\omega b(\tau_{x/\eps} \omega) \tilde \psi(x) \dd x = \int_{\R^3} \E [F(x,0) :\na_\omega b] \tilde \psi(x) \dd x $$
% \rhcomment{The $\omega$ in the argument of $b$  is the same fixed $\omega$ that generates $F$, right? Otherwise I don't see how the same proof as for lemma 3.3 applies. \dgv{Sure, it must be the same.}
% Also the above precisely assures two-scale convergence of $F(x,x/\eps)$, right? Shouldn't that be standard and found in the literature. It seems like this is the easiest nontrivial example of two-scale convergence. \dgv{I agree, but could not find it in the paper of Zhikov (and I am not sure that many other papers use this quenched two-scale cv)}} 
% so that 
% \begin{align*} 
% & - \int_{\R^3} F(x,x/\eps) : \na u(x) \dd x \\
%  = & - \int_{\R^3} F(x,x/\eps) : \na_\omega b(\tau_{x/\eps} \omega) \tilde \psi(x) \dd x - \eps \int_{\R^3}F(x,x/\eps) :(\na \tilde \psi(x) \otimes  b(\tau_{x/\eps} \omega) ) \dd x  \\
%  = & - \int_{\R^3} F(x,x/\eps) : \na_\omega b(\tau_{x/\eps} \omega) \tilde \psi(x) \dd x + O(\eps) 
%  \xrightarrow[\eps \rightarrow 0]{}   - \int_{\R^3} \E [F(x,0) :\na_\omega b] \tilde \psi(x) \dd x 
% \end{align*}
By application of the two-scale convergence of $h_\eps \circ \Phi$ and $p_\eps \circ \Phi$, keeping in mind that $\E \na_\omega b = 0$,
\begin{align*}
 &    \int_{\R^3} \na h_\eps : \na \tilde \varphi  - \int_{\R^3} (\div \tilde \varphi ) p_\eps \\
 = &  \int_{\R^3} (\na h_\eps \circ \Phi) : (\na_\omega b(\tau_{x/\eps}\omega) (\nabla \Phi)^{-1})   \tilde{\psi}   -  \int_{\R^3}  (p_\eps \circ \Phi) \na_\omega  b(\tau_{x/\eps}\omega) : (\nabla \Phi)^{-1}   \tilde{\psi} \\  
 + & \eps \int_{\R^3} (\na h_\eps(x) \circ \Phi) : ((\na \tilde{\psi} (\nabla \Phi)^{-1}) \otimes b(\tau_{x/\eps}\omega))   - \eps \int_{\R^3}  ( p_\eps \circ \Phi) (\na \tilde{\psi} \otimes b(\tau_{x/\eps}\omega)) : (\nabla \Phi)^{-1}     \\
   = &    \int_{\R^3} (\na h_\eps \circ \Phi) : (\na_\omega b(\tau_{x/\eps}\omega) (\nabla \Phi)^{-1})   \tilde{\psi}   -  \int_{\R^3}  (p_\eps \circ \Phi) \na_\omega  b(\tau_{x/\eps}\omega) : (\nabla \Phi)^{-1}   \tilde{\psi} + O(\eps)  \\\xrightarrow[\eps \rightarrow 0]{} & \int_{\R^3} \E [(D_H  (\nabla \Phi)^{-1})  : (\na_\omega b (\nabla \Phi)^{-1})] \, \tilde \psi   - \int_{\R^3} \E [(\na_\omega b : (\nabla \Phi)^{-1}) (P_0 - p_0)] \, \tilde \psi
    \end{align*}
    Similarly, we get 
    $$\int_{\R^3} (\div h_\eps) \psi \xrightarrow[\eps \rightarrow 0]{} \int_{\R^3} \E [D_H : (\nabla \Phi)^{-1}) c ]\, \tilde \psi   $$
Using also the two-scale convergence of $F_\eps \circ \Phi$, we end up with
\begin{align}
& \int_{\R^3} \E [(D_H  (\nabla \Phi)^{-1})  : (\na_\omega b (\nabla \Phi)^{-1})] \, \tilde \psi   - \int_{\R^3} \E [(\na_\omega b : (\nabla \Phi)^{-1}) (P_0 - p_0)] \, \tilde \psi \\
&\hspace{9cm} = - \int_{\R^3} \E [g : (\na_\omega b (\nabla \Phi)^{-1}) ] \, \tilde \psi   \\
& \int_{\R^3} \E [D_H : (\nabla \Phi)^{-1}) c ]\, \tilde \psi = 0, 
\end{align}
As $\tilde \psi$ is arbitrary, this formulation is equivalent to: for all $b \in C^1(\Omega)$, for almost every $x$,
\begin{equation} \label{eqDH}
\begin{aligned}
&  \E [(D_H  (\nabla \Phi)^{-1})  : (\na_\omega b (\nabla \Phi)^{-1})]    - \E [(\na_\omega b : (\nabla \Phi)^{-1}) (P_0 - p_0)]  = -  \E [g : (\na_\omega b (\nabla \Phi)^{-1}) ]   \\
&  D_H : (\nabla \Phi)^{-1}) = 0, 
\end{aligned}
\end{equation} \\[3mm]
\emph{Step 5: Conclusion.}
By a density argument, we can replace $\na_\omega b$ by $D_H(x,\cdot)$ in  \eqref{eqDH} and find
\begin{equation} \label{energy_DH.new}
\E |D_H  (\nabla \Phi)^{-1}|^2  = -  \E [g : (D_H  (\nabla \Phi)^{-1})]  
\end{equation}
On the other hand, taking $\varphi = h_\eps$ in the weak formulation of \eqref{h^eps}, we have 
\begin{align}
\int_{R^3} |\na h_\eps|^2 = - \int_{\R^3} F_\eps : \na h_\eps &= \int_{\R^3} (F_\eps \circ \Phi) : \na (h_\eps \circ \Phi) (\nabla \Phi)^{-1}  \\
&\xrightarrow{\eps \rightarrow 0} - \int_{\R^3} \E [g(x,\cdot) : (\na h_0 + D_H)(\nabla \Phi)^{-1}].
\end{align}
where the convergence comes from  the two-scale convergence of $\na h_\eps$, and the fact that  
$(F_\eps \circ \Phi)(x) = g(x,\tau_{x/\eps}\omega)$ is an admissible test function. More precisely, by definition of the two-scale convergence, one has 
\begin{equation} \label{2scaleh}
    \int_{\R^3}  G(x,\tau_{x/\eps})  : \na h_\eps(x) \:  \xrightarrow[\eps \rightarrow 0]{} \int_{\R^3} \E [ G : (\na h_0 + D_H)]
\end{equation}
for all functions $G$ which are sum of functions of the form $\varphi(x) b(\tau_y \omega)$ 
with $b \in L^2(\Omega)$ and smooth compactly supported $\varphi$.  One can, however, using  the same notation as in Step 3, approximate  $g(x,\tau_{x/\eps}\omega)$ uniformly in $L^2$  by 
$$\sum_{C \in \mC, C \subset \mO}  g(x_C,\tau_{x/\eps} \omega) \1_{C}(x).$$   
Thanks to this approximation property and further smoothing of the characteristic functions $1_C$, one can substitute $g$ to $G$ in \eqref{2scaleh}. We leave the details to the reader.

\medskip
Using the energy estimate $\int_{\R^3} |\na h_0 (\na \Phi)^{-1}|^2 = - \int_{\R^3} \E[g] : (\na h_0 (\na \Phi)^{-1})$  following from \eqref{eqh0}, and combining with \eqref{energy_DH.new}, we find 
$$  \int_{\R^3} |\na h_\eps|^2  \rightarrow \int_{\R^3} |\na h_0 (\na \Phi)^{-1}|^2 + \int_{\R^3} \E |D_H  (\nabla \Phi)^{-1}|^2. $$
We remind that so far, this convergence is only along a subsequence that depends on the realization $\omega$. To finish the proof, we still need to show that this accumulation point is unique. It is enough to show that  $h_0$ and $D_H$ are defined independently of the subsequence. As regards $h_0$, it is characterized as the unique solution in $\dot{H}^1(\R^3)$ of the Stokes type equation \eqref{eqh0}. As regards $D_H$, we deduce from \eqref{eqDH} that for almost every $x$, $D_H(x,\cdot)$ is a solution in the space $V = \{ D \in L^2_{pot}(\Omega),  D_H : (\nabla \Phi)^{-1} = 0\}$ of the weak formulation 
$$ \E [(D_H  (\nabla \Phi)^{-1})  : (D(\nabla \Phi)^{-1})]  = -  \E [g : D], \quad \forall D \in V   $$
By the Lax-Milgram lemma, this weak formulation has a unique solution, which characterizes $D_H$ and concludes the proof.  
\begin{proposition}
Almost surely 
$$\sum_i \int |\nabla \bar h^{N,\eta}_i[\Psi]|^2$$
has a deterministic limit. 
\end{proposition}
{\em Proof.}   We notice that for all $i=1\dots N$, 
$$ \bar h^{N,\eta}_i[\Psi](x) = \frac{1}{N \eps^2}\mH\Big(\frac{x-\bar X_i^{app}}{\eps}\Big) \Psi(\bar X_i^{app},\bar \xi_i^{app})   $$
 and  $\mH(x')$ is a linear map from $M_3(\R)$ to $\R^3$ defined by: for all constant $\Psi \in M_3(\R)$, $\mH \Psi$ is the solution of the Stokes equation
$$ -\Delta (\mH \Psi) + \na (\mP \cdot \Psi)  = \div (\Theta^{c''}[\Psi]) $$   
It follows that
\begin{equation} \label{sumhi}
\begin{aligned}
\sum_i \int |\nabla h^{N,\eta}_i[\Psi]|^2  & = \frac{1}{N^2 \eps^3} \sum_{i=1}^N \Big( \int_{\R^3} \na \mH^T \na \mH \Big) \Psi(\bar X_i^{app}, \bar \xi_i^{app}) : \Psi(\bar X_i^{app},\bar \xi_i^{app}) \\
& \hspace{-2cm} = \frac{1}{N^2 \eps^3} \sum_{k \in \N, \eps X'_k \in \mO} \Big( \int_{\R^3} \na \mH^T \na \mH \Big) \Psi(\Phi(\eps X'_k), \Xi(\eps X_k',\xi'_k)) : \Psi(\Phi(\eps X'_k), \Xi(\eps X_k',\xi'_k))  
\end{aligned}
\end{equation}
Based on this representation, we claim that almost surely 
\begin{align}
& \sum_i \int |\nabla h^{N,\eta}_i[\Psi]|^2 \\
&\xrightarrow[\eps \rightarrow 0]{} \frac{1}{\lambda |\mO|^2} \int_{\mO\times \S^2}  \Big( \int_{\R^3 } \na \mH^T \na \mH \Big) \Psi(\Phi(x), \Xi(x,\xi)) : \Psi(\Phi(x), \Xi(x,\xi))  \kappa(\d\xi) \dd x
\end{align}
This claim, concluding the proof, comes from two facts: first,  almost surely,  by  \eqref{asymptN}
$$ \frac{1}{\eps^6 N^2} \xrightarrow[\eps \rightarrow 0]{} \frac{1}{\lambda^2 |\mO|^2} $$
and second, almost surely, 
\begin{align*} 
\eps^3 \sum_{k \in \N, \eps X'_k \in \mO} \Big( \int_{\R^3} \na \mH^T \na \mH \Big) \Psi(\Phi(\eps X'_k), \Xi(\eps X_k',\xi'_k)) : \Psi(\Phi(\eps X'_k), \Xi(\eps X_k',\xi'_k))    \\
\xrightarrow[\eps \rightarrow 0]{} \lambda\int_{\mO\times \S^2}  \Big( \int_{\R^3 } \na \mH^T \na \mH \Big) \Psi(\Phi(x), \Xi(x,\xi)) : \Psi(\Phi(x), \Xi(x,\xi))  \kappa(\d\xi) \dd x
\end{align*}
This latter convergence can be shown by a partitioning argument much as in Step 3 of the previous proof. By locally uniform continuity of the function $\Big( \int_{\R^3} \na \mH^T \na \mH \Big) \Psi(x,\xi) : \Psi(x,\xi)$, one can partition a neighborhood of $\mO$ into small cubes $C$ on which the oscillation of this function is small. We then split the total sum $\sum_{k \in \N, \eps X'_k \in \mO}$ into sums of the form  $\sum_{k \in \N, \eps X'_k \in C}$, for which the summand does not depend on $\eps X'_k$ anymore, up to a small error. This allows to apply the ergodic theorem in the form \eqref{convergence_ergodic2}. We skip the details and refer to Step 3 of the previous proof for more.

\subsection{Proof of Theorem \ref{cor1}} \label{sec:correlation} 

In all what follows, we consider a fixed time $t$ and omit it from the notations. We also denote, for any matrices $A,B \in \R^{3 \times 3}$,  $A \otimes B :  \na^2 G$ for $A : (\na^2 G B)$. From Theorem \ref{theo3},  Lemma \ref{lem.sum.app}, and the fact that almost surely 
$$ \lim_{\eps \rightarrow  0} N \eps^3 = \lambda |\mO|$$
we know that for all $\Psi_1,\Psi_2$ continuous, almost surely, 
\begin{align*} 
\lim_{N \rightarrow \infty} \frac{1}{N^2}  \sum_i \sum_{j \neq i} \Psi_1(X^{app}_i,\xi_i^{app}) \otimes \Psi_2(X_j^{app},\xi_j^{app}) : \nabla^2 G(X_i^{app} - X_j^{app})    = \frac{1}{\lambda^2 |\mO|^2} \lim_{\eps \rightarrow 0} I^\eps
\end{align*} 
where 
\begin{align*}
    I^\eps  & = \eps^6 \sum_i \sum_{j \neq i} \Psi_1(\bar X_i^{app},\bar \xi_i^{app}) \otimes \Psi_2(\bar X_j^{app},\bar \xi_j^{app}) : \nabla^2 G(\bar X_i^{app}- \bar X_j^{app}) \\
    & = \eps^6 \sum_k \sum_{l \neq k} 1_{\mO}(\eps X'_k) 1_{\mO}(\eps X'_l)\Psi_1(\Phi(\eps X'_k),\Xi(\eps X'_k,\xi'_k)) \otimes \Psi_2(\Phi(\eps X'_l),\Xi(\eps X'_l,\xi'_l)) \\
  & \hspace{10cm}: \nabla^2 G(\Phi(\eps X'_k) - \Phi(\eps X'_l)) 
    \end{align*}
and that this limit is deterministic. Moreover, there exists $C > 0$ such that almost surely, $|I^\eps| \le C$:  this follows from Corollary \ref{coro:bound_IN} and the bound $N \lesssim \eps^{-3}$. Hence, by the dominated convergence theorem: 
\begin{equation}
\lim_{\eps \rightarrow 0} I^\eps = \E \lim_{\eps \rightarrow 0} I^\eps = \lim_{\eps \rightarrow 0} \E I^\eps.
\end{equation}
It remains to compute the right-hand side. Defining 
\begin{align*}
    I^\eps[\tilde{\Psi}_1,\tilde{\Psi}_2] & = \eps^6 \sum_k \sum_{l \neq k} \tilde \Psi_1(\Phi(\eps X'_k),\Xi(\eps X'_k,\xi'_k)) \otimes \tilde \Psi_2(\Phi(\eps X'_l),\Xi(\eps X'_l,\xi'_l))  \\
    & \hspace{9cm} : \nabla^2 G(\Phi(\eps X'_k) - \Phi(\eps X'_l))  , \\
J^0[\tilde{\Psi}_1,\tilde{\Psi}_2] & = \int_{(\S^2)^2 \times (\R^3)^2}  \tilde \Psi_1(x_1,\xi_1)  \otimes  \tilde \Psi_2(x_2,\xi_2) : \nabla^2 G(x_1-x_2)   \tilde \mu(\d x_1,\d \xi_1) \tilde \mu(\d x_2,\d\xi_2)  \\
 & +  \int_{(\S^2)^2  \times (\R^3)^2} \tilde {\Psi}_1(x,\xi_1) \otimes \tilde \Psi_2(x,\xi_2) : \na^2 G(y)\tilde \nu_{2}(\d x, \d y,\d\xi_1,\d\xi_2), 
\end{align*}
where
\begin{align}
    \tilde \mu(\d x, \d \xi) &= (\Phi(x),\Xi(x,\xi))\# (\lambda  \dd x \,  \kappa(\d \xi)), \\
    \tilde \nu_{2}(\d x, \d y, \d \xi_1, \d \xi_2)  
    &= (\Phi(x), \nabla \Phi(x) y , \Xi(x,\xi_1),\Xi(x,\xi_2)) \\
    & \hspace{1cm}\#( \d x \,  \nu_2(\d y,\d \xi_1,\d \xi_2) - \lambda^2   \dd x \dd y\,\kappa(\d \xi_1)  \kappa(\d \xi_2)),
\end{align}
we want to show that 
\begin{equation} \label{convJeps}
\E \, I^\eps[\tilde{\Psi}_1,\tilde{\Psi}_2] \xrightarrow[\eps \rightarrow 0]{} J^0[\tilde{\Psi}_1,\tilde{\Psi}_2]  
\end{equation}
with the special choice 
\begin{equation} \label{def_tilde_Psi}
\tilde{\Psi}_i(x,\xi)= 1_{\Phi(\mO)}(x) \Psi_i(x,\xi).
\end{equation}
% Indeed, through a change of variables
% \begin{align}
%      J^0[\tilde{\Psi}_1,\tilde{\Psi}_2] &= \int_{(\S^2)^2 \times (\R^3)^2}  \tilde \Psi_1(x_1,\xi_1) : \nabla^2 G(x_1-x_2) :  \tilde \Psi_2(x_2,\xi_2)  \mu_t(\dd x_1,\dd \xi_1) \mu_t(\dd x_2,\dd\xi_2)  \\
%  & +  \int_{(\S^2)^2  \times (\R^3)^2} \tilde {\Psi}_1(x,\xi_1) \otimes \tilde \Psi_2(x,\xi_2) \na^2 G(y)\tilde \nu_{2,t}(\dd x, \dd y,\dd\xi_1,\dd\xi_2), 
% \end{align}

Let us stress that, due to the singularity of $\na^2 G$, which is homogeneous of order $-3$, even the definition of $J^0[\tilde{\Psi}_1,\tilde{\Psi}_2]$ is not obvious.  We shall proceed in two steps.\\[3mm]
\emph{Step 1: Proof of \eqref{convJeps} for smooth and compactly supported $\tilde{\Psi}_1, \tilde{\Psi}_2$.}  
Let  $c'$ be the constant in \eqref{ass:separation'} and take $d > 0$ such that for all $y_1,y_2 \in \R^3$ we have $|\Phi(\eps y_1) - \Phi(\eps y_2)| \le d \eps \: \Rightarrow \: |y_1 - y_2| \le c'$.
In this case, by definition of the 2-point correlation function $$\mu_2(\d y_1,\d\xi_1,\d y_2,\d\xi_2) = \nu_2(y_1 - y_2,\d\xi_1,\d\xi_2) \dd y_1 \dd y_2$$
and the fact that it is zero for $|y_1 - y_2| \le c'$
\begin{align*}
    \E I^\eps[\tilde{\Psi}_1,\tilde{\Psi}_2]  = \eps^6 \int_{(\S^2)^2 \times (\R^3)^2} & 1_{|\Phi(\eps y_1) - \Phi(\eps y_2)| \ge d\eps} \tilde{\Psi}_1(\Phi(\eps y_1),\Xi(\eps y_1,\xi_1)) \otimes \tilde{\Psi}_2(\Phi(\eps y_2),\Xi(\eps y_2,\xi_2)) \\
    &\hspace{2cm}: (\nabla^2 G(\Phi(\eps y_1) - \Phi(\eps y_2))  )  \nu_2(y_1-y_2,\d\xi_1,\d\xi_2) \dd y_1 \dd y_2. 
\end{align*}
Hence, 
    \begin{align*}
  \E I^\eps[\tilde{\Psi}_1,\tilde{\Psi}_2] = J^{\eps,d}_1 + J^{\eps,d}_2
\end{align*}
where
\begin{align*}
   J^{\eps,d}_1 &= \eps^6 \int_{(\S^2)^2 \times (\R^3)^2} 1_{|\Phi(\eps y_1) - \Phi(\eps y_2)| \ge d\eps}  \tilde{\Psi}_1(\Phi(\eps y_1),\Xi(\eps y_1,\xi_1)) \otimes  \tilde{\Psi}_2(\Phi(\eps y_2),\Xi(\eps y_2,\xi_2))  \\
   & \hspace{4.5cm}  : \nabla^2 G(\Phi(\eps y_1) - \Phi(\eps y_2))    \lambda^2 \kappa(\d\xi_1) \kappa(\d\xi_2) \dd y_1 \dd y_2, \\
   J^{\eps,d}_2 &=  \eps^6 \int_{(\S^2)^2 \times (\R^3)^2} 1_{|\Phi(\eps y_1) - \Phi(\eps y_2)| \ge d\eps} \tilde{\Psi}_1(\Phi(\eps y_1),\Xi(\eps y_1,\xi_1)) \otimes \tilde{\Psi}_2(\Phi(\eps y_2),\Xi(\eps y_2,\xi_2))   \\
    &\hspace{2.5cm}  : (\nabla^2 G(\Phi(\eps y_1) - \Phi(\eps y_2))  \Big( \nu_2(y_1-y_2,\d\xi_1,\d\xi_2) - \lambda^2 \kappa(\d\xi_1) \kappa(\d\xi_2) \Big) \dd y_1 \dd y_2.  
\end{align*} \\[3mm]
\emph{Substep 1.1: Convergence of $J^{\eps,d}_1$.}
With the change of variables 
$$ x_1 = \Phi(\eps y_1), \quad x_2 = \Phi(\eps y_2) $$
taking into account that $\det(\na \Phi) = 1$, we get
\begin{align*}
J^{\eps,d}_1 & = \int_{(\S^2)^2 \times (\R^3)^2}  1_{|x_1 - x_2| \ge d\eps} \check \Psi_1(x_1,\xi_1) \otimes  \check \Psi_2(x_2,\xi_2) : \nabla^2 G(x_1 - x_2)   \lambda^2 dx_1 dx_2 \kappa(d\xi_1) \kappa(d\xi_2) 
\end{align*}
where
\begin{align} \label{check.Psi}
    \check \Psi_i(x,\xi) = \tilde \Psi_i (x,\Xi(\Phi^{-1}(x),\xi)).
\end{align}
Hence,
\begin{align*}
     J^{\eps,d}_1 
     \xrightarrow[\eps \rightarrow 0]{} &  \int_{(\S^2)^2 \times (\R^3)^2}   \check \Psi_1(x_1,\xi_1) \otimes  \check \Psi_2(x_2,\xi_2) : \nabla^2 G(x_1 - x_2)     \lambda^2 dx_1 dx_2 \kappa(d\xi_1) \kappa(d\xi_2) 
\end{align*}
 where the last integral over $(\R^3)^2$  has to be understood in the sense of a principal value.
This principal value is well-defined by the smoothness of $\tilde{\Psi}_1, \tilde{\Psi}_2$ in $x_1,x_2$. More generally, for any $F \in C_c^1(\R^3 \times \R^3)$, we define 
\begin{align*}
&\int_{(\R^3)^2}  \na^2 G(x_1 - x_2)  F(x_1,x_2)  \dd x_1  \dd x_2   = \lim_{\delta \rightarrow 0}  \int_{\{|x_1 - x_2|>\delta\}}  \na^2 G(x_1  - x_2)  F(x_1,x_2)  \dd x_1  \dd x_2.
\end{align*}
Through a change of variables, we arrive at
 \begin{align}
       J^{\eps,d}_1 
     \xrightarrow[\eps \rightarrow 0]{} & 
\int_{(\S^2)^2 \times (\R^3)^2}  \tilde \Psi_1(x_1,\xi_1) \otimes \tilde \Psi_2(x_1,\xi_1) : \nabla^2 G(x_1-x_2) \tilde \mu(\d x_1,\d \xi_1) \tilde \mu(\d x_2,\d\xi_2).
\end{align}\\[3mm]
\emph{Substep 1.2: Convergence of $J^{\eps,d}_2$.}
We introduce a covering $\mC$ of the compact  set
$$K = \bigcup_{\xi_1,\xi_2 \in \S^2} \supp (\tilde{\Psi}_1(\cdot,\xi_1)) \cup  \supp (\tilde{\Psi}_2(\cdot,\xi_2)) $$
by  a collection  of cubes of size $\delta$. As  the functions
\begin{align} \label{hat.Psi}
   \hat {\Psi}_i(x,\xi)  := \tilde \Psi_i (\Phi(x),\Xi(x,\xi)),
\end{align}
$i=1,2$ are Lipschitz, the oscillation of $\tilde{\Psi}_i$ on each cube $C \in \mC$ is of order $\delta$. The idea is then to approach $\hat {\Psi}_i(x,\xi)$ by $\sum_{C \in \mC} \hat {\Psi}_i(x_C, \xi) 1_{C}$, i.e. 
\begin{align*}
J^{\eps,d}_2 & =  \eps^6 \sum_{C_1,C_2 \in \mathcal{C}} \int_{(\S^2)^2 \times \eps^{-1} C_1 \times \eps^{-1} C_2} 1_{|\Phi(\eps y_1) - \Phi(\eps y_2)| \ge d\eps}  \hat {\Psi}_1(x_{C_1},\xi_1) \otimes \hat {\Psi}_2(x_{C_2},\xi_2) \\
    &\hspace{1.5cm} : \nabla^2 G(\Phi(\eps y_1) - \Phi(\eps y_2))    \Big( \nu_2(y_1-y_2,\d\xi_1,\d\xi_2) - \lambda^2 \kappa(\d\xi_1) \kappa(\d\xi_2) \Big) \dd y_1 \dd y_2 + R_{\eps,\delta}. 
\end{align*}
Using $|\nabla^2 G(x)| \lesssim |x|^{-3}$ and \eqref{mixing}, the remainder $R_{\eps,\delta}$ satisfies 
\begin{align*}
|R_{\eps,\delta}| & \lesssim  \eps^6 \, \delta \sum_{C_1,C_2 \in \mathcal{C}} \int_{\eps^{-1} C_1 \times \eps^{-1} C_2} 1_{|\Phi(\eps y_1)  - \Phi(\eps y_2)| \ge d \eps} \frac{h(y_1 - y_2)}{|\Phi(\eps y_1)  - \Phi(\eps y_2)|^3} \dd y_1 \dd y_2 \\
& \lesssim \delta \sum_{C_1,C_2 \in \mathcal{C}} \int_{C_1}  \int_{\eps^{-1} (x_1 - C_2)} 
\frac{h(y)}{1+|y|^3} \dd y \, \dd x_1 \: \lesssim \: \delta \int_{\R^3} \frac{h(y)}{1+|y|^3} \dd y = O(\delta) .
 \end{align*}
Then, making again the change of variable 
$$ x_1 = \eps y_1, \quad y = y_1 - y_2, $$
we get 
\begin{align}
  J^{\eps,d}_2 & =  \sum_{C_1,C_2 \in \mC}   \int_{C_1}  \left(\int_{\eps^{-1}(x_1 - C_2)}  \int_{(\S^2)^2}  \hat{\Psi}_1(x_{C_1},\xi_1) \otimes   \hat{\Psi}_2(x_{C_2},\xi_2) :  \nu^{d,\eps}_G(x_1,y,\d\xi_1,\d\xi_2) \dd y \right) \dd x_1 \\
  &+ O(\delta) \label{eq:J_2.parenth}
  \end{align}
where 
\begin{align*} \nu^{d,\eps}_G(x_1,y,\d\xi_1,\d\xi_2) & = \eps^3 1_{|\Phi(x_1) - \Phi(x_1 - \eps y)| \ge d\eps}\nabla^2 G(\Phi(x_1) - \Phi(x_1 - \eps y)) \\
    &\hspace{3cm}   \Big( \nu_2(y,\d\xi_1,\d\xi_2) - \lambda^2 \kappa(\d\xi_1) \kappa(\d\xi_2) \Big)
\end{align*}
The expression inside the parenthesis of \eqref{eq:J_2.parenth} defines a function $K^{\eps,d}(x_1)$. We first claim that 
 \begin{align} \label{Keps1}
            K^{\eps,d}(x_1) & \xrightarrow[\eps \rightarrow 0]{} 
       0  &\quad \text{if } \: 0 \not\in x_1-C_2   
  \end{align}
Indeed, thanks to the decay of $G$, we get for all $x_1$
$$ |K^{\eps,d}(x_1)| \lesssim \int_{\eps^{-1}(x_1 - C_2)} \frac{h(y)}{1 +|y|^3}\dd y $$
and the right-hand side goes to zero by the dominated convergence theorem: indeed, if $0 \not\in x_1-C_2$, the function $ 1_{\eps^{-1}(x_1 - C_2)} \frac{h(y)}{1 +|y|^3}$ goes to zero pointwise, and is dominated by  $\frac{h(y)}{1 +|y|^3}$, which is integrable thanks to \eqref{mixing}.

On the other hand, we notice that  $|\Phi(x_1) - \Phi(x_1 - \eps y)| \ge d \eps$ implies $y \geq d' $, for  $d' = \mathrm{Lip}(\Phi^{-1}) d$. Using also that $\nabla^2 G$ is $-3$ homogeneous, we get for $|\Phi(x_1) - \Phi(x_1 - \eps y)| \ge d \eps$ that
\begin{align*}
   \left| \eps^3 \nabla^2 G(\Phi(x_1) - \Phi(x_1 - \eps y)) - \nabla^2 G(\nabla \Phi(x_1) y)\right| \lesssim \|\nabla^2 \Phi\|_\infty \eps \lesssim \eps
\end{align*}
%\dgv{or equivalently, using that $\na^2 G$ is even, 
%$$ \left| \eps^3 \nabla^2 G(\Phi(x_1) - \Phi(x_1 + \eps y)) - \nabla^2 G(- \nabla \Phi(x_1) %y)\right| \lesssim \eps.$$}
Thus, defining 
\begin{align*} \nu^{d}_G(x_1,y,\d\xi_1,\d\xi_2) & = 1_{|\na \Phi(x_1)y| \ge d} \nabla^2 G(\na \Phi(x_1)y)  \Big( \nu_2(y,\d\xi_1,\d\xi_2) - \lambda^2 \kappa(\d\xi_1) \kappa(\d\xi_2) \Big)
\end{align*}
we get 
  \begin{align} \label{Keps2}
      K^{\eps,d}(x_1) & \xrightarrow[\eps \rightarrow 0]{} 
           \int_{\R^3} \int_{(\S^2)^2 } \hat{\Psi}_1(x_{C_1},\xi_1) \otimes  \hat{\Psi}_2(x_{C_2},\xi_2) :  \nu^d_G(x_1,y,\d\xi_1,\d\xi_2) \dd y &\: \text{if } \: 0 \in x_1-C_2.
  \end{align}
Combining this last inequality with \eqref{Keps1}, and the uniform bound:  
   $$ |K^{\eps,d}(x_1)| \red{\lesssim \int_{\eps^{-1}(x_1 - C_2)} \frac{h(y)}{1 +|y|^3}\dd y} \lesssim \int_{\R^3} \frac{h(y)}{1 +|y|^3}\dd y \lesssim 1. $$  
 we can apply the dominated convergence theorem to $\int_{C_1} K^{\eps,d}(x_1) dx_1$ to get that    
\begin{align*}&\limsup_{\eps \to 0} \Big|J^{\eps,d}_2 - \sum_{C_1,C_2 \in \mC}   \int_{C_1} 1_{C_2}(x_1) \int_{(\S^2)^2 \times \R^3} \hat{\Psi}_1(x_{C_1},\xi_1) \otimes  \hat{\Psi}_2(x_{C_2},\xi_2) \\
&\hspace{8cm} :  \nu_G^d(x_1,y,\d\xi_1,\d\xi_2) \dd y \dd x_1 \Big| =  O(\delta)  
\end{align*}
that is 
$$ \limsup_{\eps \to 0} \Big|J^{\eps,d}_2 - \sum_{C_1 \in \mC} \int_{C_1}  \int_{(\S^2)^2 \times \R^3} \tilde{\Psi}_1(x_{C_1},\xi_1) \otimes  \tilde{\Psi}_2(x_{C_1},\xi_2) :  \nu^d_G(x_1,y,\d\xi_1,\d\xi_2) \dd y \dd x_1\Big| =  O(\delta)  $$
Interchanging the limits $\delta \to 0$ and $\eps \to 0$ thanks to the uniform bound on $R_{\eps,\delta}$, we conclude
% and using $\nabla^2G(-x) = \nabla^2 G(x)$
\begin{align*} 
&\lim_{\eps \rightarrow 0} J^{\eps,d}_2 =  \int_{(\S^2)^2 \times (\R^3)^2} \hat{\Psi}_1(x,\xi_1) \otimes  \hat{\Psi}_2(x,\xi_2) : \nu^d_G(x,y,\d\xi_1,\d\xi_2) \dd  y \dd x \\
 % &=  \int_{(\S^2)^2 \times \{|y|\ge d\} \times \R^3} \hat{\Psi}_1(x,\xi_1) \otimes  \hat \Psi_2(x,\xi_2) \na^2 G(\na \Phi(x)y) \left(\nu_2(y,\dd\xi_1,\dd\xi_2)  -  \lambda^2 \kappa(\dd\xi_1) \kappa(\dd\xi_2) \right)  \dd y \dd x \\
&\xrightarrow[d \rightarrow 0]{}    \int_{(\S^2)^2 \times (\R^3)^2} \hat{\Psi}_1(x,\xi_1) \otimes  \hat{\Psi}_2(x,\xi_2)  : \na^2 G(\na \Phi(x)y) \\
& \hspace{6cm} \left(\nu_2(y,\d\xi_1,\d\xi_2) - \lambda^2 \kappa(\d\xi_1) \kappa(\d\xi_2)   \right)  \dd y \dd x 
\end{align*}
where the last  integral in $y$ has to be understood again in the sense of a principal value. 
A change of variables, finally yields
\begin{align}
\lim_{\eps \rightarrow 0} J^{\eps,d}_2 =   \int_{(\S^2)^2  \times (\R^3)^2} \tilde {\Psi}_1(x,\xi_1) \otimes \tilde \Psi_2(x,\xi_2) :  \na^2 G(y)\tilde \nu_{2}(\d x, \d y,\d\xi_1,\d\xi_2).
\end{align}

Actually, the integral 
$$\lim_{\eps \rightarrow 0} J^{\eps,d}_2 =  \int_{(\S^2)^2 \times (\R^3)^2} \hat{\Psi}_1(x,\xi_1) \otimes  \hat{\Psi}_2(x,\xi_2) : \nu^d_G(x,y,\d\xi_1,\d\xi_2) \dd  y \dd x  $$ 
is constant in  $d$, as for all $d < d' < c'$
\begin{align*}
&\Big|  \int_{(\S^2)^2 \times (\R^3)^2} \hat{\Psi}_1(x,\xi_1) \otimes  \hat{\Psi}_2(x,\xi_2) : \nu^d_G(x,y,\d\xi_1,\d\xi_2) \dd  y \dd x \\
- &  \int_{(\S^2)^2 \times (\R^3)^2} \hat{\Psi}_1(x,\xi_1) \otimes  \hat{\Psi}_2(x,\xi_2) : \nu^{d'}_G(x,y,\d\xi_1,\d\xi_2) \dd  y \dd x\Big| \\
& = \Big| \int_{(\S^2)^2 \times \{d \le |y|\le d'\} \times \R^3} \hat{\Psi}_1(x,\xi_1) \otimes  \hat{\Psi}_2(x,\xi_2) : \na^2 G(y)   \lambda^2 \kappa(\d\xi_1) \kappa(\d\xi_2)   \dd y \dd x\Big|  = 0 
\end{align*} 
using that 
$$\int_{\{d \le |y|\le d'\}} \na^2 G(y) \dd y = 0.$$ This independence with respect to $d$ is consistent with the fact that $\lim_{\eps \rightarrow 0} J^{\eps,d}_1$ is also independent of $d$, see above. This concludes the proof of \eqref{convJeps} in the case of smooth $\tilde{\Psi}_1, \tilde{\Psi}_2$. \\[3mm]
\emph{Step 2: Proof of \eqref{convJeps}--\eqref{def_tilde_Psi}.}  This follows from Step 1 through an approximation argument. For $i=1,2$ let $(\Psi^\delta_i)_{\delta > 0}$ be a sequence of smooth and compactly supported functions converging uniformly to $\Psi_i$ in a neighborhood of $\Phi(\mO)$. Let also $(\chi^\delta_{\Phi(\mO)})_{\delta > 0}$ be a sequence of  smooth functions with values in $[0,1]$, which is $1$ in $\Phi(\mO)^\delta = \{ x \in \Phi(\mO), d(x,\pa \Phi(\mO)) > \delta\}$, and $0$ outside $\Phi(\mO)^{\delta/2}$. Introducing $\tilde{\Psi}^\delta_i = \chi_{\Phi(\mO)}^\delta \Psi_i^\delta$, we find using the bilinearity of $I^\eps[\tilde{\Psi}_1, \tilde{\Psi}_2]$: 
\begin{align*}
  |  I^\eps[\tilde{\Psi}_1,\tilde{\Psi}_2] -  I^\eps[\tilde{\Psi}^\delta_1, \tilde{\Psi}^\delta_2] |
 & \le |  I^\eps[(1_{\Phi(\mO)} - \chi^\delta_{\Phi(\mO)})\Psi_1, \tilde{\Psi}_2] |  + |I^\eps[ \chi^\delta_{\Phi(\mO)}(\Psi_1 - \Psi^\delta_1),  \tilde{\Psi}_2]| \\
  & + |  I^\eps[\tilde{\Psi}^\delta_1, (1_{\Phi(\mO)} - \chi^\delta_{\Phi(\mO)})\Psi_2] |  + |I^\eps[\tilde{\Psi}^\delta_1,  \chi^\delta_{\Phi(\mO)}(\Psi_2 - \Psi^\delta_2)]|
\end{align*}
Applying a straightforward analogue of Corollary \ref{coro:bound_IN},  taking into account that 
$$ \sharp \{k, (1_{\Phi(\mO)} - \chi_{\Phi(\mO)}^\delta)(\eps X'_k) \neq 0 \}  \lesssim \delta \eps^{-3}$$
we end up with 
\begin{align} &\sup_{\eps > 0} |  I^\eps[\tilde{\Psi}_1,\tilde{\Psi}_2] -  I^\eps[\tilde{\Psi}^\delta_1, \tilde{\Psi}^\delta_2] |  \\
&\le C \left( \sqrt{\delta}  \|\Psi_1\|_\infty \|\Psi_2\|_\infty + \|\Psi_1 - \Psi_1^\delta\|_\infty \|\Psi_2\|_\infty +  \|\Psi_1\|_\infty 
\|\Psi_2 - \Psi_2^\delta\|_\infty \right) 
\end{align}
so that in particular
\begin{equation}
\lim_{\delta \rightarrow 0} \, \sup_{\eps > 0} \, |  \E I^\eps[\tilde{\Psi}_1,\tilde{\Psi}_2] -  \E I^\eps[\tilde{\Psi}^\delta_1, \tilde{\Psi}^\delta_2] | = 0 
\end{equation}
By Step 1, it remains to show that 
$$ J^0[\tilde{\Psi}^\delta_1,\tilde{\Psi}^\delta_2] \xrightarrow[\delta \rightarrow 0]{}  J^0[\tilde{\Psi}_1,\tilde{\Psi}_2] $$
We split 
$$J^0[\tilde{\Psi}^\delta_1,\tilde{\Psi}^\delta_2] = J^0_1[\tilde{\Psi}^\delta_1,\tilde{\Psi}^\delta_2] + J^0_2[\tilde{\Psi}^\delta_1,\tilde{\Psi}^\delta_2]$$
with 
\begin{align*}
    J^0_1[\tilde{\Psi}^\delta_1,\tilde{\Psi}^\delta_2] & = \int_{(\S^2)^2 \times  (\R^3)^2} \tilde{\Psi}^\delta_1(x_1,\xi_1) \otimes \tilde{\Psi}^\delta_2(x_2,\xi_2)  : \nabla^2 G(x_1 - x_2) \tilde \mu(\d x_1,\d \xi_1) \tilde \mu(\d x_2,\d\xi_2)  \\
   J^0_2[\tilde{\Psi}^\delta_1,\tilde{\Psi}^\delta_2] & = \int_{(\S^2)^2 \times  \R^3 \times \R^3}  \tilde{\Psi}^\delta_1(x,\xi_1) \otimes \tilde{\Psi}^\delta_2(x,\xi_2) : \nabla^2 G(y)  \tilde \nu_{2}(\d x, \d y,\d\xi_1,\d\xi_2), 
   \end{align*}
For  the convergence of $J^0_2$, we use the fact noticed above that for $d < c'$, 
 \begin{align*}
     J^0_2[\tilde{\Psi}_1^\delta,\tilde{\Psi}^\delta_2] & =  \int_{(\S^2)^2 \times (\R^3)^2} \hat{\Psi}_1^\delta(x,\xi_1) \otimes  \hat{\Psi}_2^\delta(x,\xi_2) : \nu^d_G(y,\d\xi_1,\d\xi_2) \dd  y \dd x
     \end{align*}
(the same with $\hat{\Psi}_1, \hat{\Psi}_2$) where $\hat \Psi_i^\delta$ is defined analogously to \eqref{hat.Psi}. We deduce, using \eqref{mixing},
\begin{align*}
& \big|J^0_2[\tilde{\Psi}_1^\delta,\tilde{\Psi}^\delta_2] -  J^0_2[\tilde{\Psi}_1,\tilde{\Psi}_2] \big| \le   
\big|J^0_2[\tilde{\Psi}_1^\delta - \tilde{\Psi}_1,\tilde{\Psi}^\delta_2] \big| + \big|
J^0_2[\tilde{\Psi}_1,\tilde{\Psi}^\delta_2 - \tilde{\Psi}_2]  \big| \\
&\le  \left(\| \hat{\Psi}_1^\delta - \hat{\Psi}_1 \|_{L^1_x L^\infty_{\xi}} \|\hat{\Psi}_2^\delta \|_{L^\infty_{x,\xi}} +  \|\hat{\Psi}_1 \|_{L^\infty_{x,\xi}} \| \hat{\Psi}_2^\delta - \hat{\Psi}_2 \|_{L^1_x L^\infty_{\xi}} \right) \int_{(\S^2)^2 
\times \R^3} |\nu_G^d(y,\d \xi_1, \d \xi_2)|\dd y  \\
&\lesssim   \left(\| \tilde{\Psi}_1^\delta - \tilde{\Psi}_1 \|_{L^1_x L^\infty_{\xi}} + \| \tilde{\Psi}_2^\delta - \tilde{\Psi}_2 \|_{L^1_x L^\infty_{\xi}} \right) \xrightarrow[\delta \rightarrow 0]{} 0 
\end{align*}
Finally, we reformulate 
\begin{align*}
J^0_1[\tilde{\Psi}^\delta_1,\tilde{\Psi}^\delta_2] = \lambda^2 \int_{(\S^2)^2} \bigg( \int_{\R^3} \check{\Psi}_1^\delta(x_1, \xi_1) \big(\na^2 G \ast \check{\Psi}^\delta_2(\cdot, \xi_2)\big)(x_1)  \dd x_1 \bigg)  \kappa(\d \xi_1) \kappa(\d \xi_2) 
\end{align*}
where $\check{\Psi}_1^\delta$ is defined analogously to \eqref{check.Psi} and
where the convolution integral 
$$\na^2 G \ast \check{\Psi}^\delta_2(\cdot, \xi_2)\big)(x_1)  = \int_{\R^3} \na^2 G(x_1 - x_2) : \check{\Psi}^\delta_2(x_2,\xi_2) \dd x_2 $$
is understood as a principal value. We now take advantage of the fact that $\na^2 G$ is a Calderon-Zygmund operator, so that the operator $\na^2 G \ast$ extends as a  continuous operator from $L^p(\R)$ to $L^p(\R)$ for all $1 < p < \infty$. In particular, as $\check{\Psi}_i^\delta(\cdot, \xi_i)$ converges to  $\check{\Psi}_i(\cdot, \xi_i)$ in $L^2(\R^3)$, uniformly in $\xi_i$ we deduce that 
$$J^0_1[\tilde{\Psi}^\delta_1,\tilde{\Psi}^\delta_2] \xrightarrow[\delta \rightarrow 0]{} J^0_1[\tilde{\Psi}_1,\tilde{\Psi}_2]  $$
where $J^0_1[\tilde{\Psi}_1,\tilde{\Psi}_2]$ is understood as 
$$J^0_1[\tilde{\Psi}_1,\tilde{\Psi}_2] = \lambda^2 \int_{(\S^2)^2} \bigg( \int_{\R^3} \check{\Psi}_1(x_1, \xi_1) \big(\na^2 G \ast \check{\Psi}_2(\cdot, \xi_2)\big)(x_1)  \dd x_1 \bigg)  \kappa(\d \xi_1) \kappa(\d \xi_2) $$
This finishes the proof.

\subsection{Proof of Theorem \ref{thm.f-f_phi}} \label{subsec:proof_final}

\emph{Step 1: Well-posedness of \eqref{f_phi}.}

The regularity assumptions on $B$ are enough to define $f_\phi$ through the flow: the characteristics
\begin{align} \label{flow.phi}
\begin{aligned}
    \partial_t \Phi_\phi(t,x,\xi) = u_\phi(t, \Phi_\phi(t,x,\xi)),  \\
    \partial_t \Xi_\phi(t,x,\xi) = M(\Xi_\phi(t,x,\xi))(\nabla u_\phi(t, \Phi_\phi(t,x,\xi))    + \phi B(t, \Phi_\phi(t,x,\xi),\Xi_\phi(t,x,\xi)) )\Xi_\phi(t,x,\xi), \\
    \Phi_\phi(0,x,\xi) = x,  \qquad \Xi_\phi(0,x,\xi) = \xi
\end{aligned}
\end{align}
are well-defined under Lipschitz continuity of $B$ with respect to $\xi$. Indeed, the ODE for $\Phi_\phi$ does not involve $\Xi_\phi$, and can be solved independently as its vector field $u_\phi$ is Lipschitz in $\Phi_\phi$. Once $\Phi_\phi$ is known, the ODE for $\Xi_\phi$ has a vector field that can be seen as bounded in time and Lipschitz in $\Xi_\phi$ so that it can be solved as well.   Uniqueness of the solution $f_\phi$  is then obtained by duality.
%\rhcomment{
%Can we significantly weaken the assumptions on $B$? Probably a bit with DiPerna-Lions theory but some %differentiability will always be required, I think.
%}
\\[3mm]
\emph{Step 2: Proof that $f_N^{app}$ is a distributional solution to 
 \begin{align}
        \partial_t f_N^{app} + u_\phi \cdot \nabla f_N^{app} + \dv_\xi(M \nabla u \xi f_N^{app}) & =    - \dv_\xi(M \nabla u_{diff} \xi f + \phi M B \xi f) + R_N^{app}, \label{f_N^app.remainder} \\
        f_N^{app}(0) &  = f_N^0 = \frac1N \sum_{i=1}^N \delta_{X_i^0,\xi_i^0} \label{f_N^app.0},
    \end{align}
    where the remainder satisfies 
    \begin{align} \label{bound.R_N}
              R_N^{app} & = o(\phi) \:  \text{ in } L^\infty(0,T; W^{-1,q'}_x W^{-2,q'}_\xi) 
    \end{align}
(see Theorem \ref{thm.f-f_phi} for a definition of this space).}

We recall from \eqref{u_diff} the notation $u_{diff} = u - u_\phi$.
Let $\varphi \in C_c^\infty([0,T) \times \R^3 \times \S^2) $.
Then, for all $s \in (0,T)$, by \eqref{X^app}, \eqref{xi^app}, with the convention $\nabla^2 G(0) = 0$ 
\begin{align}
    &\int_{\R^3 \times \S^2} \varphi(t,x,\xi) f^{app}_{N}(t,\d x, \d \xi) - \int_{\R^3 \times \S^2} \varphi(0,x,\xi) f_N^0(\d x, \d \xi) \\
    &= \int_0^s \int_{\R^3 \times \S^2} \left(\partial_t \varphi(t,x,\xi) +  u_\phi(t,x) \cdot \nabla_x \varphi(t,x,\xi) \right) f_N^{app}(t, \d x, \d \xi) \dd t \\
   &+ \int_0^s \int_{\R^3 \times \S^2} \int_{\R^3 \times \S^2}  \nabla_\xi \varphi(t,x,\xi)  \cdot  M(\xi)  \left( \nabla u(t,x) + \phi \nabla^2 G(x-y) S(\zeta) D u (t,y) \right) \xi \\
   &\hspace{8cm}  f_N^{app}(t, \d y, \d \zeta)  f_N^{app}(t, \d x, \d \xi) \dd t\\
   &= \int_0^s \int_{\R^3 \times \S^2} \left(\partial_t \varphi(t,x,\xi) +  u_\phi(t,x) \cdot \nabla_x \varphi(t,x,\xi) + \nabla_\xi \varphi(t,x,\xi)  \cdot  M(\xi)  \nabla u(t,x) \right) \\
    &\hspace{6.5cm}  f_N^{app}(t, \d x, \d \xi) \dd t 
   + \phi \int_0^s I_N[ \bar{M}\nabla_\xi \varphi,S D u] \dd t,
\end{align}
where we recall the notations $\bar{M}$ from \eqref{def_barM} and $I^N[\Psi_1,\Psi_2]$ from \eqref{def:IN}. Hence, \eqref{f_N^app.remainder}--\eqref{f_N^app.0} is satisfied with
\begin{align}
    \frac 1 \phi \langle R_N^{app}, \varphi\rangle  &:= \int_0^T I_N[ \bar{M}\nabla_\xi \varphi,S D u] -  \int_0^T\int_{\R^3 \times \S^2}  \nabla_\xi  \varphi(t,x,\xi) \cdot  M(\xi) B(t,x,\xi) \xi f(t,\d x, \d \xi) \dd t\\
    &-\int_0^T\int_{\R^3 \times \S^2} \nabla_\xi  \varphi(t,x,\xi) \cdot   M \frac{\nabla u_{diff}(t,x)}{\phi} \xi  f(t,\d x,\d \xi) \dd t .
\end{align}
Theorem \ref{cor1} and \eqref{B.indentity} imply
\begin{align}
   \lim_{N \to \infty} \frac 1 \phi \langle R_N^{app}(t), \varphi \rangle  = 0 \quad \text{for all } \varphi \in C_c^\infty(\R^3 \times \S^2).
\end{align}
By Corollary \ref{coro:bound_IN} and boundedness of $S$ and $\nabla u$, we have
\begin{align}
    \sup_{t \in [0,T]} I_N[\bar{M}\nabla_\xi \varphi,S D u]  \le C   \|\varphi(t,\cdot)\|_{L^\infty_x (W^{1,\infty}_\xi)}.
\end{align}
Arguing as in Remark \ref{rem:well.defined} shows that $\frac 1 \phi  R_N^{app}$ satisfies the same bound.
In particular, through the compact embedding $W^{1,q}_x W^{2,q}_\xi \Subset L^\infty_x (W^{1,\infty}_\xi)$
\begin{align}
  \frac 1 \phi   R_N^{app}(t) \to 0 \quad \text{in } W^{-1,q'}_x W^{-2,q'}_\xi.
\end{align}
The dominated convergence theorem implies \eqref{bound.R_N}.\\[3mm]
% \begin{align}
%     Z \Subset L^1\big(0,T ;  L^\infty_x W^{1,\infty}_\xi\big)
% \end{align}
% and hence 
% \begin{align}
%     (L^1\big(0,T ;  L^\infty_x W^{1,\infty}_\xi\big))' \Subset Z'.
% \end{align}
% Hence, $\frac 1 \phi  R_N^{app}$ converges strongly in $Z'$ and by Theorem \ref{cor1} the limit is $0$. 
% \rhcomment{Is there a more direct approach that avoids Aubin-Lions?}
% \dgvcomment{I think that for each $t$, one has $\frac{R_n(t)}{\phi}$ that goes weakly to zero in the dual of $L^\infty_x W^{1,\infty}_\xi$,  and so by compact embedding it converges strongly to zero in $W^{-1,q'}_x W^{-2,q'}_\xi$. As $R_n(t)$ is bounded uniformly on $(0,T)$ in this dual space by a constant, one can conclude by the dominated convergence theorem that it goes to zero in $L^p(0,T;W^{-1,q'}_x W^{-2,q'}_\xi)$ for any finite $p$.    }
\emph{Step 3: Proof that    
\begin{align} \label{f^app.tilde.f.phi}
    \sup_{t \in (0,T)}\|f_N^{app}(t)-\tilde f_\phi(t)\|_{W^{-1,q'}_x W^{-2,q'}_\xi} = o(\phi),
\end{align} 
where $\tilde f_\phi$ is defined as the solution to
\begin{align}
        \partial_t \tilde f_\phi + u_\phi \cdot \nabla \tilde f_\phi + \dv_\xi(M \nabla u \xi \tilde f_\phi) & =    -  \dv_\xi(M \nabla u_{diff} \xi f + \phi M B \xi f), \\
        \tilde f_\phi(0) &= f^0 .
    \end{align}
}   

% We introduce as an intermediate function the solution $\tilde f_\phi$ of
% \begin{align}
%         \partial_t \tilde f_\phi + u_\phi \cdot \nabla \tilde f_\phi + \dv_\xi(M \nabla u_0 \xi \tilde f_\phi) & =    - \phi \dv_\xi(M \nabla u_1 \xi f - \phi M B \xi f), \\
%         \tilde f_\phi(0) &= f_0 = \lim_N \frac1N \sum_{i=1}^N \delta_{X_i(0),\xi_i(0)}.
%     \end{align}
%     Let us show first that 
%     $$ \sup_{t \in (0,T)}\|f_N^{app}(t)-\tilde f_\phi(t)\|_{W^{-1,q}_x W^{-2,q}_\xi} = o(\phi)$$   
Given $\varphi \in W^{1,q}_x W^{2,q}_\xi$ and some fixed time $s$, we introduce the solution $\overline{\varphi}$ of the backwards  transport equation
\begin{align}
        \partial_t \overline{\varphi} + u_\phi \cdot \nabla_x \overline{\varphi} + (M \nabla u \xi)  \cdot \na_\xi \overline{\varphi}  &= 0,\\
        \overline{\varphi} (s) & = \varphi
\end{align}
We find, denoting $g_N=f_{N}^{app}-\tilde f_\phi$
\begin{align}
\int_{\R^3 \times \S^2} g_N(s) \varphi(s) = \int_{\R^3 \times \S^2} (f_N^0 - f^0) \overline{\varphi}(0) +  \int_0^s \langle R_N^{app}(t), \overline{\varphi}(t) \rangle \dd t 
\end{align}
Moreover, one can express $\overline{\varphi}$ in terms of $\varphi$ and of the flow $\big(\Phi_\phi(t,t',x,\xi), \Xi(t,t',x,\xi)\big)$  of the field $(x,\xi) \rightarrow \big(u_\phi(x), M(\xi) \na u(x) \xi \big)$: 
$$ \overline{\varphi}\big(t,\Phi_\phi(t,s,x,\xi), \Xi(t,s,x,\xi)\big) = \varphi(x,\xi).  $$
Note that $\Phi_\phi$ does not actually depend on $\xi$. It has the same $W^{1,\infty}$ regularity in $x$ as $u_\phi$, see \eqref{regularity.u}. Similarly,  $\Xi$ has the same  $W^{1,\infty}$ regularity in $x$ as $u_\phi$  and $\na u$ (see again \eqref{regularity.u})  and is smooth in $\xi$. It follows that, for all $r$, 
$$\sup_{0 \le t \le T}\|\overline{\varphi}(t)\|_{W^{1,r}_x W^{2,r}_\xi}  \le C \|\varphi\|_{W^{1,r}_x W^{2,r}_\xi}.  $$
Hence, \eqref{f^app.tilde.f.phi} follows from this estimate with $r=q$, from our assumption on $f^0_N - f^0$ and \eqref{bound.R_N}.\\[3mm]
\emph{Step 4: Proof that 
 \begin{align} \label{f.phi.tilde.f.phi}
 \sup_{t \in (0,T)}\|f_\phi(t)-\tilde f_\phi(t)\|_{W^{-1,q'}_x W^{-2,q'}_\xi} = o(\phi).
 \end{align} 
}

We write 
\begin{align}
        \partial_t  (f-f_\phi) + u_\phi \cdot \nabla (f-f_\phi) & + \dv_\xi(M \nabla u \xi  (f-f_\phi))    \\
        & = u_{diff} \cdot \nabla_x f  +  \dv_\xi(M \nabla u_{diff} \xi f_\phi + \phi M B \xi f_\phi), \\
       (f-f_\phi)(0) &= 0.
    \end{align}
The right-hand side is $O(\phi)$ in $W^{-1,q'}_x W^{-1,q'}_\xi$ so by duality as in Step 3 we get
\begin{align} \label{f.f.phi}
     \sup_{t \in (0,T)} \|f{(t)}-f_\phi(t)\|_{W^{-1,q'}_x W^{-1,q'}_\xi} \lesssim \phi.
\end{align}

 The function $g_\phi = f_\phi-\tilde f_\phi$ satisfies 
 \begin{align}
        \partial_t  g_\phi + u_\phi \cdot \nabla g_\phi + \dv_\xi(M \nabla u_0 \xi g_\phi) & =   \phi \dv_\xi(M \nabla u_1 \xi (f-f_\phi) + \phi M B \xi (f-f_\phi)), \\
       g_\phi(0) &= 0.
    \end{align}
By \eqref{f.f.phi} the right-hand side is  $O(\phi^2)$ in $ W^{-1,q'}_x W^{-2,q'}_\xi$. Hence, reproducing the arguments of Step 2 yields \eqref{f.phi.tilde.f.phi}.\\[3mm]
\emph{Step 5: Conclusion.}
    We claim that
    \begin{align} \label{W_infty.bounds.negative.sobolev}
       \|f_N(t)-f_N^{app}(t)\|_{W^{-1,q'}_x W^{-2,q'}_\xi} \lesssim  W_\infty(f_N(t),f_N^{app}(t)) + N^{-1/3 + 1/q}
    \end{align}
    Then, we can bound
   \begin{align}
       \|f_N(t)-f_\phi(t)\|_{W^{-1,q'}_x W^{-2,q'}_\xi} &\leq \sup_{t \in (0,T)} \mathcal W_\infty(f_N(t),f_N^{app}(t)) + N^{-1/3 + 1/q} \\
       &+ \|f_N^{app}(t)-\tilde f_\phi(t)\|_{W^{-1,q'}_x W^{-2,q'}_\xi} + \|f_\phi(t)-\tilde f_\phi(t)\|_{W^{-1,q'}_x W^{-2,q'}_\xi}
   \end{align}  
     and inserting estimates \eqref{W_infty.app}, \eqref{f^app.tilde.f.phi} and \eqref{f.phi.tilde.f.phi} and using the assumption $N^{-1/3 + 1/q}= o(\phi)$ yields the desired estimate.

It remains to prove \eqref{W_infty.bounds.negative.sobolev}. Fix $t > 0$ which we suppress in the following. 
Let $c' > 0$ be chosen such that $|X_i - X_j| \geq c'$ and $|X_i^{app} - X_j^{app}| \geq c'$ for all $i \leq j$ which is possible thanks to \eqref{ass:separation}, \eqref{dmin.control} and \eqref{dmin.control.app}.
For $\varphi \in C_c^\infty(\R^3 \times \S^2)$ we estimate
\begin{align} \label{f_N-f_N^app.negative.0}
\begin{aligned}
    \Big|\int \varphi \dd (f_N - f_N^{app})\Big| &= \Big|\frac 1 N \sum_i (\varphi(X_i,\xi_i) - \varphi(X_i^{app},\xi_i^{app})\Big| \\
    & \lesssim \Big|\frac 1 N \sum_i (\varphi(X_i,\xi_i) - \varphi(X_i^{app},\xi_i)\Big| +\Big|\frac 1 N \sum_i (\varphi(X_i^{app},\xi_i) - \varphi(X_i^{app},\xi_i^{app})\Big| \\
   & \lesssim \mathcal W_\infty( f_N, f_N^{app}) \| \varphi\|_{W^{1,\infty}_\xi(L^
\infty_x)} +  N^{-1/3 + 1/q} \|\varphi \|_{ L^\infty_\xi(C_x^{0,1- 3/q})} \\
   &+ \Big|\frac 1 N \sum_i \fint_{B(X_i,c' N^{-1/3})} \varphi(x,\xi_i) \dd x - \fint_{B(X_i^{app},c' N^{-1/3})} \varphi(x,\xi_i) \dd x\Big|.
   \end{aligned}
\end{align}
By the embedding $W^{1,q}_x W^{2,q}_\xi \subset W^{1,\infty}_\xi(L^
\infty_x)$ as well as $W^{1,q}_x W^{2,q}_\xi \subset W^{1,q}_\xi(W^{1,q}_x) \subset L^\infty_\xi(C_x^{0,1- 3/q})$ we bound 
\begin{align}  \label{f_N-f_N^app.negative.1}
    \mathcal W_\infty( f_N, f_N^{app}) \| \nabla_\xi \varphi\|_\infty +  N^{-1/3 + 1/q} \|\varphi \|_{C^{0,1- 3/q}} \lesssim \left( W_\infty( f_N, f_N^{app})  +  N^{-1/3 + 1/q} \right) \|\varphi\|_{W^{1,q}_x W^{2,q}_\xi}. \qquad 
\end{align}

Consider the measures $\bar \rho_N, \bar \rho_N^{app} \in \mathcal P(\R^3)$ defined as
\begin{align}
    \bar \rho_N := \frac 1 {c'^3} \sum_i \1_{B(X_i,c' N^{-1/3})}, \qquad 
    \bar \rho_N^{app} : = \frac 1 {c'^3}  \sum_i \1_{B(X_i^{app},c' N^{-1/3})}.
\end{align}
 Then $\|\bar \rho^N\|_\infty + \|\bar \rho_N^{app}\|_\infty \lesssim 1$. 
In particular, there exists an ($\bar \rho_N$ a.e. unique) optimal transport map $T \colon \R^3 \to \R^3$ for $\mathcal W_{q'}(\bar \rho_N,\bar \rho_N^{app})$, i.e.\ $T$ satisfies $\bar \rho_N^{app} = T\# \bar \rho_N$ and
\begin{align} \label{W_infty.bar}
   \mathcal W_{q'}(\bar \rho_N,\bar \rho_N^{app})  = \bigg(\int_{\R^3}  |x - T(x)|^{q'}  \dd \bar \rho_N \bigg)^{1/{q'}}&\leq  \mathcal W_\infty(\bar \rho_N, \bar \rho_N^{app}) \\
   &\leq \mathcal W_\infty(\bar \rho_N, \rho_N) + \mathcal  W_\infty(\bar \rho_N^{app}, \rho_N^{app}) + \mathcal W_\infty( \rho_N, \rho_N^{app})  \\
   & \leq 2 c' N^{-1/3} +  \mathcal W_\infty( f_N, f_N^{app}).
\end{align}
Moreover, for $\theta \in [0,1]$, let $\nu_\theta$ be the constant-speed geodesic connecting $\nu_0 = \bar \rho_N,\nu_1 =\bar \rho_N^{app} $ given through
\begin{align}
 \nu_\theta = (\theta T + (1- \theta) \Id)\# \bar \rho_N.
\end{align}
By \cite[Proposition 5.27 and Proposition 7.29]{Santambrogio15}, we have for all $\theta \in [0,1]$
\begin{align} \label{nu_theta}
    \|\nu_\theta\|_\infty \leq \max\{ \|\bar \rho_N\|_\infty,  \|\bar \rho^{app}_N\|_\infty\} \lesssim 1.
\end{align}
We then estimate the last right-hand side term in \eqref{f_N-f_N^app.negative.0}
\begin{align}
    & \bigg|\frac 1 N \sum_i \fint_{B(X_i,c' N^{-1/3})} \varphi(x,\xi_i) \dd x - \fint_{B(X_i^{app},c' N^{-1/3})} \varphi(x,\xi_i) \dd x\bigg| \\
    & \leq \int_{\R^3} |\sup_\xi \varphi(x,\xi) - \varphi(T(x),\xi)| \dd \bar \rho_N(x) \\
    & =  \int_{\R^3} |\sup_\xi \int_0^1 (x - T(x)) \cdot \nabla_x \varphi(\theta x + (1- \theta) T(x),\xi) \dd \theta | \dd \bar  \rho_N (x) \\
    & \leq \bigg(\int_{\R^3}  |x - T(x)|^{q'}  \dd \bar \rho_N \bigg)^{1/{q'}} \int_0^1 \bigg(\int_{\R^3} \sup_\xi   |\nabla_x \varphi (x,\xi) |^q \dd \nu_\theta(x) \dd \theta\bigg)^{1/q} \\
    & \lesssim  \mathcal W_{\red{q'}}(\bar \rho_N, \bar \rho_N^{app}) \|\nabla_x \varphi\|_{L^{q}_x(L^\infty_\xi)} \sup_\theta \|\nu_\theta\|_\infty.
\end{align}
Combining this inequality with \eqref{nu_theta}, \eqref{W_infty.bar}, \eqref{f_N-f_N^app.negative.0} and \eqref{f_N-f_N^app.negative.1} yields \eqref{W_infty.bounds.negative.sobolev}.
%\rhcomment{The last lines literally follow the proof of the fact $\|\nu_1 - \nu_0\|_{W^{-1,p}} \lesssim %\mathcal W_p(\nu_0,\nu_1) \max\{\nu_0\|_\infty, \|\nu_1\|_\infty \}$ in \cite[Proposition 5.1]%{Hofer&Schubert}. Is there a  way to use  this statement instead of going through the proof to see %that $\sup_\xi$ does not cause any trouble?}
% Then,
% \begin{align}
%     \langle\varphi, f^N - f_N^{app}\rangle &= \int_{(\R^3 \times S^2)^2} \varphi(x,\xi) - \varphi(y,\zeta) \dd \gamma(x,\xi,y,\zeta) \\
%     &= \int_{(\R^3 \times S^2)^2} \varphi(x,\xi) - \varphi(x,\zeta) \dd \gamma(x,\xi,y,\zeta) \\
%     &+ \int_{(\R^3 \times S^2)^2} \varphi(x,\zeta) - \varphi(y,\zeta) \dd \gamma(x,\xi,y,\zeta) \\
%     & \leq \int_{(\R^3 \times S^2)^2} |\xi - \zeta| \dd \gamma(x,\xi,y,\zeta) \|\varphi\|_{L^\infty_x W^{1,\infty}_\xi} \\ 
%     &+ \int_{\R^3} \varphi \dd (\rho_N - \rho_N^{app}) \\
%     & \lesssim \mathcal W_1(f_N,f_N^{app}) \|\varphi\|_{W^{1,q}_x W^{2,q}_\xi} + \int_{\R^3} \int_{S^2} \varphi  \dd \xi \dd (\rho_N - \rho_N^{app})
% \end{align}

\appendix

\section{Quenched two-scale convergence} \label{appA}
The goal of this section is to provide a few elements of the theory of quenched stochastic two-scale convergence, introduced in \cite{ZhiPia2006}. It has to be distinguished from other notions of two-scale convergence, notably the so-called two-scale stochastic convergence in the mean of \cite{BoMiWr}. See \cite{HeNeVa2022} for a review. We only focus on the properties that are of direct use in this paper.  Let $\Omega$ a compact metric space, and $P$ a probability measure on $(\Omega, \mB(\Omega))$. We assume that $\Omega$ is equipped with a stationary ergodic shift $(\tau_y)_{y \in \R^3}$. We call typical realization an element $\tilde \omega \in \Omega$ such that property \eqref{convergence_ergodic} holds for all $g \in C(\Omega)$ and for any bounded Borel set $A$ with $|A| > 0$. The set of typical elements is of full measure, combining the ergodic theorem and the separability of $C(\Omega)$. 
\begin{definition} \label{def:2scale}
Let $\tilde \omega$ a typical realization (fixed). Let $(u_\eps = u_\eps(x))_{\eps> 0}$ in $L^2_{loc}(\R^3)$, and $U = U(x,\omega)$ in  $L^2_{loc}(\R^3, L^2(\Omega))$. We say that $u_\eps$ two-scale converges to $U$ (under the realization $\tilde\omega$)  if for all $b \in L^2(\Omega)$, for all $\varphi \in C^\infty_c(\R^3)$,   
$$\int_{\R^3} u^\eps(x) \varphi(x) b(\tau_{x/\eps}\tilde \omega) \dd x \xrightarrow[\eps \rightarrow 0]{}  \int_{\Omega}\int_{\R^3} U(x,\omega) \varphi(x) b(\omega) \dd x \dd P(\omega) $$ 
(notation : $u_\eps \xrightarrow[]{2} U$)
\end{definition}
The main compactness property of two-scale convergence is provided by 
\begin{theorem}
Let $\tilde \omega$ a typical realization and $(u_\eps)_{\eps > 0}$ a sequence bounded in $L^2_{loc}(\R^3)$. There exists $U \in L^2_{loc}(\R^3, L^2(\Omega))$ and a subsequence in $\eps$ along which 
$$u_\eps \xrightarrow[]{2} U$$ 
(under the realization $\tilde \omega$). 
\end{theorem}
\begin{itemize}
\item  We insist that in the previous theorem, $U$ and the subsequence depend \emph{a priori} on $\tilde \omega$. 
\item By taking $b=1$, one sees immediately, that along the same subsequence, $u_\eps$ converges weakly in $L^2_{loc}$ to $\E U$.
\end{itemize}
\begin{proposition}
Let $\tilde \omega$ a typical realization and $(u_\eps)_{\eps > 0}$ a sequence  in $H^1_{loc}(\R^3)$, satisfying for any bounded $K$
$$\|u_\eps\|_{L^2(K)}  \le C(K, \tilde \omega), \quad \eps \| \na u_\eps\|_{L^2(K)} \rightarrow 0$$
 Then, 
 $$ u^\eps  \xrightarrow[]{2} u^0 $$
 along a subsequence, where $u_0 = u_0(x) \in L^2_{loc}(\R^3)$ does not depend on $\omega$. 
\end{proposition}
In our setting, we shall consider a sequence that is bounded in $H^1_{loc}(\R^3)$. In such case, one can get more than in the previous theorem. We first need to remind the notion of stochastic gradient. We say that $b = b(\omega)$ belongs to $C^1(\Omega)$ if  for all $\omega$, the map
$$y \rightarrow b(\tau_y \omega)$$
is differentiable at $0$, and its gradient  at zero, denoted $\na_\omega b$ is continuous on $\Omega$. One can then define 
$$ L^2_{pot}(\Omega) = \text{ the closure of } \{ \na_\omega b, \, b \in C^1(\Omega) \} \text{ in } \: L^2(\Omega)$$
We shall use the following result: 
\begin{proposition}
Let $\tilde \omega$ a typical realization and $(u_\eps)_{\eps > 0}$ a sequence bounded  in $H^1_{loc}(\R^3)$. Then, there exists 
$$u_0 = u_0(x) \in L^2_{loc}(\R^3),$$  
independent of $\omega$, and 
$$D_U = D_U(x,\omega) \in  L^2_{loc}(\R^3, L^2_{pot}(\Omega))$$
such that along a subsequence in $\eps$ 
\begin{align*}
    & u_\eps \xrightarrow[]{2} u_0, \quad \na u_\eps   \xrightarrow[]{2} \na u_0 + D_U
\end{align*}
\end{proposition}

\section{Proof of Lemma \ref{lem:Disintegration}}
\label{sec:App.Disintegration}

    % The point is show that we can write this term in the form of $ \int_{\R^3 \times \S^2} \nabla_\xi \varphi : (M B \xi)  f_\phi$, or more precisely $\int_{\R^3 \times \S^2}  \nabla_\xi \varphi : (M B \xi)  f$. 
     First, we claim that the (non-negative) measure
\begin{equation} \label{measure_correlation}
    |1_{\mO}(x) \dd x  \, \nu_2(y,\d \xi_1,\d \xi_2) \d y - \lambda^2  1_{\mO}(x) \dd x \dd y \, \kappa(\d \xi_1)  \kappa(\d \xi_2)|
    \end{equation} 
    has a projection on $\R^3_x \times \S^2_{\xi_1}$ that is  absolutely continuous with respect to $f^0(\d x,\d\xi_1) = \frac{1}{|\mO|}1_\mO(x) \d x \kappa(\d\xi_1)$. Admitting this claim temporarily, it follows that $|\nu_{2,t}|$ has a projection on $\R^3_x \times \S^2_{\xi_1}$ which is absolutely continuous with respect to $(\Phi(t,x),\Xi(t,x,\xi_1)) \# f^0 = f(t, \cdot)$. The same holds for $\Big|\na^2 G(y) S(\xi_2) Du(x)  \nu_{2,t}\Big|$, or more precisely $\Big| \na^2 G(y) 1_{|y| \ge \delta} S(\xi_2) Du(x) \nu_{2,t} \Big|$, see Remark \ref{rem:well.defined}. By the disintegration theorem and the Radon-Nikodym theorem, it follows that there exists a matrix field $B = B(t,x,\xi_1)$ such that
\begin{align} \label{B.indentity}
   & \int_{(\S^2)^2  \times (\R^3)^2} \Big(\bar M(\xi_1) \nabla_\xi \varphi(x,\xi_1) \Big) : \Big( \na^2 G(y) S(\xi_2) Du(x) \Big) \nu_{2,t}(\d x, \d y,\d\xi_1,\d\xi_2)\\
   = & \int_{\R^3 \times \S^2}   \Big(\bar M(\xi_1) \nabla_\xi \varphi(x,\xi_1)\Big) : B(t,x,\xi_1)   f(t,\d x, \d\xi_1) \\
   = &  \int_{\R^3 \times \S^2}   \nabla_\xi \varphi(x,\xi_1) : \big(M(\xi_1) B(t,x,\xi_1) \xi_1 \big) f(t,\d x, \d\xi_1) .
\end{align}
It remains to prove the claim about the measure in \eqref{measure_correlation}. Clearly, it is enough to show that the projection on $\S^2_{\xi_1}$ of 
$$  | \nu_2(y,\d \xi_1,\d \xi_2) \d y - \lambda^2 \dd y \, \kappa(\d \xi_1)  \kappa(\d \xi_2)|$$
is absolutely continuous with respect to $\kappa(d\xi_1)$. 
By standard approximation, it is enough to prove that for any non-negative $\varphi \in C^\infty_c(\R^3)$, 
$$ \varphi(y) | \nu_2(y,\d \xi_1,\d \xi_2) \d y - \lambda^2  \dd y \, \kappa(\d \xi_1)  \kappa(\d \xi_2)|$$
satisfies the same property. It is then clearly enough to show that the same property holds for the (non-negative) measure
$$ \varphi(y) \, \nu_2(y,\d \xi_1,\d \xi_2) \d y $$
This is in turn equivalent to showing that for any $y_1 \in \R^3$, for any non-negative $\psi \in C^\infty_c(\R^3)$, this property holds for $\int_{\R^3} \psi(y_1) \varphi(y_1 - y)  \nu_2(y, \cdot) \dd y_1$. 
But, for any non-negative continuous function $h = h(\xi_1)$ on $\S^2$, we find 
\begin{align*}
  &  \int_{\S^2 \times \S^2 \times \R^3 \times \R^3} h(\xi_1) \psi(y_1) \varphi(y_1 - y)  \nu_2(y, \d \xi_1, \d \xi_2) \dd y \dd y_1 \\
   = &  \int_{\S^2 \times \S^2 \times \R^3 \times \R^3} h(\xi_1) \psi(y_1) \varphi(y_2) \nu_2(y_1 - y_2, \d \xi_1, \d \xi_2) \dd y_2 \dd y_1 \\
    = & \int_{\R^3 \times \R^3  \times  \S^2 \times \S^2} h(\xi_1) \psi(y_1) \varphi(y_2) \mu_2(\d y_1 , \d y_2, \d \xi_1, \d \xi_2) \\
& =  \E \, \sum_{i \neq j}  h(\xi'_i) \psi(X'_i) \, \varphi(X'_j) \,   \\
& \le \|\varphi\|_\infty \,  \E \, \text{card}\big(\{j, X'_j \in \supp \varphi \}\big)  \sum_i h(\xi'_i) \psi(X'_i)  \\
& \le C  \,  \E \,  \sum_i h(\xi'_i) \psi(X'_i) = C \lambda \int_{\R^3} \psi(y_1) dy_1 \int_{\S^2} h(\xi_1) \kappa(d\xi_1).   
\end{align*}
where we have used the hardcore condition on the point process to bound the cardinal by a deterministic constant. This shows that the linear form on $C(\S^2)$ 
$$h \mapsto C' \int_{\S^2} h(\xi_1) \kappa(d\xi_1) - \int_{\S^2 \times \S^2 \times \R^3 \times \R^3} h(\xi_1) \psi(y_1) \varphi(y_1 - y)  \nu_2(y, \d \xi_1, \d \xi_2) \dd y \dd y_1   $$
where $C' := C \lambda \int_{\R^3} \psi(y_1) dy_1$ is non-negative. Hence, by the Riesz representation theorem \cite[Theorems 2.14 and 6.19]{Rudin}, it can be uniquely identified with a regular Borel measure, and this measure is nonnegative. Since the projection $\int_{\R^3} \psi(y_1) \varphi(y_1 - y)  \nu_2(y, \cdot) \dd y_1$ onto $\S^2_{\xi_1}$ as well as $\kappa(d\xi_1)$ are Borel measures, and all Borel measures on $\S^2$ are regular (see \cite[Theorem 2.18]{Rudin}), it follows that $C'\kappa$ dominates the projection of $\int_{\R^3} \psi(y_1) \varphi(y_1 - y)  \nu_2(y, \cdot) \dd y_1$ onto $\S^2_{\xi_1}$. It follows that the latter is absolutely continuous with respect to the former, which concludes the proof.

\emergencystretch=1em
 \begin{refcontext}[sorting=nyt]
\printbibliography

@article{HoeferLeocataMecherbet22,
    AUTHOR = {H\"{o}fer, Richard M. and Leocata, Marta and Mecherbet, Amina},
     TITLE = {Derivation of the viscoelastic stress in {S}tokes flows
              induced by nonspherical {B}rownian rigid particles through
              homogenization},
   JOURNAL = {Pure Appl. Anal.},
  FJOURNAL = {Pure and Applied Analysis},
    VOLUME = {5},
      YEAR = {2023},
    NUMBER = {2},
     PAGES = {409--460},
      ISSN = {2578-5885,2578-5893},
   MRCLASS = {76M50 (35Q70 35R60 76D07)},
  MRNUMBER = {4605901},
       DOI = {10.2140/paa.2023.5.409},
       URL = {https://doi.org/10.2140/paa.2023.5.409},
}

@article {ZhiPia2006,
    AUTHOR = {Zhikov, V. V. and Pyatnitskiĭ, A. L.},
     TITLE = {Homogenization of random singular structures and random
              measures},
   JOURNAL = {Izv. Ross. Akad. Nauk Ser. Mat.},
  FJOURNAL = {Izvestiya Rossiiskoi Akademii Nauk. Seriya Matematicheskaya},
    VOLUME = {70},
      YEAR = {2006},
    NUMBER = {1},
     PAGES = {23--74},
      ISSN = {1607-0046,2587-5906},
   MRCLASS = {35B27 (60G57 74Q05)},
  MRNUMBER = {2212433},
MRREVIEWER = {Taras\ A.\ Mel\cprime nyk},
       DOI = {10.1070/IM2006v070n01ABEH002302},
       URL = {https://doi.org/10.1070/IM2006v070n01ABEH002302},
}

@article {BoMiWr,
    AUTHOR = {Bourgeat, Alain and Mikeli\'{c}, Andro and Wright, Steve},
     TITLE = {Stochastic two-scale convergence in the mean and applications},
   JOURNAL = {J. Reine Angew. Math.},
  FJOURNAL = {Journal f\"{u}r die Reine und Angewandte Mathematik. [Crelle's
              Journal]},
    VOLUME = {456},
      YEAR = {1994},
     PAGES = {19--51},
      ISSN = {0075-4102,1435-5345},
   MRCLASS = {35B27 (35R60 60H15 73B27)},
  MRNUMBER = {1301450},
}

@book {Guenter67,
    AUTHOR = {G\"unter, N. M.},
     TITLE = {Potential theory and its applications to basic problems of
              mathematical physics},
      NOTE = {Translated from the Russian by John R. Schulenberger},
 PUBLISHER = {Frederick Ungar Publishing Co., New York},
      YEAR = {1967},
     PAGES = {xi+338},
   MRCLASS = {31.00 (69.00)},
  MRNUMBER = {222316},
}

@book{Ladyzhenskaya69,
  title={The mathematical theory of viscous incompressible flow},
  author={Ladyzhenskaya, Olga Aleksandrovna},
  journal={Gordon \& Breach},
  year={1969}
}

@book{JKO2012,
  title={Homogenization of differential operators and integral functionals},
  author={Jikov, Vasili Vasilievitch and Kozlov, Sergei M and Oleinik, Olga Arsen'evna},
  year={2012},
  publisher={Springer Science \& Business Media}
}

@article {HeNeVa2022,
    AUTHOR = {Heida, Martin and Neukamm, Stefan and Varga, Mario},
     TITLE = {Stochastic two-scale convergence and {Y}oung measures},
   JOURNAL = {Netw. Heterog. Media},
  FJOURNAL = {Networks and Heterogeneous Media},
    VOLUME = {17},
      YEAR = {2022},
    NUMBER = {2},
     PAGES = {227--254},
      ISSN = {1556-1801,1556-181X},
   MRCLASS = {74Q05 (47J30 49J45)},
  MRNUMBER = {4421524},
MRREVIEWER = {Jun\ Geng},
       DOI = {10.3934/nhm.2022004},
       URL = {https://doi.org/10.3934/nhm.2022004},
}

@article {Hoefer19,
    AUTHOR = {H\"{o}fer, Richard M.},
     TITLE = {Convergence of the method of reflections for particle
              suspensions in {S}tokes flows},
   JOURNAL = {J. Differential Equations},
  FJOURNAL = {Journal of Differential Equations},
    VOLUME = {297},
      YEAR = {2021},
     PAGES = {81--109},
      ISSN = {0022-0396,1090-2732},
   MRCLASS = {76T20 (76D07)},
  MRNUMBER = {4278120},
MRREVIEWER = {Gangavamsam\ P.\ Raja Sekhar},
       DOI = {10.1016/j.jde.2021.06.020},
       URL = {https://doi.org/10.1016/j.jde.2021.06.020},
}

@book {Daley.Jones.book,
    AUTHOR = {Daley, D. J. and Vere-Jones, D.},
     TITLE = {An introduction to the theory of point processes. {V}ol. {I}: Elementary theory and methods},
    SERIES = {Probability and its Applications (New York)},
   EDITION = 2,
 PUBLISHER = {Springer-Verlag, New York},
      YEAR = {2003},
     PAGES = {xxii+469},
      ISBN = {0-387-95541-0},
   MRCLASS = {60-01 (60G10 60G55 60K05 60K35)},
  MRNUMBER = {1950431},
MRREVIEWER = {Volker Schmidt},
}

@book {Serfaty15,
    AUTHOR = {Serfaty, Sylvia},
     TITLE = {Coulomb gases and {G}inzburg-{L}andau vortices},
    SERIES = {Zurich Lectures in Advanced Mathematics},
 PUBLISHER = {European Mathematical Society (EMS), Z\"urich},
      YEAR = {2015},
     PAGES = {viii+157},
      ISBN = {978-3-03719-152-1},
   MRCLASS = {82B05 (82B21 82B26 82D05 82D25 82D55)},
  MRNUMBER = {3309890},
       DOI = {10.4171/152},
       URL = {https://doi.org/10.4171/152},
}

@book {Daley.Jones.book2,
    AUTHOR = {Daley, D. J. and Vere-Jones, D.},
     TITLE = {An introduction to the theory of point processes. {V}ol. {II}: General theory and structure},
    SERIES = {Probability and its Applications (New York)},
   EDITION = 2,
 PUBLISHER = {Springer, New York},
      YEAR = {2008},
     PAGES = {xviii+573},
      ISBN = {978-0-387-21337-8},
   MRCLASS = {60G55 (60-02 60G57)},
  MRNUMBER = {2371524},
MRREVIEWER = {Gail Ivanoff},
       DOI = {10.1007/978-0-387-49835-5},
       URL = {https://doi.org/10.1007/978-0-387-49835-5},
}

@article {Gerard-VaretMecherbet20,
    AUTHOR = {G\'{e}rard-Varet, David and Mecherbet, Amina},
     TITLE = {On the correction to {E}instein's formula for the effective
              viscosity},
   JOURNAL = {Ann. Inst. H. Poincar\'{e} C Anal. Non Lin\'{e}aire},
  FJOURNAL = {Annales de l'Institut Henri Poincar\'{e} C. Analyse Non
              Lin\'{e}aire},
    VOLUME = {39},
      YEAR = {2022},
    NUMBER = {1},
     PAGES = {87--119},
      ISSN = {0294-1449,1873-1430},
   MRCLASS = {35Q35 (35Q70 76D07 76M50 76T20)},
  MRNUMBER = {4412065},
MRREVIEWER = {Boqiang\ L\"{u}},
       DOI = {10.4171/aihpc/3},
       URL = {https://doi.org/10.4171/aihpc/3},
}

@article{Mecherbet19,
  title={Sedimentation of particles in Stokes flow},
  author={Mecherbet, Amina},
  journal={Kinetic and Related Models},
  year={2019},
  pages={995--1044},
  volume={12},
 number={5}
}

@article {Hofer18MeanField,
    AUTHOR = {H{\"o}fer, Richard M.},
     TITLE = {Sedimentation of inertialess particles in {S}tokes flows},
   JOURNAL = {Comm. Math. Phys.},
  FJOURNAL = {Communications in Mathematical Physics},
    VOLUME = {360},
      YEAR = {2018},
    NUMBER = {1},
     PAGES = {55--101},
      ISSN = {0010-3616},
   MRCLASS = {76T20 (35Q35 76D07)},
  MRNUMBER = {3795188},
       DOI = {10.1007/s00220-018-3131-y},
       URL = {https://doi.org/10.1007/s00220-018-3131-y},
}

@book{Rudin,
    AUTHOR = {Rudin, Walter},
     TITLE = {Real and complex analysis},
   EDITION = {Third},
 PUBLISHER = {McGraw-Hill Book Co., New York},
      YEAR = {1987},
     PAGES = {xiv+416},
      ISBN = {0-07-054234-1},
   MRCLASS = {00A05 (26-01 30-01 46-01)},
  MRNUMBER = {924157},
}

@article {Gerard-Varet21,
    AUTHOR = {G\'{e}rard-Varet, David},
     TITLE = {Derivation of the {B}atchelor-{G}reen formula for random
              suspensions},
   JOURNAL = {J. Math. Pures Appl. (9)},
  FJOURNAL = {Journal de Math\'{e}matiques Pures et Appliqu\'{e}es. Neuvi\`eme S\'{e}rie},
    VOLUME = {152},
      YEAR = {2021},
     PAGES = {211--250},
      ISSN = {0021-7824},
   MRCLASS = {76T20},
  MRNUMBER = {4280836},
MRREVIEWER = {Anatoliy Prykarpatsky},
       DOI = {10.1016/j.matpur.2021.05.002},
       URL = {https://doi.org/10.1016/j.matpur.2021.05.002},
}

@article{HoferSchubert23,
author={H\"ofer, Richard and Schubert, Richard},
 doi = {10.48550/ARXIV.2302.04637},
    url = {https://arxiv.org/abs/2302.04637},
  title = {Sedimentation of particles with very small inertia I: convergence to the transport-Stokes equations},
  publisher = {arXiv},
  year = {2023}
}

@article{NiethammerSchubert19,
  title={A local version of Einstein's formula for the effective viscosity of suspensions},
  author={Niethammer, Barbara and Schubert, Richard},
  journal={SIAM J. Math. Anal.},
  volume={52},
  number={3},
  pages={2561-–2591},
  year={2020}
}

@article{Hofer&Schubert,
  title={The influence of Einstein’s effective viscosity on sedimentation at very small particle volume fraction},
  author={H\"ofer, Richard and Schubert, Richard},
 journal = {Annales de l'Institut Henri Poincaré C, Analyse non linéaire},
volume = {38},
number = {6},
pages = {1897-1927},
year = {2021},
issn = {0294-1449},
doi = {https://doi.org/10.1016/j.anihpc.2021.02.001},
}

@article{HillairetWu19,
  title={Effective viscosity of a polydispersed suspension},
  author={Hillairet, Matthieu and Wu, Di},
journal = {Journal de Math\'ematiques Pures et Appliqu\'ees},
volume = {138},
pages = {413-447},
year = {2020},
issn = {0021-7824},
doi = {https://doi.org/10.1016/j.matpur.2020.03.001},
}

@article{Gerard-VaretHoefer21,
  title={Mild assumptions for the derivation of Einstein’s effective viscosity formula},
  author={G{\'e}rard-Varet, David and H{\"o}fer, Richard M},
  journal={Communications in Partial Differential Equations},
  volume={46},
  number={4},
  pages={611--629},
  year={2021},
  publisher={Taylor \& Francis}
}

@article{Gerard-VaretHillairet19,
  title={Analysis of the viscosity of dilute suspensions beyond Einstein's formula},
  author={Gerard-Varet, David and Hillairet, Matthieu},
  journal={Arch Rational Mech Anal},
  vol={238},
  pages={1349-–1411},
  year={2020}
}

@article{HoferMecherbetSchubert22,
  title={Non-existence of mean-field models for particle orientations in suspensions},
  author={H{\"o}fer, Richard M and Mecherbet, Amina and Schubert, Richard},
 journal={J Nonlinear Sci},
volume={34},
number={3},
year={2024},
doi={https://doi.org/10.1007/s00332-023-09959-1} 
}

@article{Duerinckx23,
  title={Semi-dilute rheology of particle suspensions: derivation of Doi-type models},
  author={Duerinckx, Mitia},
  journal={arXiv preprint arXiv:2302.01466
},
  year={2023}
}

@article{DuerinckxGloria21,
  title={Corrector equations in fluid mechanics: Effective viscosity of colloidal suspensions},
  author={Duerinckx, Mitia and Gloria, Antoine},
  journal={Archive for Rational Mechanics and Analysis},
  volume={239},
  number={2},
  pages={1025--1060},
  year={2021},
  publisher={Springer}
}

@book{DuerinckxGloria20,
    AUTHOR = {Duerinckx, Mitia and Gloria, Antoine},
     TITLE = {On {E}instein's effective viscosity formula},
    SERIES = {Memoirs of the European Mathematical Society},
    VOLUME = {7},
 PUBLISHER = {EMS Press, Berlin},
      YEAR = {2023},
     PAGES = {viii+186},
  MRNUMBER = {4659164},
       DOI = {10.4171/mems/7},
       URL = {https://doi.org/10.4171/mems/7},
}

@article{Santambrogio15,
  title={Optimal transport for applied mathematicians},
  author={Santambrogio, Filippo},
  journal={Birk{\"a}user, NY},
  volume={55},
  number={58-63},
  pages={94},
  year={2015},
  publisher={Springer}
}

@book{DoiEdwards88,
  title={The theory of polymer dynamics},
  author={Doi, Masao and Edwards, Samuel Frederick},
  volume={73},
  year={1988},
  publisher={Oxford University Press}
}
 \end{refcontext}
\end{document}